\documentclass{article}
\usepackage{amssymb,amsmath,bm,geometry,graphics,graphicx,url,color}
\usepackage{amsfonts}
\usepackage{graphicx}
\usepackage{multirow}
\usepackage{amscd}
\usepackage{enumerate}
\usepackage{mathrsfs}
\usepackage{xypic}
\usepackage{stmaryrd}
\usepackage{fancybox}
\usepackage{exscale}
\usepackage{enumitem,array}%
\usepackage{graphicx}

\def\pdt2{\partial_t^2}
\def\pdx2{\partial_x^2}

\newcommand{\normmm}[1]{{\left\vert\kern-0.25ex\left\vert\kern-0.25ex\left\vert #1
    \right\vert\kern-0.25ex\right\vert\kern-0.25ex\right\vert}}

\newcommand{\abs}[1]{\left\vert#1\right\vert}

\def\RR{{\mathbb{R}}}

\newtheorem{theo}{Theorem}[section]
\newtheorem{lem}[theo]{Lemma}
\newtheorem{cor}[theo]{Corollary}
\newtheorem{rem}[theo]{Remark}
\newtheorem{defi}[theo]{Definition}

\def\no{\noindent}

\title{Volume-preserving exponential integrators}

\author{Bin Wang\,\footnote{School of Mathematical Sciences, Qufu Normal
University, Qufu  273165,  P.R.China; Mathematisches Institut,
University of T\"{u}bingen, Auf der Morgenstelle 10, 72076
T\"{u}bingen, Germany. E-mail:~{\tt wang@na.uni-tuebingen.de}} \and
Xinyuan Wu\thanks{School of Mathematical Sciences, Qufu Normal
University, Qufu  273165,  P.R.China; Department of Mathematics,
Nanjing University,  Nanjing 210093, P.R. China. E-mail:~{\tt
xywu@nju.edu.cn}} }

\begin{document}
\maketitle
\begin{abstract}
As is known that  various  dynamical systems including all
Hamiltonian systems preserve volume in phase space. This qualitative
geometrical property of the analytical solution should be respected
in the sense of Geometric Integration. This paper analyses the
volume-preserving property of
 exponential integrators in  different vector fields. We derive  a  necessary and sufficient condition
of volume preservation for exponential integrators, and with this
condition,   volume-preserving exponential integrators
 are analysed in detail for four kinds of vector fields.  It turns out
that symplectic exponential integrators can be   volume preserving
for a much larger class of vector fields than Hamiltonian systems.
 On the  basis  of the analysis,  novel
volume-preserving  exponential integrators  are derived for solving
highly oscillatory second-order systems and extended
Runge--Kutta--Nystr\"{o}m  (ERKN) integrators  of volume
preservation are presented for  separable partitioned  systems.
 Moreover,  the volume preservation of Runge--Kutta--Nystr\"{o}m
(RKN) methods is also discussed. Four illustrative numerical
experiments are carried out to demonstrate the  notable  superiority
of volume-preserving exponential integrators in comparison with
volume-preserving Runge-Kutta methods.
\medskip

\no{Keywords:}
 exponential integrators, volume preservation, geometric
 integrators, extended RKN integrators, highly oscillatory
 systems

\medskip
\no{MSC (2000):} 65L05, 65L06, 65L99,  34C60
\end{abstract}

\section{Introduction}\label{intro}
Geometric integrators (also called as structure-preserving
algorithms) have been an area of great interest and active research
in the recent  decades. The main advantage  of such methods for
solving
 ordinary differential equations (ODEs) is that they can  exactly preserve some qualitative geometrical
property of the analytical solution, such as the symplecticity,
energy preservation, and symmetry for  long-term integration.
Various geometric integrators have been designed and analysed
recently and   the reader is referred  to
\cite{Brugnano2014,Celledoni09,Cohen-2011,Hairer2010,hairer2000,Hairer16,Li_Wu(sci2016),McLachlan14,Sanz-Serna92,Wang2018-1}
for example on this topic.  For a good theoretical foundation of
geometric numerical integration for ODEs, we refer
 the reader to \cite{Feng-2010,hairer2006}.  Surveys of
structure-preserving algorithms for oscillatory  differential
equations are referred to \cite{wubook2018,wu2013-book}.

It is well known that volume preservation is an important property
shared by several dynamical systems. By the classical theorem due to
Liouville, it is clear  that  all Hamiltonian systems are also  of
volume preservation. Preservation of volume by a numerical method
has become  a desirable property and many methods have been proposed
and shown to be (or not to be) of volume preservation. We refer the
reader to
\cite{Chartier07,Feng95,He2015,Iserles07,McLachlan08,Quispel95,Zanna2013}
 and
references therein. From the researches of this topic, it follows
that all symplectic methods are of volume preservation for
Hamiltonian systems. However, this result does not hold  for the
system beyond Hamiltonian systems.  Feng and Shang have proved in
\cite {Feng95} that no Runge-Kutta (RK) method can be of volume
preservation even for linear divergence free vector fields. The
authors in \cite{Chartier07,Iserles07} showed that no B-series
method can be of volume preservation for all possible
divergence-free vector fields. Thus,  the design of efficient
volume-preserving (VP) methods  is still an open problem in  the
geometric numerical integration (see \cite{McLachlan98}). Recently,
various  VP  methods have been constructed and analysed for
different vector fields, such as splitting methods (see
\cite{McLachlan08,Zanna2013}), RK methods
 (see \cite{Webb2016}) and the methods  using generating  functions (see \cite{Quispel95,Zanna2014}).

On the other hand,  exponential integrators have been  proposed and
researched as an efficient approach to integrating ODEs/PDEs. The
reader is referred to
 \cite{Hochbruck2005,Hochbruck2010,Hochbruck2009,wang2017-JCM}
for example. However, it seems that  exponential integrators with
volume-preserving property for different vector fields have not been
researched yet  in the literature, which motives this paper.

In this paper, we   study the volume  preservation  of exponential
integrators when solving the following   first-order ODEs
\begin{equation}\label{IVP1}
y^{\prime}(t)=Ky(t)+g(y(t)):=f(y(t)),\quad
y(0)=y_{0}\in\mathbb{R}^{n},
\end{equation}
where $K$ is an $n \times n$   matrix which is assumed that
$\abs{e^{hK}}\neq-1$ for $0<h<1$, and $g:
\mathbb{R}^{n}\rightarrow\mathbb{R}^n$ is a differentiable nonlinear
function. In this paper, $\abs{\cdot}$ denotes the determinant.  The
function $f$ is assumed to be divergence free such that this system
is of volume preservation.  It is well known that the exact solution
of \eqref{IVP1} can be presented by the variation-of-constants
formula
  \begin{equation}
 y(t)= e^{tK}y_0+  t\int_{0}^1 e^{(1-\tau)tK}g(y(\tau t))d\tau.\\
\label{VOC}%
\end{equation}

 The  main  contributions of this paper are to derive the  volume-preserving condition  for
 exponential integrators   and  analyse  their  volume  preservations  for
different classes of vector fields which are larger   than
Hamiltonian systems.  Furthermore, based on the analysis,
volume-preserving  adapted exponential integrators are formulated
for highly oscillatory second-order systems and extended
Runge--Kutta--Nystr\"{o}m (ERKN) integrators  of volume preservation
are derived for  separable partitioned systems.  We also discuss the
volume preservation of  Runge--Kutta--Nystr\"{o}m (RKN) methods by
considering them as a special class of ERKN integrators. To our
knowledge, the results  presented in  this paper are novel that
rigorously study the
 volume-preserving properties of exponential integrators and ERKN/RKN integrators.

 We organise the remainder of this paper as follows. In
Section \ref{sec: methods}, the scheme of exponential integrators is
presented and   some useful results of these integrators are
summarised. Then   a necessary and sufficient condition for
exponential integrators to be of volume preservation is derived  in
Section \ref{sec: VPP}. On the basis of this condition, we study the
volume-preserving properties of exponential integrators for four
kinds of vector fields in Section \ref{sec: VPP for DVF}. Section
\ref{sec:applications} discusses the application  to various
problems including highly oscillatory second-order systems and
separable partitioned  systems, and shows the volume preservation of
adapted exponential integrators and ERKN/RKN integrators for these
different problems.  Four illustrative numerical experiments are
implemented in Section \ref{sec:examples} and the concluding remarks
are included in Section \ref{sec:conclusions}.

\section{Exponential integrators}\label{sec: methods}
In order to solve   \eqref{IVP1} effectively, one needs to
approximate the integral appearing in \eqref{VOC} by a quadrature
formula  with suitable nodes $c_1,c_2,\ldots,c_s$.  This leads to
the following exponential integrators proposed in
\cite{Hochbruck2010}, which have been successfully used for solving
 different kinds of ODEs/PDEs.
\begin{defi}
\label{scheme EI} (See \cite{Hochbruck2010}) An
$s$-stage exponential integrator for numerical integration of  \eqref{IVP1} is defined by%
\begin{equation}\label{EI}
\left\{\begin{array}[c]{ll} &k_i=e^{c_i
hK}y_n+h\sum\limits_{j=1}^{s}\bar{a}_{ij}(hK)g(k_{j}),\ \
i=1,2,\ldots,s,
\\
&y_{n+1}=e^{ hK}y_n+h\sum\limits_{i=1}^{s}\bar{b}_{i}(hK)g(k_{i}),
\end{array}\right.
\end{equation}
where   $h$ is  a  stepsize,  $c_i$ are real constants for $
i=1,\cdots,s$, $\bar{b}_{i}(hK)$ and $\bar{a}_{ij}(hK)$ are
matrix-valued functions of $hK$ for $i,j=1,\ldots,s$.
\end{defi}
The coefficients of  the   integrator  can be displayed in a Butcher
tableau  (omit     $(hK)$ for brevity):
\[%
\begin{tabular}
[c]{l}%
\\
\\[2mm]%
\begin{tabular}
[c]{c|c}%
$c$ & $\bar{A}$\\\hline &
$\raisebox{-1.3ex}[0pt]{$\bar{b}^{\intercal}$}$
\end{tabular}
$\ \quad=$ $\ $%
\end{tabular}%
\begin{tabular}
[c]{c|ccc}%
$c_{1}$ & $\bar{a}_{11}$ & $\ldots$ & $\bar{a}_{1s}$\\
$\vdots$ & $\vdots$ & $\ddots$ & $\vdots$\\
$c_{s}$ & $\bar{a}_{s1}$ & $\cdots$ & $\bar{a}_{ss}$\\\hline &
$\raisebox{-1.3ex}[0.5pt]{$\bar{b}_1$}$ &
\raisebox{-1.3ex}{$\cdots$} &
$\raisebox{-1.3ex}[0.5pt]{$\bar{b}_s$}$
\end{tabular}
\]
It is noted that when $K=\mathbf{0}$, this integrator reduces to a
classical $s$-stage RK   method represented by the Butcher tableau
\[%
\begin{tabular}
[c]{l}%
\\
\\[2mm]%
\begin{tabular}
[c]{c|c}%
$c$ & $A$\\\hline & $\raisebox{-1.3ex}[0pt]{$b^{\intercal}$}$
\end{tabular}
$\ \quad=$ $\ $%
\end{tabular}%
\begin{tabular}
[c]{c|ccc}%
$c_{1}$ & $a_{11}$ & $\ldots$ & $a_{1s}$\\
$\vdots$ & $\vdots$ & $\ddots$ & $\vdots$\\
$c_{s}$ & $a_{s1}$ & $\cdots$ & $a_{ss}$\\\hline &
$\raisebox{-1.3ex}[0.5pt]{$b_1$}$ & \raisebox{-1.3ex}{$\cdots$} &
$\raisebox{-1.3ex}[0.5pt]{$b_s$}$
\end{tabular}
\]

In this paper, a kind of special  and important  exponential
integrators will be considered and analysed, which was proposed in
\cite {Mei2017}.
\begin{defi}
\label{scheme EI spe} (See \cite {Mei2017}) Define a kind of
$s$-stage exponential integrators   by%
\begin{equation}
\begin{aligned} \bar{a}_{ij}(h K)=a_{ij}e^{(c_i-c_j) h K},\ \ \bar{b}_{i}(h K)=b_{i}e^{(1-c_i)h
K},\ \ i,j=1,\ldots,s,
\end{aligned}
\label{rev co}%
\end{equation}
where
\begin{equation}c=(c_1,\ldots,c_s)^{\intercal},\ \
b=(b_1,\ldots,b_s)^{\intercal},\ \
A=(a_{ij})_{s\times s}\label{RK co}%
\end{equation} are   the coefficients of an $s$-stage  RK
method.
\end{defi}

With regard to this kind of exponential integrators, two useful
properties are shown in \cite {Mei2017} and we  summarise them as
follows.

\begin{theo}
\label{order thm} (See \cite {Mei2017}) If  an  RK method with the
coefficients \eqref{RK co} is of order $p$, then the exponential
integrator given by \eqref{rev co} is also of order $p$.
\end{theo}

 \begin{theo}
\label{sympl thm}(See \cite {Mei2017}) The
 exponential integrator   defined by \eqref{rev co} is
 symplectic if the corresponding   RK method  \eqref{RK co}  is symplectic.
\end{theo}

In this paper, we  supplement  an additional requirement to $b$ of
\eqref{RK co}  and define the following special symplectic
exponential integrators (SSEI).
\begin{defi}
\label{scheme EI sspe}  An $s$-stage exponential integrator
\eqref{rev co} is called as special symplectic exponential
integrator (SSEI) if the RK method  \eqref{RK co}  is symplectic and
$b_j\neq0$ for all $j=1,\ldots,s.$
\end{defi}

\begin{rem}\label{rem00}
We note that a kind of special symplectic  RK (SSRK) methods was
considered in \cite{Webb2016} and our SSEI integrators  reduce to
the SSRK methods when $K=\mathbf{0}$.
\end{rem}

\section{VP condition of exponential integrators}\label{sec: VPP}

For each stepsize $h$, denote the $s$-stage exponential integrator
  \eqref{EI} by a map $\varphi_h: \mathbb{R}^{n}\rightarrow \mathbb{R}^{n}$,
which is
\begin{equation}\label{EI map}
\left\{\begin{array}[c]{ll} &\varphi_h(y)=e^{
hK}y+h\sum\limits_{i=1}^{s}\bar{b}_{i}(hK)g(k_{i}(y)),\\
&k_i(y)=e^{c_i
hK}y+h\sum\limits_{j=1}^{s}\bar{a}_{ij}(hK)g(k_{j}(y)),\ \
i=1,2,\ldots,s.
\end{array}\right.
\end{equation}
It is noted that $k_i(y)$ for $i=1,2,\ldots$ are identical to
$\varphi_{c_ih}(y)$ for $i=1,2,\ldots$ in \eqref{EI map}.

We firstly derive the Jacobian matrix and its determinant for the
general exponential integrator \eqref{EI map}.

\begin{lem}\label{lem jacob}
The Jacobian matrix of the exponential integrator \eqref{EI map} can
be expressed as
\begin{equation*}
\varphi'_h(y)=e^{ hK}+h \bar{b}^{\intercal} F(I_s\otimes I-h\bar{A}
F)^{-1} e^{chK},
\end{equation*}
where $F=\textmd{diag}(g'(k_1),\ldots,g'(k_s))$,  $I_s$ and $I$ are
the  $s\times s$ and $n\times n$ identity matrices, respectively,
and $e^{chK}= (e^{c_1hK},\ldots,e^{c_shK})^{\intercal}.$ Its
determinant reads
\begin{equation}\label{determinant}
\abs{\varphi'_h(y)}= \frac{\abs{e^{ hK}}\abs{I_s\otimes
I-h(\bar{A}-e^{(c- \textbf{1} )hK}\bar{b}^{\intercal})
F}}{\abs{I_s\otimes I-h\bar{A} F}},
\end{equation}
where $\textbf{1} $  is an $s\times1$ vector of  units  and
$e^{(c-\textbf{1})hK}=
(e^{(c_1-1)hK},\ldots,e^{(c_s-1)hK})^{\intercal}.$ Here we make use
of the Kronecker product $\otimes$ throughout this paper.
\end{lem}

\begin{proof}
According to   the first formula of \eqref{EI map}, we obtain
\begin{equation}\label{compu dire}
\varphi'_h(y)=e^{ hK}+h\sum\limits_{i=1}^{s}\bar{b}_{i} g'(k_{i}(y))
k'_{i}(y)=e^{ hK}+h\bar{b}^{\intercal} F (k'_{1},
\ldots,k'_{s})^{\intercal}.
\end{equation}
 Likewise, it follows from  $k_i(y)$  in \eqref{EI map} that
\begin{equation*}
 \left(
   \begin{array}{cccc}
     I-h\bar{a}_{11} g'(k_{1}) & -h\bar{a}_{12} g'(k_{2}) &\cdots & -h\bar{a}_{1s} g'(k_{s}) \\
     -h\bar{a}_{21} g'(k_{1}) & I-h\bar{a}_{22} g'(k_{2}) & \cdots & -h\bar{a}_{2s} g'(k_{s}) \\
     \vdots & \vdots & \vdots & \vdots \\
     -h\bar{a}_{s1} g'(k_{1})  & -h\bar{a}_{s2} g'(k_{2})  & \cdots& I-h\bar{a}_{ss} g'(k_{s}) \\
   \end{array}
 \right)\left(
          \begin{array}{c}
            k'_{1} \\
            k'_{2} \\
            \vdots \\
            k'_{s} \\
          \end{array}
        \right)=e^{chK},
\end{equation*}
which can be rewritten as
\begin{equation}\label{ki}
(I_s\otimes I-h\bar{A} F)(k'_{1},
\ldots,k'_{s})^{\intercal}=e^{chK}.
\end{equation}
Substituting \eqref{ki} into \eqref{compu dire} yields the first
statement of this lemma.

For the second statement, we will use the following block
determinant identity (see \cite{Webb2016,hairer2006}):
\begin{equation*}
\abs{U}\abs{X-WU^{-1}V}=\abs{\begin{array}{cc}
                          U & V \\
                          W & X
                        \end{array}}=\abs{X}\abs{U-VX^{-1}W},
\end{equation*}
which is yielded from  Gaussian elimination.  By letting
$$X=e^{ hK},\ W=-h \bar{b}^{\intercal} F,\ U=I_s\otimes I-h\bar{A}
F,\ V= e^{chK},$$ it is obtained that
\begin{equation*}
 \begin{array}[c]{ll}  \abs{I_s\otimes
I-h\bar{A} F} \abs{\varphi'_h(y)}&=\abs{e^{ hK}}\abs{I_s\otimes
I-h\bar{A} F+he^{chK}e^{-hK} \bar{b}^{\intercal} F}\\
&=\abs{e^{ hK}}\abs{I_s\otimes I-h(\bar{A} -e^{(c- \textbf{1} )hK}
\bar{b}^{\intercal}) F},
\end{array}
\end{equation*}
which leads to the second result \eqref{determinant}. \hfill
\end{proof}

For the SSEI methods of Definition \ref{scheme EI sspe},  a
necessary and sufficient condition  for these integrators to be of
volume preservation is shown by the following lemma.

\begin{lem}\label{lem VP condition}
An $s$-stage SSEI method defined in Definition \ref{scheme EI sspe}
is of volume preservation if and only if the following VP condition
is satisfied
\begin{equation}\label{VP cond}
\abs{I_s\otimes I-h(A\otimes I .* E(hK)) F}=
\abs{e^{hK}}\abs{I_s\otimes I+h(A^{\intercal}\otimes I .* E(hK)) F}
,
\end{equation}
where $ E(hK)$ is a block matrix defined by
\begin{equation}\label{ehk} E(hK)=(E_{i,j}(hK))_{s\times s}=\left(
                                          \begin{array}{cccc}
                                            I & e^{(c_1-c_2)hK} & \cdots & e^{(c_1-c_s)hK} \\
                                            e^{(c_2-c_1)hK} & I & \cdots & e^{(c_2-c_s)hK} \\
                                            \vdots & \vdots & \ddots & \vdots \\
                                            e^{(c_s-c_1)hK} & e^{(c_s-c_2)hK} & \cdots & I \\
                                          \end{array}
                                        \right),
\end{equation}   and $.*$ denotes the element-wise multiplication of two matrices.
\end{lem}

\begin{proof}In terms of the choice \eqref{rev co} of the coefficients,
 it is computed that
\begin{equation*}
 \begin{array}[c]{ll}  &\bar{A}-e^{(c- \textbf{1}
 )hK}\bar{b}^{\intercal}\\=&(A\otimes I).*E(hK)-(e^{(c_1-1)hK},\ldots,e^{(c_s-1)hK})^{\intercal}
 (b_{1}e^{(1-c_1)h
K},\ldots,b_{s}e^{(1-c_s)h K})\\
=&(A\otimes I).*E(hK)-(\textbf{1} b^{\intercal}\otimes
I).*E(hK)\\
=&(A -\textbf{1} b^{\intercal})\otimes I.*E(hK).
\end{array}
\end{equation*}
Thus, we obtain
\begin{equation}\label{determinant-1}
\abs{\varphi'_h(y)}= \frac{\abs{e^{ hK}}\abs{I_s\otimes I- h(A
-\textbf{1} b^{\intercal})\otimes I.*E(hK) F}}{\abs{I_s\otimes
I-h\bar{A} F}}.
\end{equation}
Moreover, with careful calculations, it can be verified that for
$B=\textmd{diag}(b_1,\ldots,b_s)$, the following result holds
\begin{equation}\label{con1-1}
 \begin{array}[c]{ll}  &\abs{I_s\otimes I-
h(A -\textbf{1} b^{\intercal})\otimes I.*E(hK) F}\\
=& \abs{I_s\otimes I- h(B\otimes I)(A -\textbf{1}
b^{\intercal})\otimes I.*E(hK) F(B^{-1}\otimes I)}\\
=&\abs{I_s\otimes I- h(B\otimes I)(A -\textbf{1}
b^{\intercal})\otimes I.*E(hK) (B^{-1}\otimes I)F}\\
=&\abs{I_s\otimes I- hB(A -\textbf{1} b^{\intercal})B^{-1}\otimes
I.*E(hK)  F}.
\end{array}
\end{equation}
Since  the RK method is symplectic, one has that
$BA+A^{\intercal}B-bb^{\intercal}=0$ (see \cite {hairer2006}), which
leads to $B(A -\textbf{1} b^{\intercal})B^{-1}=-A^{\intercal}.$
Therefore, the result \eqref{con1-1} can be simplified as
\begin{equation*}
 \begin{array}[c]{ll}  \abs{I_s\otimes I-
h(A -\textbf{1} b^{\intercal})\otimes I.*E(hK) F} =\abs{I_s\otimes
I+ h(A^{\intercal}\otimes I.*E(hK))  F}.
\end{array}
\end{equation*}
The proof is complete by considering \eqref{determinant-1}. \hfill
\end{proof}

\begin{rem}
It is noted that when $K=\mathbf{0}$, this VP condition \eqref{VP
cond} reduces  to the condition of  RK   methods presented in
\cite{Webb2016}.  Consequently, the condition \eqref{VP cond} can be
regarded as a generalisation of that of RK   methods.
\end{rem}

\section{VP results for different vector fields}\label{sec: VPP for DVF}
In this section, we study the volume-preserving properties of
exponential integrators for four kinds of vector fields, which are
defined as follows.
\begin{defi}
\label{defi H} (See \cite{Webb2016}) Define the following four
classes of vector fields on Euclidean space using vector fields
$f(y)$
\begin{equation*} \begin{array}[c]{ll}
\mathcal{H}=&\{f|\ \textmd{there exists}\ P\ \textmd{such
that for all}\ y, Pf'(y)P^{-1}=-f'(y)^{\intercal}\},\\
\mathcal{S}=&\{f|\ \textmd{there exists}\ P\ \textmd{such
that for all}\ y, Pf'(y)P^{-1}=-f'(y)\},\\
\mathcal{F}^{(\infty)}=&\{f(y_1,y_2)=(u(y_1),v(y_1,y_2))^{\intercal}\
\textmd{where}\ u\in \mathcal{H} \cup \mathcal{F}^{(\infty)}|\
\textmd{there exists}\ P\\
&\ \   \textmd{such
that for all}\ y_1,y_2, P\partial_{y_2}v(y_1,y_2)P^{-1}=-\partial_{y_2}v(y_1,y_2)^{\intercal}\},\\
\mathcal{F}^{(2)}=&\{f(y_1,y_2)=(u(y_1),v(y_1,y_2))^{\intercal}\
\textmd{where}\ u\in \mathcal{H} \cup \mathcal{S} \cup
\mathcal{F}^{(2)}|\
\textmd{there exists}\ P\\
&\ \  \textmd{such that for all}\ y_1,y_2,\ \textmd{either}\
P\partial_{y_2}v(y_1,y_2)P^{-1}=-\partial_{y_2}v(y_1,y_2)^{\intercal},\\
&\ \ \textmd{or}\
P\partial_{y_2}v(y_1,y_2)P^{-1}=-\partial_{y_2}v(y_1,y_2)\}.
\end{array}\end{equation*}
\end{defi}

\begin{rem}\label{rem1}  It has been proved in \cite{Webb2016}  that all these sets are  equal to divergence free
vector fields.  The relationships of these vector fields are also
given in \cite{Webb2016} by $\mathcal{H}\subset
\mathcal{F}^{(\infty)}\subset \mathcal{F}^{(2)}$ and
$\mathcal{S}\subset \mathcal{F}^{(\infty)}\subset
\mathcal{F}^{(2)}.$ According to Lemma 3.2 of \cite{Webb2016}, we
know that the set $\mathcal{H}$ contains all Hamiltonian systems.
Denote by $H$ the set of  Hamiltonian systems. See Figure \ref{p000}
for the venn diagram illusting the relationships. This figure
clearly shows that the  sets $\mathcal{H},\ \mathcal{F}^{(\infty)}$
and $\mathcal{F}^{(2)}$  are larger classes of vector fields than
Hamiltonian systems.
 \begin{figure}[ptb]
\centering
\includegraphics[width=8cm,height=3.5cm]{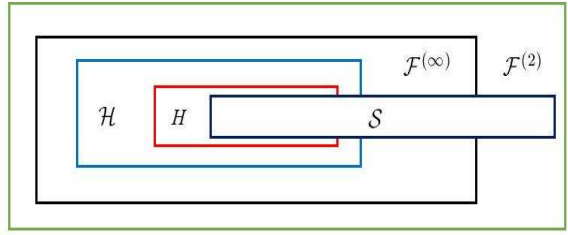}
\caption{Venn diagram illusting the relationships.} \label{p000}
\end{figure}
\end{rem}

\subsection{Vector field $\mathcal{H}$}\label{subsec: VPP for H}

\begin{theo} \label{thm vp for H} All SSEI methods for solving \eqref{IVP1} are of volume preservation
for vector fields $f$ and $g$ in $\mathcal{H}$ with the same $P$.
\end{theo}
\begin{proof}
Let  $P$ be such that for all $y$, $Pf'(y)P^{-1}=-f'(y)^{\intercal}$
and $Pg'(y)P^{-1}=-g'(y)^{\intercal}$. According to this condition
and the expression $f(y)=Ky+g(y)$, one has that
$PKP^{-1}=-K^{\intercal}$. Thus it is easily obtained that
$Pe^{hK}P^{-1}=e^{-hK^{\intercal}}$. In the light of this result,
one gets
\begin{equation*}
 \begin{array}[c]{ll}
 &\abs{Pe^{hK}P^{-1}}=\abs{e^{hK}}=\abs{e^{-hK^{\intercal}}}=\abs{e^{-hK}}=\abs{(e^{hK})^{-1}}=\frac{1}{\abs{e^{hK}}},
\end{array}
\end{equation*}
which yields $\abs{e^{hK}}=1$ (it is assumed that
$\abs{e^{hK}}\neq-1$ in the introduction of this paper). We then
compute the  left-hand  side of \eqref{VP cond} as follows
\begin{equation*}
 \begin{array}[c]{ll}  &\abs{I_s\otimes I-h(A\otimes I .* E(hK)) F}\\
=&\abs{(I_s\otimes P)(I_s\otimes P^{-1})-h(I_s\otimes P)(A\otimes I .* E(hK)) (I_s\otimes P^{-1})(I_s\otimes P)F(I_s\otimes P^{-1})}\\
=&\abs{I_s\otimes I+h (I_s\otimes P)(A\otimes I .* E(hK)) (I_s\otimes P^{-1})  F^{\intercal}}\\
=&\abs{I_s\otimes I+h  (A\otimes I .* E(-hK^{\intercal}))
F^{\intercal}}\\
=&\abs{I_s\otimes I+h  F(A^{\intercal}\otimes I .* E(-hK^{\intercal})^{\intercal}) }\ \textmd{(transpose)}\\
=&\abs{I_s\otimes I+h  (A^{\intercal}\otimes I .*
E(-hK^{\intercal})^{\intercal}) F}\ \textmd{(Sylvester's law)}.
\end{array}
\end{equation*}
It follows from the definition \eqref{ehk} that
\begin{equation}\label{E-hk}
 \begin{array}[c]{ll}  &E(-hK^{\intercal})^{\intercal}=(E_{i,j}(-hK^{\intercal}))_{s\times
s}^{\intercal}=(E_{j,i}(-hK))_{s\times s}=(E_{i,j}(hK))_{s\times
s}=E(hK).
\end{array}
\end{equation}
Therefore, one arrives at
\begin{equation*}\abs{I_s\otimes I-h(A\otimes I .* E(hK)) F} =\abs{I_s\otimes I+h  (A^{\intercal}\otimes I .*
E(hK)) F},
\end{equation*}
which shows  that all  SSEI methods are of volume preservation for
vector fields in $\mathcal{H}$ by considering Lemma \ref{lem VP
condition}.\hfill
\end{proof}

\subsection{Vector
field $\mathcal{S}$}\label{subsec: VPP for S}
\begin{theo} \label{thm vp for S} All one-stage SSEI methods
and all two-stage SSEI methods with $c_1=c_2$ (and any composition
of such methods)  are of volume preservation for vector fields $f$
and $g$ in $\mathcal{S}$ with the same $P$.
\end{theo}
\begin{proof} Let  $P$ be such that for all $y$,
$Pf'(y)P^{-1}=-f'(y)$  and $Pg'(y)P^{-1}=-g'(y)$. Similarly to the
proof of last theorem, we obtain that $PKP^{-1}=-K$. Thus it is true
that $Pe^{hK}P^{-1}=e^{-hK}$ and $\abs{e^{hK}}=1$.

For the one-stage SSEI, according to Lemma \ref{lem VP condition},
it is of volume preservation if and only if
$$
\abs{ I-ha_{11} g'(k_{1}) }= \abs{
        I+ha_{11} g'(k_{1})},
$$
which  can be  verified  by considering
\begin{equation*}
  \abs{ I-ha_{11} g'(k_{1}) }
 =\abs{  PP^{-1}-ha_{11} Pg'(k_{1})P^{-1}}
 =\abs{  I+ha_{11}  g'(k_{1}) }.
\end{equation*}

For a two-stage SSEI method,  according to Lemma \ref{lem VP
condition} again, this two-stage SSEI method is of volume
preservation if and only if
\begin{equation*} \begin{array}[c]{ll}
&\abs{ \begin{array}{cc}
        I-ha_{11} g'(k_{1}) & -ha_{12}e^{(c_1-c_2)hK} g'(k_{2}) \\
        -ha_{21}e^{(c_2-c_1)hK} g'(k_{1}) & I-ha_{22}  g'(k_{2})
      \end{array}
}\\
= &\abs{\begin{array}{cc}
        I+ha_{11} g'(k_{1}) & ha_{21}e^{(c_1-c_2)hK}  g'(k_{2}) \\
        ha_{12}e^{(c_2-c_1)hK} g'(k_{1}) & I+ha_{22} g'(k_{2})
      \end{array}}  ,
\end{array}\end{equation*}
which gives the condition
\begin{equation}\label{resul 2}
 \begin{array}[c]{ll}  & |I-ha_{11} g'(k_{1})-ha_{22}  g'(k_{2})+h^2a_{11}a_{22} g'(k_{1}) g'(k_{2})\\
 &-h^2a_{12}a_{21}e^{(c_1-c_2)hK} g'(k_{2})e^{(c_2-c_1)hK}
 g'(k_{1})|\\
 =&|I+ha_{11} g'(k_{1})+ha_{22}  g'(k_{2})+h^2a_{11}a_{22} g'(k_{1}) g'(k_{2})\\
 &-h^2a_{12}a_{21}e^{(c_1-c_2)hK} g'(k_{2})e^{(c_2-c_1)hK}
 g'(k_{1})|.
\end{array}
\end{equation}
It  also can be verified  that
\begin{equation*}
 \begin{array}[c]{ll}  & \textmd{the left hand side of}\ \eqref{resul 2}\\
 =& \mid PP^{-1}-ha_{11} Pg'(k_{1})P^{-1}-ha_{22} P g'(k_{2})P^{-1}+h^2a_{11}a_{22} Pg'(k_{1}) g'(k_{2})P^{-1}\\
 &\ \ -h^2a_{12}a_{21}Pe^{(c_1-c_2)hK} g'(k_{2})e^{(c_2-c_1)hK}
 g'(k_{1})P^{-1}\mid\\
  =& \mid I+ha_{11} g'(k_{1})+ha_{22}   g'(k_{2}) +h^2a_{11}a_{22} Pg'(k_{1})P^{-1} P g'(k_{2})P^{-1}\\
 &\ \ -h^2a_{12}a_{21}Pe^{(c_1-c_2)hK} P^{-1} Pg'(k_{2})P^{-1} Pe^{(c_2-c_1)hK}P^{-1} P
 g'(k_{1})P^{-1}\mid\\
   =& \mid I+ha_{11} g'(k_{1})+ha_{22}   g'(k_{2}) +h^2a_{11}a_{22} g'(k_{1}) g'(k_{2})\\
 &\ \ -h^2a_{12}a_{21}e^{(c_2-c_1)hK} g'(k_{2})   e^{(c_1-c_2)hK}
 g'(k_{1})\mid.
\end{array}
\end{equation*}
Under the assumption  that $c_1=c_2$, this result becomes
\begin{equation*}
 \begin{array}[c]{ll}  &   \mid I+ha_{11} g'(k_{1})+ha_{22}   g'(k_{2}) +h^2a_{11}a_{22} g'(k_{1}) g'(k_{2})\\
 &\ \ -h^2a_{12}a_{21}e^{(c_1-c_2)hK}  g'(k_{2}) e^{(c_2-c_1)hK}
 g'(k_{1})\mid.
\end{array}
\end{equation*}
Thus   \eqref{resul 2} is obtained immediately and then all
two-stage SSEI methods with $c_1=c_2$ are of volume
preservation.\hfill
\end{proof}

 \begin{rem} It is noted that for  the vector field
$\mathcal{S}$ and two-stage SSEI methods, the condition
\begin{equation}\label{cc}c_1=c_2\end{equation}   is supplemented in order to make the
following condition be true
\begin{equation*}
e^{(c_2-c_1)hK} g'(k_{2})   e^{(c_1-c_2)hK}
 g'(k_{1})=e^{(c_1-c_2)hK}  g'(k_{2}) e^{(c_2-c_1)hK}
 g'(k_{1}).
\end{equation*}
We remark that   condition \eqref{cc} is not necessary for    some
  special matrix $K$ such as $K=0$ or some special function $g$ such as scalar functions.
  The same situation happens in the analysis of Subsection \ref{subsec: VPP for F2}.
 \end{rem}

\subsection{Vector
field $\mathcal{F}^{(\infty)}$}\label{subsec: VPP for FINFI} For the
vector field $\mathcal{F}^{(\infty)}$, if the function
$f(y):=Ky+g(y)$ has the  pattern  $(u(y_1),v(y_1,y_2))^{\intercal}$,
this means that $K$ and $g$ can be expressed in blocks as
 \begin{equation}\label{KG}
\begin{array}
[c]{ll}&K=\left(
      \begin{array}{cc}
        K_{11} & 0 \\
       0 & K_{22}  \\
      \end{array}
    \right),\ \ g(y)=\left(
                       \begin{array}{c}
                         g_1(y_1) \\
                         g_2(y_1,y_2) \\
                       \end{array}
                     \right).
\end{array}
\end{equation}
Then the following relation is true
 \begin{equation}\label{KG-2}
  u(y_1)=K_{11}y_1+g_1(y_1),\ \
                     v(y_1,y_2)=K_{22}y_2+g_2(y_1,y_2).
\end{equation}

\begin{theo} \label{thm vp for FIN} Consider an $s$-stage SSEI method for solving $y_1'= u(y_1)$ that is
of volume preservation for the vector field $u(y_1):
\mathbb{R}^m\rightarrow \mathbb{R}^m$.   Let $v(y_1,y_2):
\mathbb{R}^{m+n}\rightarrow \mathbb{R}^{m+n}$  and assume that there
exists an invertible matrix $P$ such that for all $y_1,y_2,$
$$P\partial_{y_2}v(y_1,y_2)P^{-1}=-\partial_{y_2}v(y_1,y_2)^{\intercal},\ \ P\partial_{y_2}g_2(y_1,y_2)P^{-1}=-\partial_{y_2}g_2(y_1,y_2)^{\intercal}.$$
Then the SSEI method is of volume preservation for vector fields
$f(y_1,y_2)=(u(y_1),v(y_1,y_2))^{\intercal}$ in
$\mathcal{F}^{(\infty)}.$
\end{theo}

\begin{proof}
From the property of $v$, we have
$PK_{22}P^{-1}=-K_{22}^{\intercal}$ and $\abs{e^{hK_{22}}}=1$. Thus
$\abs{e^{hK}}=\abs{e^{hK_{11}}}\abs{e^{hK_{22}}}=\abs{e^{hK_{11}}}.$
The Jacobian matrix of $g(y)$ is block triangular as follows
$$g'(y_1,y_2)=\left(
                \begin{array}{cc}
                  \partial _{y_1}g_1(y_1) & 0 \\
                  * & \partial _{y_2}g_2(y_1,y_2) \\
                \end{array}
              \right).
$$
In what follows, we prove the condition \eqref{VP cond}. Using the
block  transformation, we can bring the left-hand side of \eqref{VP
cond} to the block form
\begin{equation*}
 \begin{array}[c]{ll}  &\abs{I_s\otimes I-h(A\otimes I .* E(hK))
 F}=
 \left(
   \begin{array}{cc}
     \Phi_1 & 0 \\
    * & \Phi_2 \\
   \end{array}
 \right),
\end{array}
\end{equation*}
where
\begin{equation*}
 \begin{array}[c]{ll}  &\Phi_1=
 \left(
 \begin{array}{ccc}
I-h\bar{a}_{11}(hK_{11}) \partial _{y_1}g_1(k_1)& \cdots & -h\bar{a}_{1s}(hK_{11})\partial _{y_1}g_1(k_s) \\
 \vdots & \ddots & \vdots  \\
-h\bar{a}_{s1}(hK_{11})\partial _{y_1}g_1(k_1) & \cdots & I-h\bar{a}_{ss}(hK_{11})\partial _{y_1}g_1(k_s) \\
\end{array} \right),\\
&\Phi_2=
 \left(
 \begin{array}{ccc}
I-h\bar{a}_{11}(hK_{22})\partial _{y_2}g_2(k_1) & \cdots & -h\bar{a}_{1s}(hK_{22})\partial _{y_2}g_2(k_s)  \\
 \vdots & \ddots & \vdots  \\
-h\bar{a}_{s1}(hK_{22}) \partial _{y_2}g_2(k_1)& \cdots & I-h\bar{a}_{ss}(hK_{22})\partial _{y_2}g_2(k_s) \\
\end{array} \right).
\end{array}
\end{equation*}
Let $F_1=\textmd{diag}(\partial _{y_1}g_1(k_1),\ldots,\partial
_{y_1}g_1(k_s))$ and $F_2=\textmd{diag}(\partial
_{y_2}g_2(k_1),\ldots,\partial _{y_2}g_2(k_s)).$ The above result
can be simplified as
\begin{equation*}
 \begin{array}[c]{ll}  &\abs{I_s\otimes I-h(A\otimes I .* E(hK))
 F}\\
 =& \abs{I_s\otimes I-h(A\otimes I .* E(hK_{11})) F_1}\abs{I_s\otimes
I-h(A\otimes I .* E(hK_{22})) F_2} .
\end{array}
\end{equation*}
 Since the SSEI method  is of volume preservation for the vector
field $u(y_1)$, the following condition is true
$$\abs{I_s\otimes I-h(A\otimes I .*
E(hK_{11})) F_1}=\abs{e^{hK_{11}}}\abs{I_s\otimes
I+h(A^{\intercal}\otimes I .* E(hK_{11})) F_1}.$$ On the other hand,
we compute
\begin{equation*}
 \begin{array}[c]{ll}  &\abs{I_s\otimes I-h(A\otimes I .* E(hK_{22})) F_2}\\
=&\abs{(I_s\otimes P)(I_s\otimes P^{-1})-h(I_s\otimes P)(A\otimes I .* E(hK_{22})) (I_s\otimes P^{-1})(I_s\otimes P)F(I_s\otimes P^{-1})}\\
=&\abs{I_s\otimes I+h (I_s\otimes P)(A\otimes I .* E(hK_{22})) (I_s\otimes P^{-1})  F_2^{\intercal}}\\
=&\abs{I_s\otimes I+h  (A\otimes I .* E(-hK_{22}^{\intercal}))
F_2^{\intercal}}\\
=&\abs{I_s\otimes I+h  F_2(A^{\intercal}\otimes I .* E(-hK_{22}^{\intercal})^{\intercal}) }\ \textmd{(transpose)}\\
=&\abs{I_s\otimes I+h  (A^{\intercal}\otimes I .*
E(-hK_{22}^{\intercal})^{\intercal}) F_2}\ \textmd{(Sylvester's
law)}\\
=&\abs{I_s\otimes I+h  (A^{\intercal}\otimes I .* E(hK_{22})) F_2}\
\textmd{(property \eqref{E-hk})}.
\end{array}
\end{equation*}
Therefore,   the VP condition \eqref{VP cond} holds and   the SSEI
method is of volume preservation for vector fields in
$\mathcal{F}^{(\infty)}$.
 \hfill
\end{proof}

\subsection{Vector
field $\mathcal{F}^{(2)}$}\label{subsec: VPP for F2} Suppose that
the function $f(y)$ of \eqref{IVP1} falls into $\mathcal{F}^{(2)}$.
Under this situation,  \eqref{KG} and \eqref{KG-2} are still true.
We obtain the following result about the VP property of SSEI
methods.

\begin{theo} \label{thm vp for F2} Consider  a one-stage or two-stage SSEI   with $c_1=c_2$ (or a composition of such method)
 that is of volume preservation for the vector field   $u(y_1):
\mathbb{R}^m\rightarrow \mathbb{R}^m$. Letting  $v(y_1,y_2):
\mathbb{R}^{m+n}\rightarrow \mathbb{R}^{m+n}$, we assume that
  there exists an invertible matrix $P$ such that for all
$y_1,y_2,$
$$P\partial_{y_2}v(y_1,y_2)P^{-1}=-\partial_{y_2}v(y_1,y_2),\ \ \ P\partial_{y_2}g_2(y_1,y_2)P^{-1}=-\partial_{y_2}g_2(y_1,y_2).$$
Then the SSEI method is of volume preservation for  the vector
fields $f(y_1,y_2)=(u(y_1),v(y_1,y_2))^{\intercal}$   in
$\mathcal{F}^{(2)}$.
\end{theo}

\begin{proof}
 From the conditions of this theorem, it follows that
$PK_{22}P^{-1}=-K_{22}$ and $\abs{e^{hK_{22}}}=1$.

For the one-stage SSEI, the condition for volume preservation is
\begin{equation*}
\abs{ I-ha_{11} g'(k_{1}) }= \abs{
        I+ha_{11} g'(k_{1})},
\end{equation*}
which  can be rewritten as
\begin{equation}\label{ff2}
  \abs{ I-ha_{11} \partial_{y_1}g_1 }   \abs{ I-ha_{11} \partial_{y_2}g_2 }
 =  \abs{ I+ha_{11} \partial_{y_1}g_1 }   \abs{ I+ha_{11} \partial_{y_2}g_2}.
\end{equation}
 Since the   method  is of volume preservation for the vector
field $u(y_1)$, we have $$\abs{ I-ha_{11} \partial_{y_1}g_1 } =\abs{
I+ha_{11} \partial_{y_1}g_1 }.$$ On the other hand,
$$ \abs{ I-ha_{11} \partial _{y_2}g_2 }= \abs{ PP^{-1}-ha_{11}P \partial _{y_2}g_2 P^{-1}}
= \abs{ I+ha_{11}  \partial _{y_2}g_2  }.$$ Thus \eqref{ff2} is
proved.

For the two-stage SSEI,  it is of volume preservation  if and only
if \eqref{resul 2} is true. According to the special result of $g'$,
one has
\begin{equation*}
 \begin{array}[c]{ll}  & \textmd{the left hand side of}\ \eqref{resul 2}\\
 =&|I-ha_{11}  \partial_{y_1}g_1(k_{1})-ha_{22}  \partial_{y_1}g_1(k_{2})+h^2a_{11}a_{22} \partial_{y_1}g_1(k_{1}) \partial_{y_1}g_1(k_{2})\\
 &\ \ -h^2a_{12}a_{21}
 \partial_{y_1}g_1(k_{2})
\partial_{y_1}g_1(k_{1})|\\
&|I-ha_{11}  \partial_{y_2}g_2(k_{1})-ha_{22}
 \partial_{y_2}g_2(k_{2})+h^2a_{11}a_{22}  \partial_{y_2}g_2(k_{1})  \partial_{y_2}g_2(k_{2})\\
 &\ \ -h^2a_{12}a_{21}
  \partial_{y_2}g_2(k_{2})
 \partial_{y_2}g_2(k_{1})|\\
  =&|I+ha_{11}  \partial_{y_1}g_1(k_{1})+ha_{22}  \partial_{y_1}g_1(k_{2})+h^2a_{11}a_{22} \partial_{y_1}g_1(k_{1}) \partial_{y_1}g_1(k_{2})\\
  &\ \ -h^2a_{12}a_{21}
 \partial_{y_1}g_1(k_{2})
\partial_{y_1}g_1(k_{1})|\\
&|I-ha_{11}  \partial_{y_2}g_2(k_{1})-ha_{22}
 \partial_{y_2}g_2(k_{2})+h^2a_{11}a_{22}  \partial_{y_2}g_2(k_{1})  \partial_{y_2}g_2(k_{2})\\
 &\ \ -h^2a_{12}a_{21}
  \partial_{y_2}g_2(k_{2})
 \partial_{y_2}g_2(k_{1})|.
\end{array}
\end{equation*}
It then can be verified that
\begin{equation*}
 \begin{array}[c]{ll}  & |I-ha_{11}  \partial_{y_2}g_2(k_{1})-ha_{22}
 \partial_{y_2}g_2(k_{2})+h^2a_{11}a_{22}  \partial_{y_2}g_2(k_{1})  \partial_{y_2}g_2(k_{2})\\
 &\ \ -h^2a_{12}a_{21}
  \partial_{y_2}g_2(k_{2})
 \partial_{y_2}g_2(k_{1})|\\
 =& | PP^{-1}-ha_{11} P \partial_{y_2}g_2(k_{1})P^{-1}-ha_{22}
 P\partial_{y_2}g_2(k_{2})P^{-1}+h^2a_{11}a_{22}  P\partial_{y_2}g_2(k_{1})P^{-1}\\
 &\ \ P  \partial_{y_2}g_2(k_{2})P^{-1} -h^2a_{12}a_{21}
  P\partial_{y_2}g_2(k_{2})P^{-1}P
 \partial_{y_2}g_2(k_{1})P^{-1}|\\
  =& \mid I+ha_{11}  \partial_{y_2}g_2(k_{1})+ha_{22}
  \partial_{y_2}g_2(k_{2}) +h^2a_{11}a_{22}   \partial_{y_2}g_2(k_{1})   \partial_{y_2}g_2(k_{2})\\
  &\ \  -h^2a_{12}a_{21}
  \partial_{y_2}g_2(k_{2})
 \partial_{y_2}g_2(k_{1}) \mid.
\end{array}
\end{equation*}
Consequently, \begin{equation*}
 \begin{array}[c]{ll}  & \textmd{the left hand side of}\ \eqref{resul 2}\\
  =&|I+ha_{11}  \partial_{y_1}g_1(k_{1})+ha_{22}  \partial_{y_1}g_1(k_{2})+h^2a_{11}a_{22} \partial_{y_1}g_1(k_{1}) \partial_{y_1}g_1(k_{2})\\
  &\ \ -h^2a_{12}a_{21}
 \partial_{y_1}g_1(k_{2})
\partial_{y_1}g_1(k_{1})|\\
&| I+ha_{11}  \partial_{y_2}g_2(k_{1})+ha_{22}
  \partial_{y_2}g_2(k_{2}) +h^2a_{11}a_{22}   \partial_{y_2}g_2(k_{1})   \partial_{y_2}g_2(k_{2}) \\
  &\ \ -h^2a_{12}a_{21}
  \partial_{y_2}g_2(k_{2})
 \partial_{y_2}g_2(k_{1})|\\
 =&\textmd{the right hand side of}\ \eqref{resul 2}.
\end{array}
\end{equation*}
 Therefore, all two-stage SSEI methods with $c_1=c_2$ are
 of volume preservation.
 \hfill
\end{proof}

\begin{rem}
We  note  that when $K=\mathbf{0}$, all the results given in this
section reduce  to those proposed in  \cite{Webb2016}, which
demonstrate the wilder  applications of the analysis.
\end{rem}

\section{Applications to various problems}
\label{sec:applications} In this section, we pay attention to the
application of the SSEI methods to various problems and show the
volume preservation of different exponential integrators and
ERKN/RKN methods by using the analysis given in Section \ref{sec:
VPP for DVF}.
\subsection{Highly oscillatory second-order systems}\label{subsec-highly 2}
  Consider the following first-order
 systems
\begin{equation}\label{IVPPP}
y^{\prime}(t)=J^{-1}My(t)+J^{-1}\nabla V(y(t)),
\end{equation}
where the matrix $J$ is constant and invertible, $M$ is a symmetric
matrix and $V$ is a differentiable function. This system is the
exact  pattern  \eqref{IVP1} with
\begin{equation}\label{AG}
K=J^{-1}M,\ \ g(y(t))=J^{-1}\nabla V(y(t)).
\end{equation}
It can be verified that
$$Jg'(y)J^{-1}=JJ^{-1}\nabla^2 V(y)J^{-1}=\nabla^2 V(y)J^{-1}=-g'(y)^{\intercal}$$
and
$$J(K+g'(y))J^{-1}=-(K+g'(y))^{\intercal}.$$
This shows that the set  $\mathcal{H}$ contains all vector fields of
\eqref{IVPPP} with the same $P=J$. Thus in the light of Theorem
\ref{thm vp for H}, all SSEI methods  are of volume preservation for
solving the
 system  \eqref{IVPPP}.

 When  $J=\left(
                                                                \begin{array}{cc}
                                                                  0 & I \\
                                                                  -I& 0 \\
                                                                \end{array}
                                                              \right)
,$ the system     \eqref{IVPPP} is a Hamiltonian system $
y^{\prime}(t)=J^{-1} \nabla H(y(t)) $ with  the Hamiltonian
$
 H(y)=  \dfrac{1}{2}y^{\intercal}My+V(y).
$ All SSEI methods  are of volume preservation for this Hamiltonian
system.
 This is another explanation of the fact that
 symplectic exponential integrators are  of volume preservation for
Hamiltonian systems.

Consider another special and important case of \eqref{IVPPP} by
choosing
$$y=\left(
      \begin{array}{c}
        q \\
        p \\
      \end{array}
    \right),\ \
J^{-1}=\left(
                                                                \begin{array}{cc}
                                                                  0 & I \\
                                                                  -I& N \\
                                                                \end{array}
                                                              \right),\
                                                              \
M=\left(
    \begin{array}{cc}
      \Omega & 0 \\
      0 & I \\
    \end{array}
  \right),\ \ V(y)=
                       V_1(q),
 $$
which gives the following second-order ODE
\begin{equation}q''-Nq'+\Omega q=-\nabla V_1(q).\label{HOS-2}\end{equation}
This system   stands for highly oscillatory problems and many
problems fall into this kind of equation such as the dissipative
molecular dynamics, the (damped) Duffing, charged-particle dynamics
 in a   constant magnetic field and semidiscrete nonlinear
wave equations. Applying the   SSEI methods to \eqref{HOS-2} and
considering Theorem \ref{thm vp for H}, we obtain the following
corollary.

\begin{cor}
The following $s$-stage adapted exponential integrator
\begin{equation}\label{EI for second}
\left\{\begin{array}[c]{ll} &k_i=\exp^{11}(c_i hK)q_n+\exp^{12}(c_i
hK)q'_n-h\sum\limits_{j=1}^{s}a_{ij}\exp^{12}((c_i-c_j) h K)\nabla
V_1(k_{j}),\\
&\qquad \qquad \qquad \qquad \qquad \qquad \qquad \qquad \qquad
\qquad \qquad \qquad  \ i=1,2,\ldots,s,
\\
&q_{n+1}=\exp^{11}(hK)q_n+\exp^{12}(hK)q'_n-h\sum\limits_{i=1}^{s}b_{i}\exp^{12}((1-c_i)hK)\nabla V_1(k_{i}),\\
&q'_{n+1}=\exp^{21}(hK)q_n+\exp^{22}(hK)q'_n-h\sum\limits_{i=1}^{s}b_{i}\exp^{22}((1-c_i)hK)\nabla
V_1(k_{i})
\end{array}\right.
\end{equation}
are  of volume preservation for the  second-order highly oscillatory
equation \eqref{HOS-2},  where  $\exp (hK)$ is partitioned into
$\left(
                                           \begin{array}{cc}
                                             \exp^{11}(hK) & \exp^{12}(hK) \\
                                             \exp^{21}(hK) & \exp^{22}(hK) \\
                                           \end{array}
                                         \right)
$ and the same denotations are used for other matrix-valued
functions. Here $c=(c_1,\ldots,c_s)^{\intercal}, \
b=(b_1,\ldots,b_s)^{\intercal}$ and $A=(a_{ij})_{s\times s}$ are
given in Definition \ref{scheme EI spe}. If $N$ commutes with
$\Omega$, the results of $\exp^{ij}$ for $i,j=1,2$ appearing in
\eqref{EI for second} can be expressed explicitly. After some
calculations, it is obtained that
\begin{equation}\label{EE1234}
 \begin{array}[c]{ll} & \exp^{11}(hK)=e^{\frac{h}{2}N}\Big(\cosh\big(\frac{h}{2}\sqrt{N^2-4\Omega}\big)-
 N\sinh\big(\frac{h}{2}\sqrt{N^2-4\Omega}\big) (\sqrt{N^2-4\Omega})^{-1}\Big),\\
 & \exp^{12}(hK)=2e^{\frac{h}{2}N} \sinh\big(\frac{h}{2}\sqrt{N^2-4\Omega}\big) (\sqrt{N^2-4\Omega})^{-1},\\
 &\exp^{21}(hK)=-\Omega\exp^{12}(hK),\\
 &
 \exp^{22}(hK)=e^{\frac{h}{2}N}\Big(\cosh\big(\frac{h}{2}\sqrt{N^2-4\Omega}\big)+
 N\sinh\big(\frac{h}{2}\sqrt{N^2-4\Omega}\big)
 (\sqrt{N^2-4\Omega})^{-1}\Big).
\end{array}
\end{equation}
The  results are still true if we replace $h$ by $kh$ with any $k\in
\mathbb{R}$.
\end{cor}

If we further assume that $\Omega=0$, the equation \eqref{HOS-2}
becomes
 \begin{equation}q''=Nq'-\nabla V_1(q).\label{HOS-2spe}\end{equation}
One typical example of this system is charged-particle dynamics  in
a   constant magnetic field, which can be expressed by (see
\cite{Hairer2017-2})
 \begin{equation}\label{CPD}
x''= x' \times  B +F(x). \end{equation}
 Here $x(t)\in \RR^3$
describes the position of a particle moving in an electro-magnetic
field,   $F(x) = -\nabla_x U(x)$ is an electric field with the
scalar potential $U(x)$, and $B= \nabla_x \times A(x)$ is a constant
magnetic field with the vector potential  $A(x) = -\frac{1}{2}x
\times B$.

 Under the condition that $\Omega=0$, formulae
\eqref{EE1234} can be written more succinctly as:
\begin{equation*}
  \exp^{11}(hK)=I,\ \ \exp^{12}(hK)= h\varphi_1(hN),\ \
  \exp^{21}(hK)=0,\ \
 \exp^{22}(hK)=\varphi_0(hN),
\end{equation*}
where the $\varphi$-functions are defined by (see
\cite{Hochbruck2010,Hochbruck2009})
\begin{equation}
 \varphi_0(z)=e^{z},\ \ \varphi_k(z)=\int_{0}^1
e^{(1-\sigma)z}\frac{\sigma^{k-1}}{(k-1)!}d\sigma, \ \ k=1,2,\ldots.
\label{pphi}%
\end{equation}
We then get the following volume preserving  methods for the special
and important second-order system \eqref{HOS-2spe}.
\begin{cor}
The following $s$-stage   integrator
\begin{equation}\label{spe EI for second}
\left\{\begin{array}[c]{ll} &k_i=q_n+c_ih\varphi_1(c_i
hN)q'_n-h^2\sum\limits_{j=1}^{s}a_{ij}(c_i-c_j) \varphi_1((c_i-c_j)
h N)\nabla
V_1(k_{j}),\\
&\qquad \qquad \qquad \qquad \qquad \qquad \qquad \qquad \qquad
\qquad \qquad \qquad  \ i=1,2,\ldots,s,
\\
&q_{n+1}= q_n+h\varphi_1(hN)q'_n-h^2\sum\limits_{i=1}^{s}b_{i}(1-c_i)  \varphi_1((1-c_i)hN)\nabla V_1(k_{i}),\\
&q'_{n+1}=\varphi_0(hN)q'_n-h\sum\limits_{i=1}^{s}b_{i}\varphi_0((1-c_i)hN)\nabla
V_1(k_{i})
\end{array}\right.
\end{equation}
are  of volume preservation for the  highly oscillatory second-order
system \eqref{HOS-2spe}, where $c=(c_1,\ldots,c_s)^{\intercal}, \
b=(b_1,\ldots,b_s)^{\intercal}$ and $A=(a_{ij})_{s\times s}$ are
given in Definition \ref{scheme EI spe}.
\end{cor}

\begin{rem}
 It is noted that the
above two corollaries are new
 discoveries which are of great importance to Geometric Integration for
  second-order highly oscillatory problems.
\end{rem}

\subsection{Separable partitioned  systems}
 It has been proved in \cite{Webb2016} that the set $\mathcal{S}$
contains all separable partitioned  systems.  As an example, we
consider
\begin{equation}\begin{aligned}& \left(
                   \begin{array}{c}
                     q \\
                      p \\
                   \end{array}
                 \right)
'=\left(
    \begin{array}{c}
      p \\
      -\Omega q+\tilde{g}(q) \\
    \end{array}
  \right)
 =\left(
    \begin{array}{cc}
      0 & I \\
      -\Omega  & 0 \\
    \end{array}
  \right) \left(
                   \begin{array}{c}
                     q \\
                      p \\
                   \end{array}
                 \right)+\left(
                           \begin{array}{c}
                             0 \\
                            \tilde{g}(q) \\
                           \end{array}
                         \right),
\label{2prob}%
\end{aligned}\end{equation}
which is exactly the system \eqref{IVP1} with
\begin{equation*}
 K=\left(
    \begin{array}{cc}
      0 & I \\
      -\Omega  & 0 \\
    \end{array}
  \right),\ g=\left(
                           \begin{array}{c}
                             0 \\
                             \tilde{g}(q) \\
                           \end{array}
                         \right), \ f=\left(
    \begin{array}{c}
      p \\
      -\Omega q+\tilde{g}(q) \\
    \end{array}
  \right).
\end{equation*}

It can be checked that  $f$ and $g$ both fall into $\mathcal{S}$
with the same $P=\textmd{diag}(I,-I)$. For this special matrix $K$,
it is clear that
\begin{equation}\label{exp spe}
e^{xK} = \left(
\begin{array}
[c]{cc}%
\phi_{0} (x^{2}\Omega ) & x\phi_{1} (x^{2}\Omega )\\
-x\Omega\phi_{1}  (x^{2}\Omega ) & \phi_{0} (x^{2}\Omega )
\end{array}
\right),\ \ \   \textmd{for}\ \ \   x\in \mathbb{R}\end{equation}
with $
\phi_{i}(\Omega):=\sum\limits_{l=0}^{\infty}\dfrac{(-1)^{l}\Omega^{l}}{(2l+i)!}
$ for  $i=0,1.$
 Thus the exponential
integrator \eqref{EI} has a special form, which is called as
extended RKN (ERKN) integrators  and is represented below.
\begin{defi}
\label{scheme Erkn} (See \cite{wu2010-1}) An $s$-stage
   ERKN
integrator for solving (\ref{2prob}) is defined by%
 \begin{equation*}
\left\{\begin{array}
[c]{ll}%
Q_{i} &
=\phi_{0}(c_{i}^{2}V)q_{n}+hc_{i}\phi_{1}(c_{i}^{2}V)p_{n}+h^{2}\textstyle\sum\limits_{j=1}^{s}\bar{a}_{ij}(V)\tilde{g}(Q_{j}),\ \ \ i=1,\ldots,s,\\
q_{n+1} & =\phi_{0}(V)q_{n}+h\phi_{1}(V)p_{n}+h^{2}\textstyle\sum
\limits_{i=1}^{s}\bar{b}_{i}(V) \tilde{g}(Q_{i}),\\
p_{n+1} &
=-h\Omega\phi_{1}(V)q_{n}+\phi_{0}(V)p_{n}+h\textstyle\sum\limits_{i=1}^{s}b_{i}(V)
\tilde{g}(Q_{i}),
\end{array}\right.
\end{equation*}
 where $c_i$ for $i=1,\ldots,s$ are real constants, $b_{i}(V)$,
$\bar{b}_{i}(V)$ for $i=1,\ldots,s$, and $\bar{a}_{ij}(V)$ for
$i,j=1,\ldots,s$ are matrix-valued functions of $V\equiv
h^{2}\Omega.$
\end{defi}

ERKN integrators were firstly proposed in \cite{wu2010-1}. Further
discussions about  ERKN integrators have been given recently,
including symmetric integrators (see \cite{wang2017-Cal}),
symplectic integrators (see \cite{wu2013-book}), energy-preserving
integrators  (see \cite{wu2013-JCP}) and other kinds integrators
(see \cite{wang-2016,wu2017-JCAM}). However, to our knowledge,  the
volume-preserving property of ERKN integrators has not been
researched yet in the literature. With the analysis given in this
paper, we get the following VP result of ERKN integrators.
\begin{cor}
Consider a  kind of $s$-stage ERKN integrators
\begin{equation}
\left\{\begin{array}
[c]{ll}%
Q_{i} &
=\phi_{0}(c_{i}^{2}V)q_{n}+hc_{i}\phi_{1}(c_{i}^{2}V)p_{n}+h^{2}\textstyle\sum\limits_{j=1}^{s}a_{ij}(c_i-c_j)\phi_{1}((c_i-c_j)^2V)\tilde{g}(Q_{j}),\\
& \qquad \qquad \qquad \qquad \qquad \qquad \qquad \qquad \qquad \qquad \qquad \qquad  i=1,\ldots,s,\\
q_{n+1} & =\phi_{0}(V)q_{n}+h\phi_{1}(V)p_{n}+h^{2}\textstyle\sum
\limits_{i=1}^{s}b_{i}(1-c_i)\phi_{1}((1-c_i)^2V) \tilde{g}(Q_{i}),\\
p_{n+1} &
=-h\Omega\phi_{1}(V)q_{n}+\phi_{0}(V)p_{n}+h\textstyle\sum\limits_{i=1}^{s}b_{i}
\phi_{0}((1-c_i)^2V) \tilde{g}(Q_{i}),
\end{array}\right.
\label{ERKN rev co}%
\end{equation}
where  $c=(c_1,\ldots,c_s)^{\intercal}, \
b=(b_1,\ldots,b_s)^{\intercal}$ and $A=(a_{ij})_{s\times s}$ are
given in Definition \ref{scheme EI spe}. Under the condition that
$b_j\neq0$ for all $j=1,\ldots,s,$ all one-stage and two-stage (with
$c_1=c_2$)   ERKN integrators \eqref{ERKN rev co}, and compositions
thereof, are of volume preservation for solving the separable
partitioned  system \eqref{2prob}.
\end{cor}
\begin{proof}
In the light of Definition \ref{scheme EI spe} and the result
\eqref{exp spe}, we adapt the SSEI methods to the system
(\ref{2prob}) and then get the scheme \eqref{ERKN rev co}. Based on
Theorem \ref{thm vp for S}, the  volume preserving  result of
\eqref{ERKN rev co} is immediately obtained.\hfill
\end{proof}

\begin{rem}
We note that this is a novel result which  studies the
volume-preserving ERKN integrators for \eqref{2prob}.  Moreover,
from the scheme \eqref{ERKN rev co}, it can be observed that all
one-stage and two-stage  with $c_1=c_2$  ERKN integrators  are
explicit, which means that we obtain explicit
 volume preserving   ERKN integrators for the separable partitioned system
\eqref{2prob}.
\end{rem}

  If $\Omega$ is a
 symmetric and positive  semi-definite matrix,  and
$\tilde{g}(q)=-\nabla U(q)$, the system \eqref{2prob} is  an
oscillatory  Hamiltonian system $\left(
                   \begin{array}{c}
                     q \\
                      p \\
                   \end{array}
                 \right)'= \left(
                             \begin{array}{cc}
                               0 & I \\
                               -I & 0 \\
                             \end{array}
                           \right)
                  \nabla H(q,p)$  with  the Hamiltonian
 \begin{equation} \label{HH2} H(q,p) = \frac{1}{2}p^{\intercal}p+ \frac{1}{2}
q^{\intercal}\Omega q+U (q).\end{equation} It has been noted in
Subsection \ref{subsec-highly 2} that this vector field falls into
the set $\mathcal{H}$. Thus Theorem \ref{thm vp for H} provides
another way to prove  the well-known fact that all symplectic ERKN
integrators \eqref{ERKN rev co} are of volume preservation for the
oscillatory Hamiltonian system \eqref{HH2}.

In what follows, we study the volume-preserving property of RKN
methods for  second-order ODEs. Consider $\Omega=0$ for the above
analysis and under this situation, ERKN integrators reduce to RKN
methods. Therefore, we are  now  in a position to present the
following volume-preserving property for RKN methods.
\begin{cor}\label{cor rkn}
Consider the following $s$-stage RKN methods
\begin{equation}
\left\{\begin{array}
[c]{ll}%
Q_{i} &
=q_{n}+hc_{i}q'_{n}+h^{2}\textstyle\sum\limits_{j=1}^{s}a_{ij}(c_i-c_j)\tilde{g}(Q_{j}),\ \   i=1,\ldots,s,\\
q_{n+1} & = q_{n}+h q'_{n}+h^{2}\textstyle\sum
\limits_{i=1}^{s}b_{i}(1-c_i)  \tilde{g}(Q_{i}),\\
q'_{n+1} & = q'_{n}+h\textstyle\sum\limits_{i=1}^{s}b_{i}
  \tilde{g}(Q_{i})
\end{array}\right.
\label{RKN rev co}%
\end{equation}
with the coefficients  $c=(c_1,\ldots,c_s)^{\intercal}, \
b=(b_1,\ldots,b_s)^{\intercal}$ and $A=(a_{ij})_{s\times s}$ of an
$s$-stage RK method.   If this RK method   is symplectic and
$b_j\neq0$ for all $j=1,\ldots,s,$ all one-stage and two-stage   RKN
methods \eqref{RKN rev co}  (and compositions thereof) are of volume
preservation for solving the second-order system $q''=\tilde{g}(q)$.
\end{cor}

\begin{rem}
It is noted that the fact of this corollary can be derived in a
different way. Hairer, Lubich and Wanner have proved in
\cite{hairer2006} that any symplectic  RK method with at most two
stages (and any composition of such methods) is volume preserving
for separable divergence free vector fields. Rewriting the
second-order equation $q''=\tilde{g}(q)$ as a first-order system and
applying symplectic RK methods to it implies the result of Corollary
\ref{cor rkn}.  It is seen that the  analysis of volume-preserving
ERKN integrators provides an alternative derivation of the volume
preservation of RKN methods.
\end{rem}

\subsection{Other applications}
It has been shown in \cite{Webb2016} that $\mathcal{F}^{(\infty)}$
contains the affine vector fields $f(y)=Ly+d$ such that
$\abs{I+\frac{h}{2}L}=\abs{I-\frac{h}{2}L}$ for all $h>0$. For
solving the system in the affine vector fields, the exponential
integrator \eqref{EI}  becomes
\begin{equation*}
\left\{\begin{array}[c]{ll} &k_i=e^{c_i
hL}y_n+h\sum\limits_{j=1}^{s}\bar{a}_{ij}(hL)d,\ \ i=1,2,\ldots,s,
\\
&y_{n+1}=e^{ hL}y_n+h\sum\limits_{i=1}^{s}\bar{b}_{i}(hL)d.
\end{array}\right.
\end{equation*}
In the light of Theorem  \ref{thm vp for FIN},  this SSEI method is
of volume preservation for the   affine  vector fields.

It was also noted in \cite{Webb2016} that $\mathcal{F}^{(\infty)}$
contains the vector fields $f(y)$ such that $f'(y)=JS(y)$  with a
skew-symmetric matrix  $J$ and the symmetric matrix $S(y)$. Assume
that
\begin{equation}\label{KGspe}K=JM,\ \
g'(y)=JS(y),\end{equation} where $M$ is a symmetric matrix. The
system  \eqref{IVP1} with the vector field \eqref{KGspe} falls into
$\mathcal{F}^{(\infty)}$. Thus all SSEI methods are of volume
preservation for the   vector field \eqref{KGspe}.

\section{Numerical examples} \label{sec:examples}
In order to show the performance of SSEI methods,  the solvers
chosen for comparison are:
\begin{itemize}
\item SSRK1: the Gauss-Legendre method   of order two;
\item SSEI1: the one-stage SSEI method with the coefficients \eqref{rev
co}, where   the RK method  \eqref{RK co} is chosen as  SSRK1;
\item SSRK2:  the Gauss-Legendre method   of order four;
\item SSEI2: the two-stage SSEI method with the coefficients \eqref{rev
co},  where the RK method  \eqref{RK co} is chosen as  SSRK2.
\end{itemize}
It is noted  that all these methods are  general implicit and we use
fixed-point iteration here. We
  set $10^{-16}$ as the error tolerance
 and $100$ as the maximum number of each fixed-point iteration.

\vskip2mm \vskip2mm\noindent\textbf{Problem 1.}  As the first
numerical example, we consider  the Duffing equation defined by
\begin{equation*}
\begin{aligned}& \left(
                   \begin{array}{c}
                     q \\
                      p \\
                   \end{array}
                 \right)
'= \left(
    \begin{array}{cc}
      0& 1\\
       -\omega^{2}-k^{2} &0 \\
    \end{array}
  \right)\left(
                   \begin{array}{c}
                     q \\
                      p \\
                   \end{array}
                 \right)+
\left(
                                                                           \begin{array}{c}
                                                                           0\\
2k^{2}q^{3}
                                                                           \end{array}
                                                                         \right),\
                                                                         \ \left(
                                                                            \begin{array}{c}
                                                                              q(0) \\
                                                                              p(0) \\
                                                                            \end{array}
                                                                          \right)=\left(
                                                                                    \begin{array}{c}
                                                                                      0 \\
                                                                                      \omega \\
                                                                                    \end{array}
                                                                                  \right).
\end{aligned}\end{equation*}
The exact solution of this system is $q(t)=sn(\omega t;k/\omega)$
with the  Jacobi elliptic function  $sn$. Since it is a Hamiltonian
system, all the methods chosen for comparison  are volume
preserving. For this problem, we choose $k=0.07 $ and $\omega=20$
and  then  solve it  on  the interval $[0, 100]$ with different
stepsizes $h=1/2,1/10,1/50,1/200$. The numerical flows at the time
points $\{\frac{1}{2}i\}_{i=1,\ldots,200}$  of the four methods are
given in Figure \ref{p1}.  From the results, it can be observed
clearly that the  integrators SSEI1 and SSEI2 perform better than
Runge-Kutta methods in the preservation of volume. Finally, we
integrate this problem in $[0,t_{\textmd{end}}]$ with
 $h= 0.1/2^{i}$ for $i=1,\ldots,4.$  The  relative global errors for
 different $t_{\textmd{end}}$   are
presented in Figure \ref{p2}. These results show clearly again that
exponential integrators have better accuracy than Runge-Kutta
methods.
 \begin{figure}[ptb]
\centering
\includegraphics[width=3.0cm,height=3.0cm]{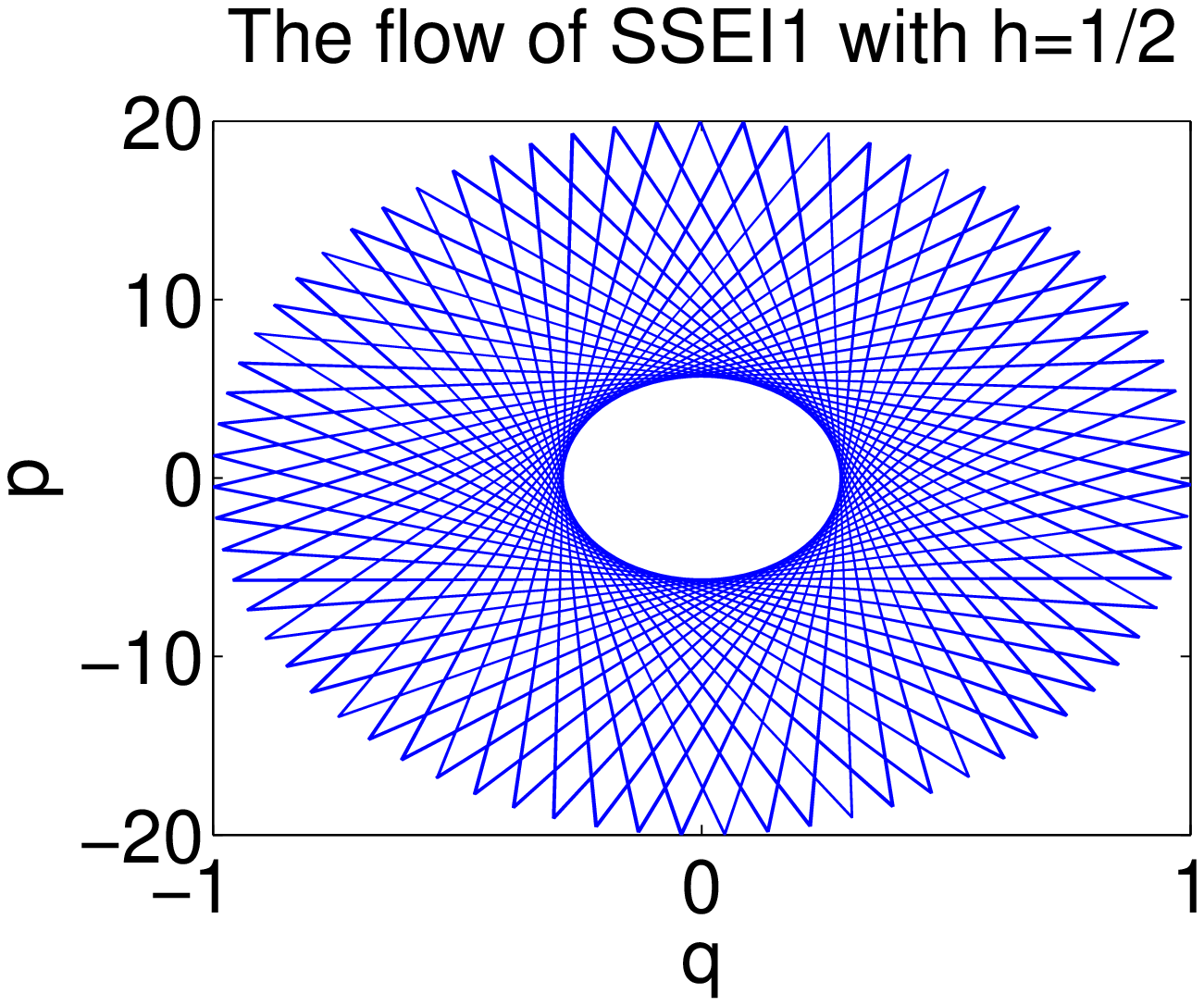}
\includegraphics[width=3.0cm,height=3.0cm]{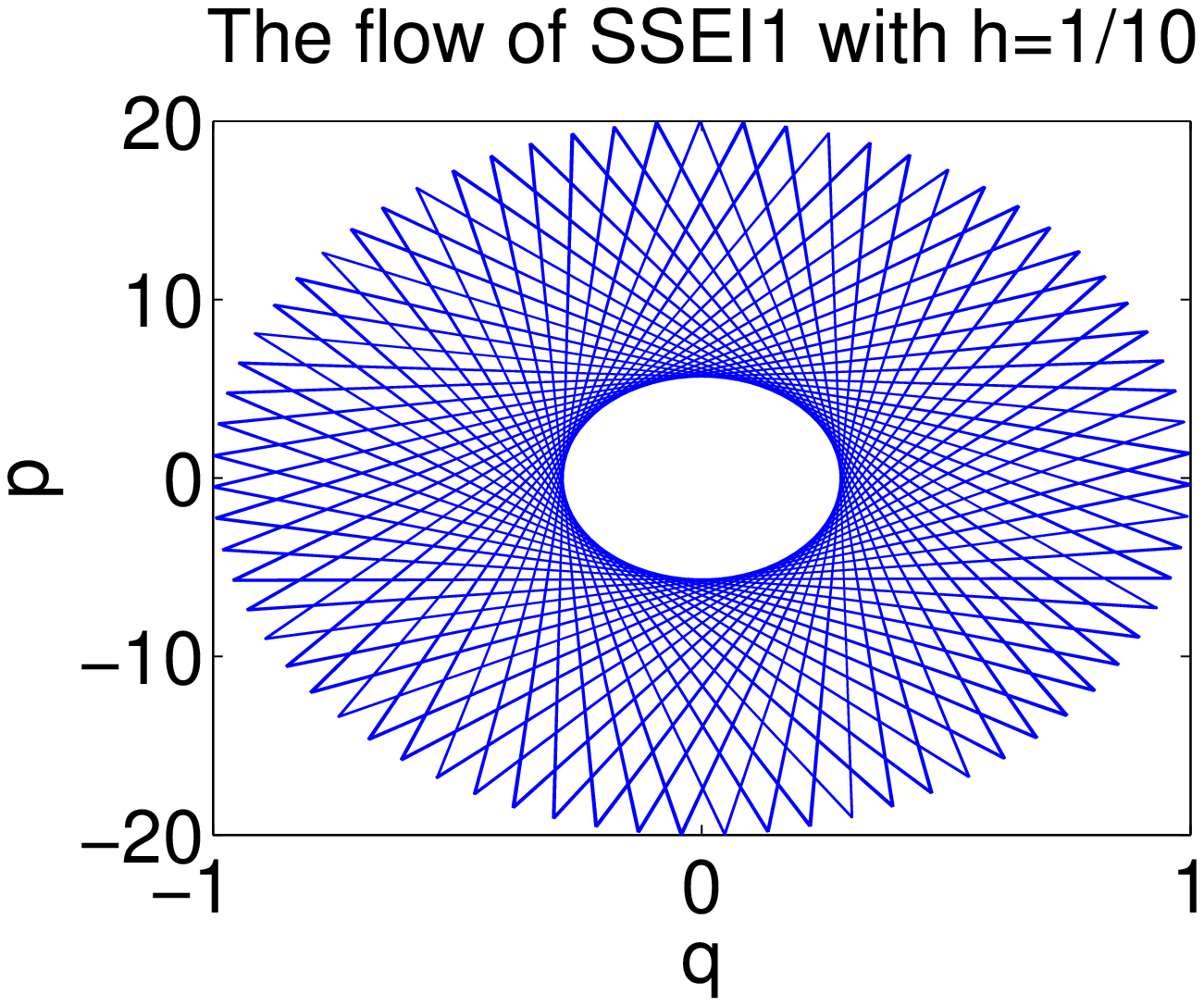}
\includegraphics[width=3.0cm,height=3.0cm]{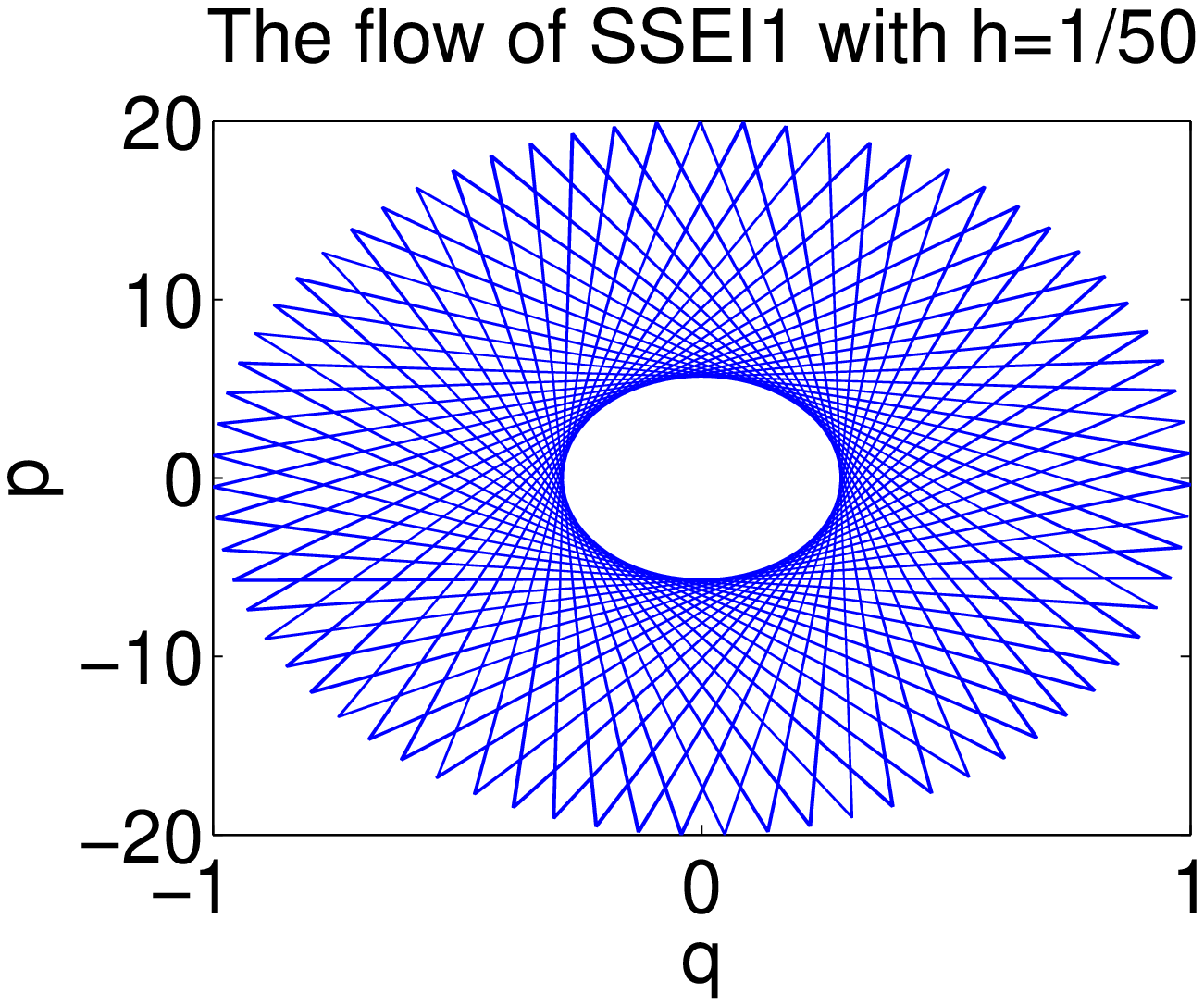}
\includegraphics[width=3.0cm,height=3.0cm]{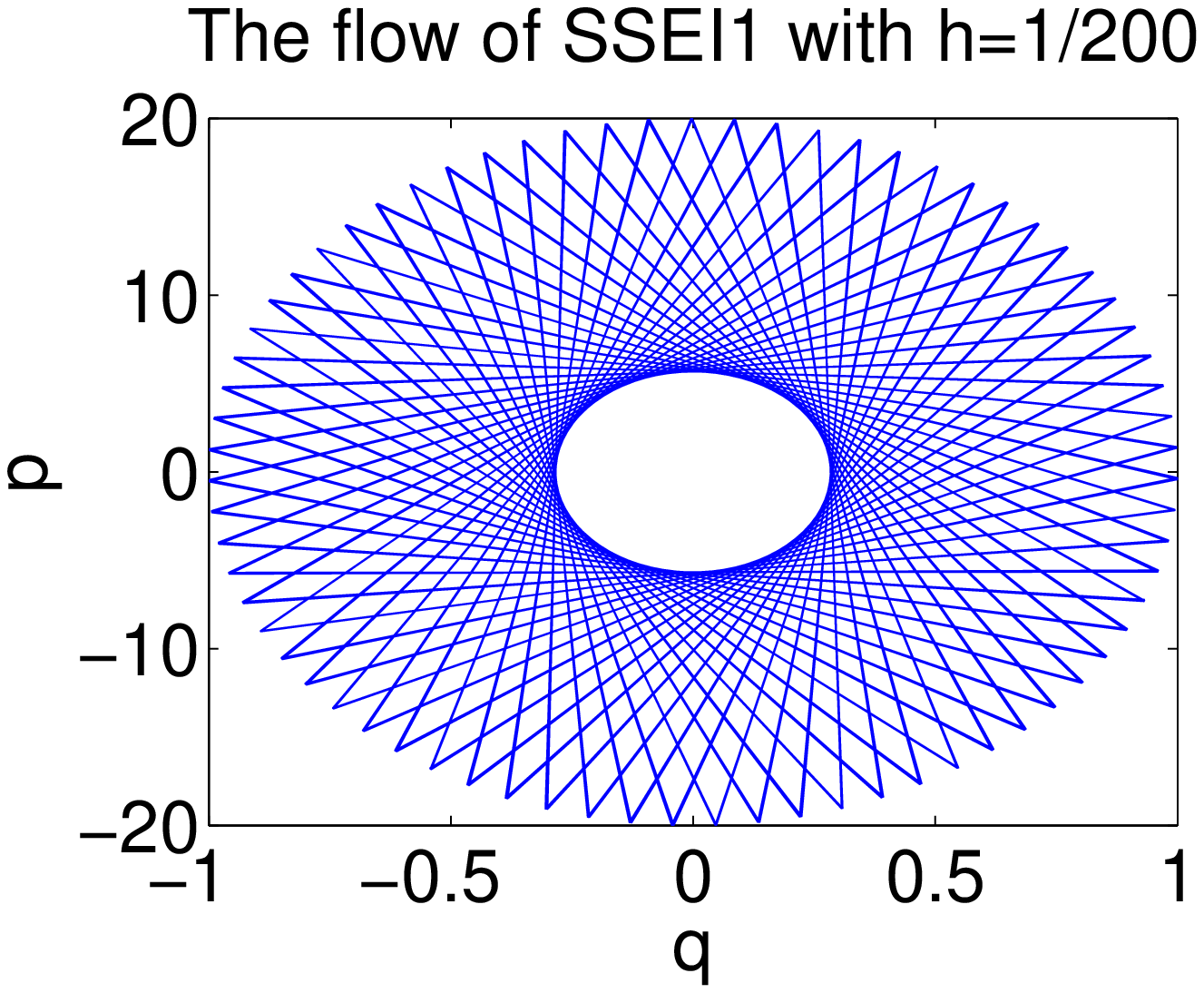}\\
\includegraphics[width=3.0cm,height=3.0cm]{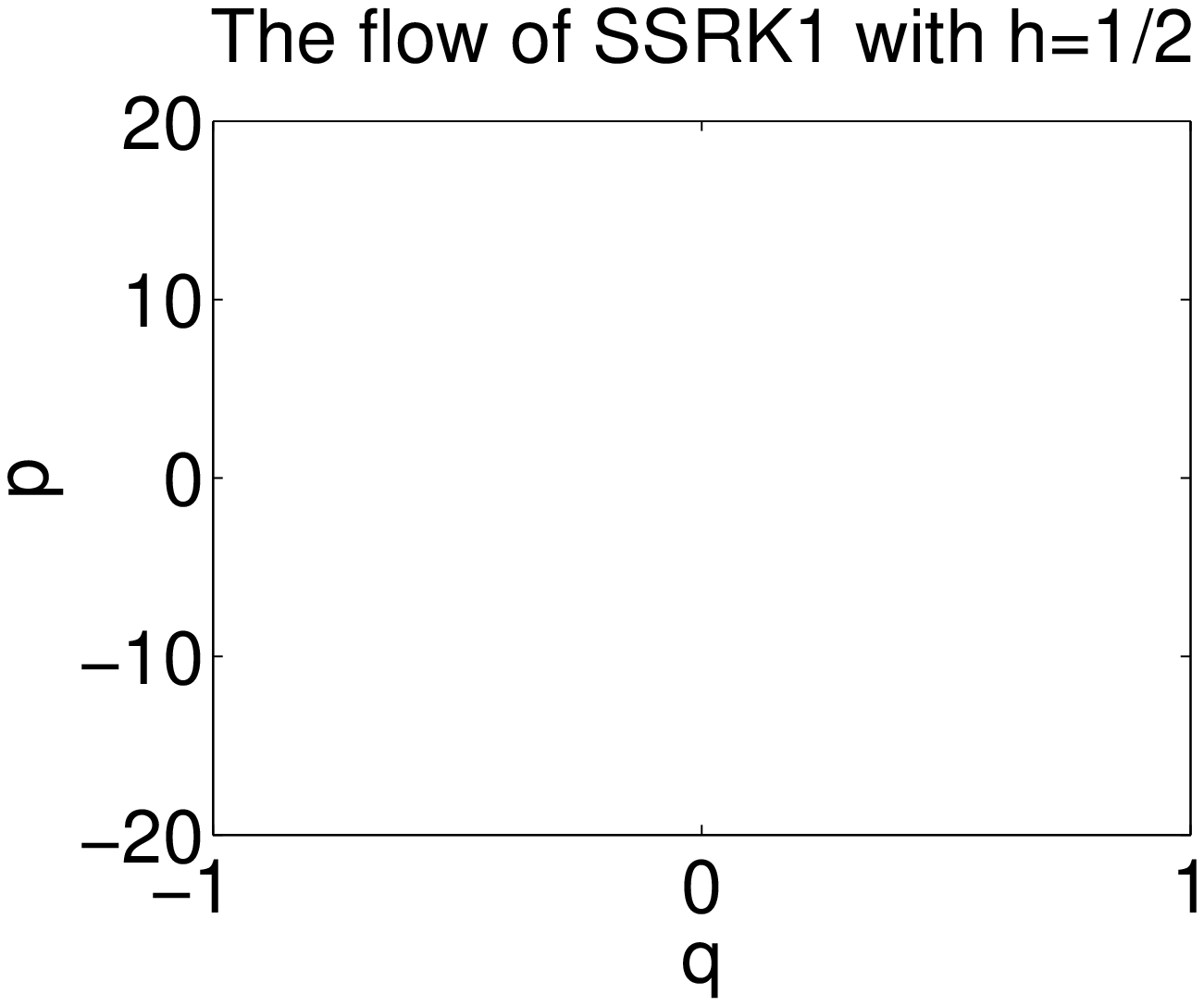}
\includegraphics[width=3.0cm,height=3.0cm]{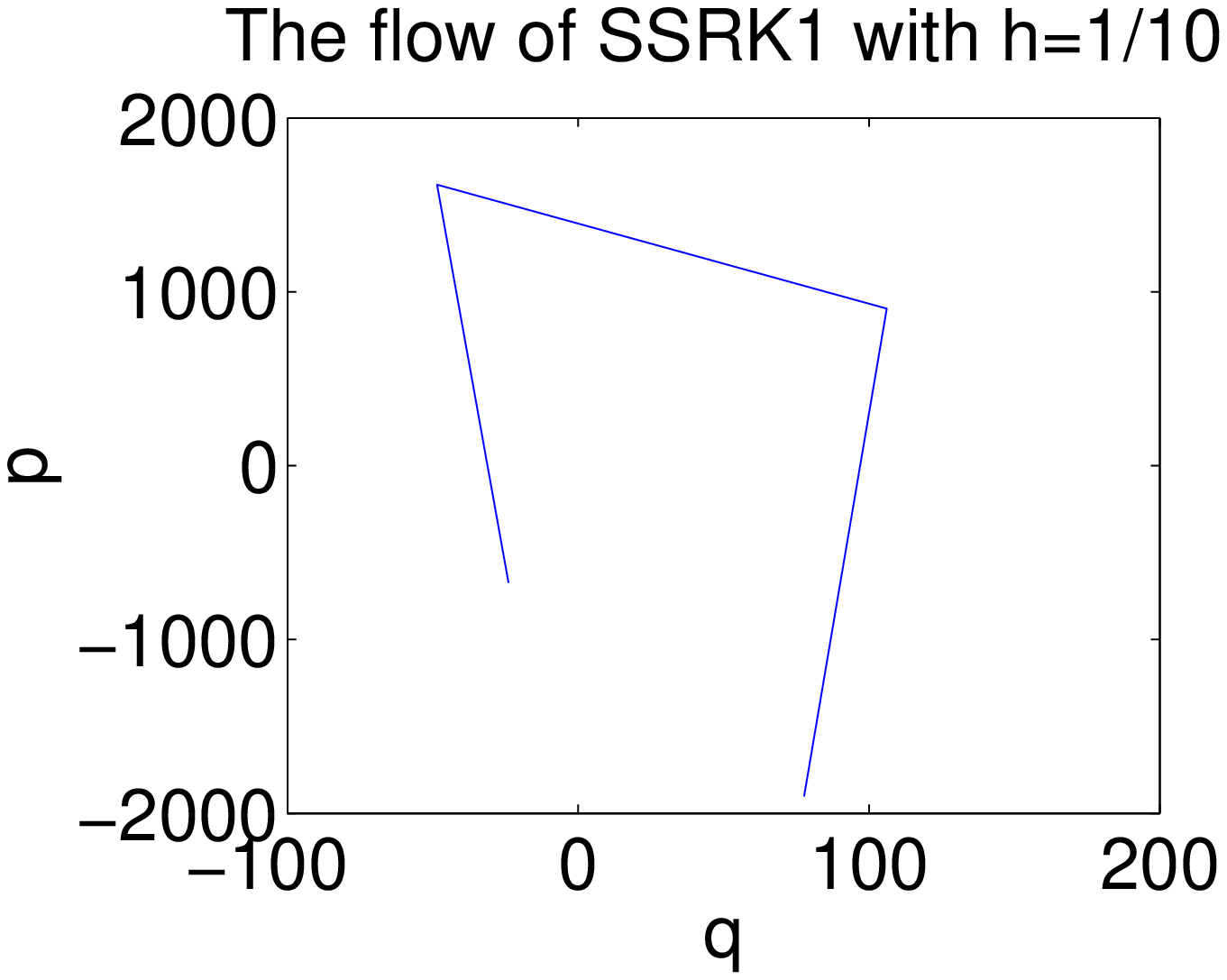}
\includegraphics[width=3.0cm,height=3.0cm]{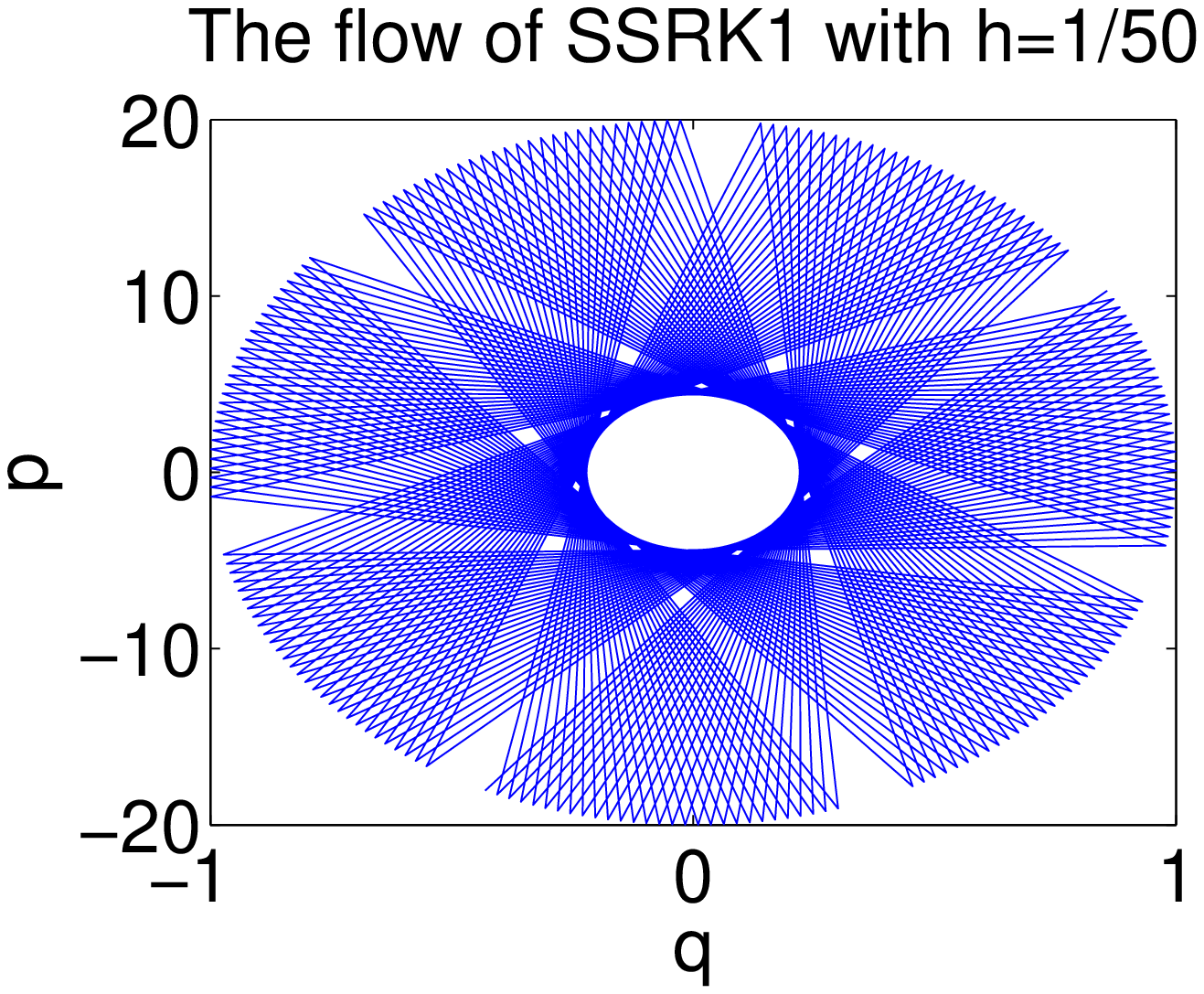}
\includegraphics[width=3.0cm,height=3.0cm]{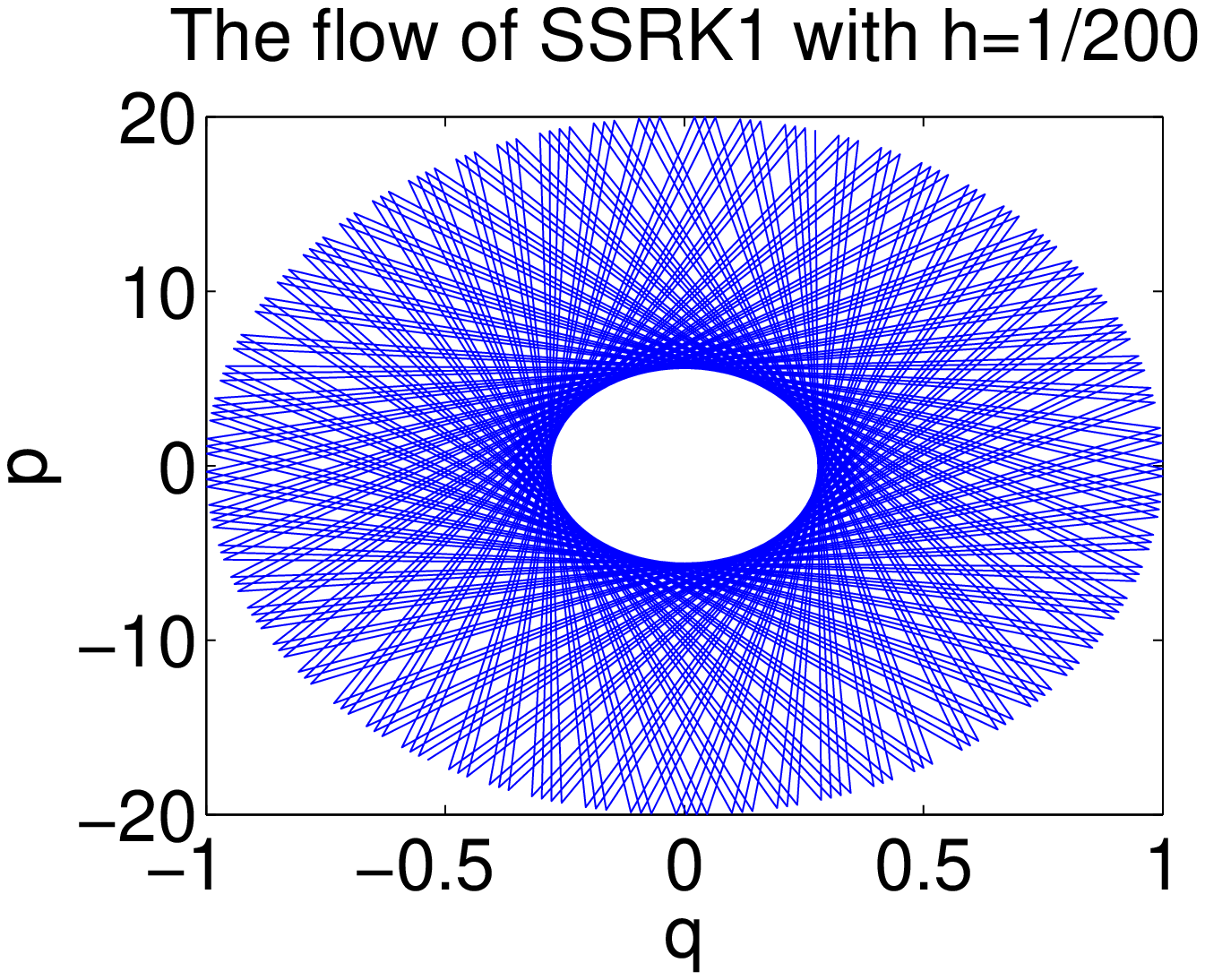}\\
\includegraphics[width=3.0cm,height=3.0cm]{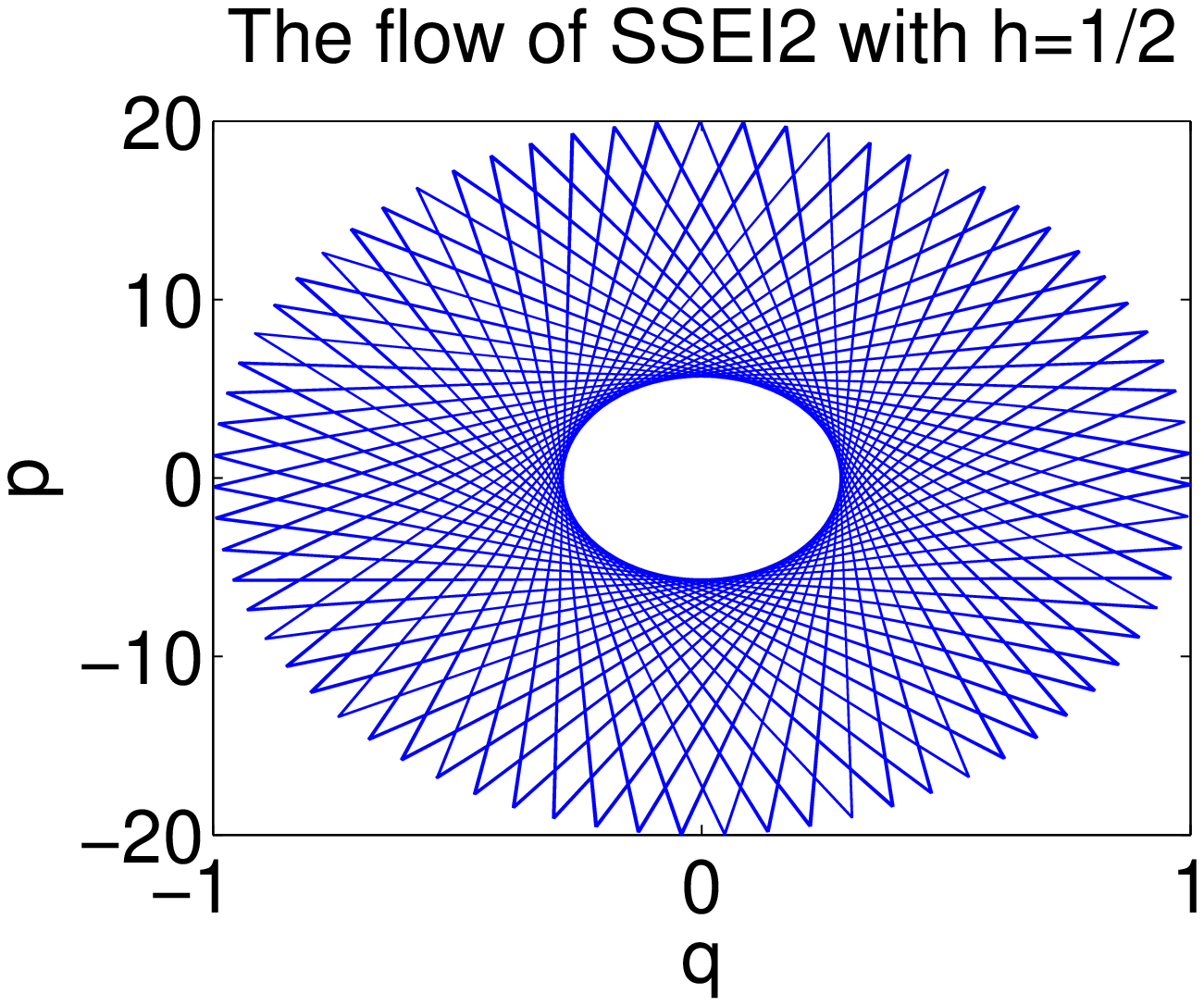}
\includegraphics[width=3.0cm,height=3.0cm]{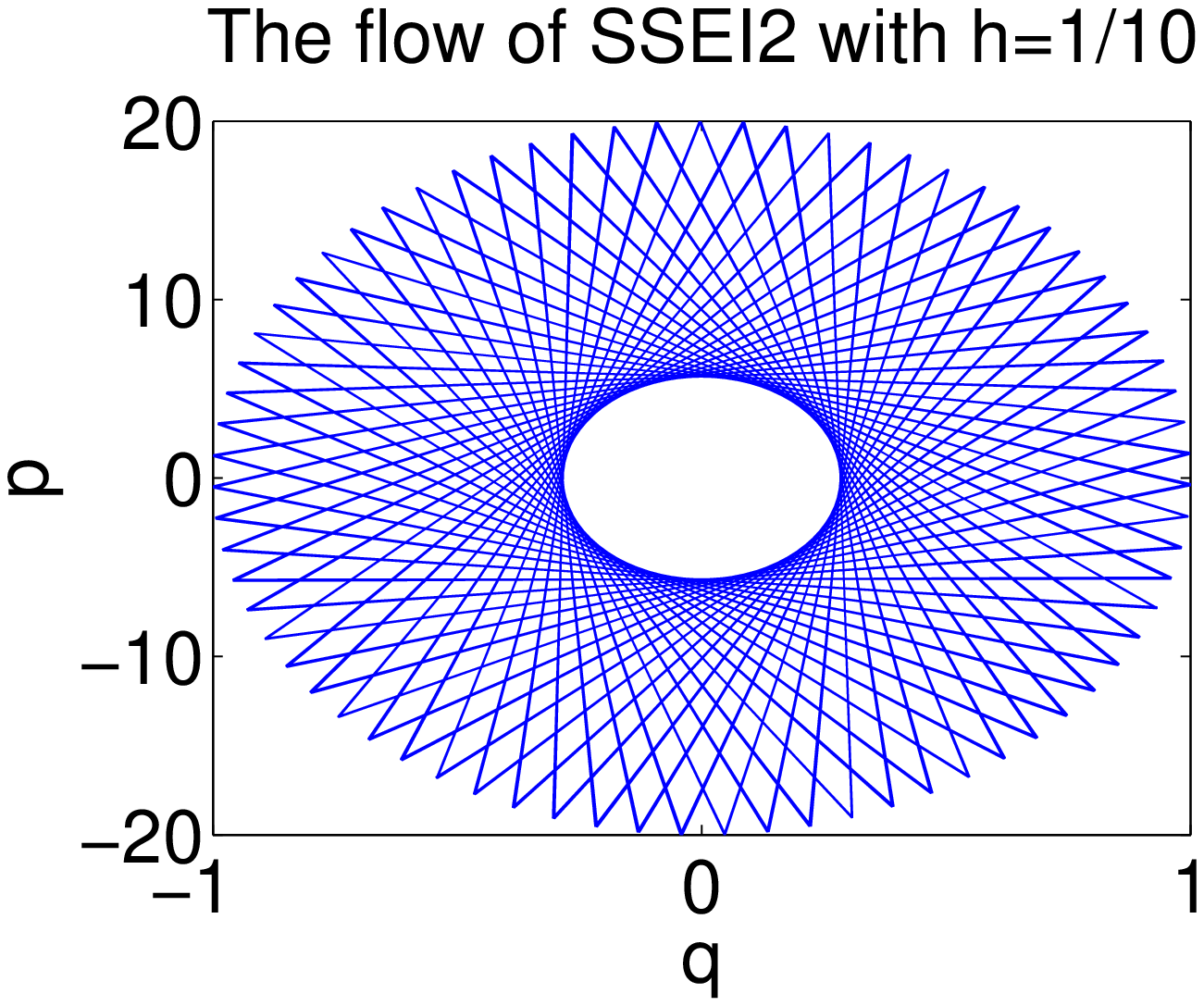}
\includegraphics[width=3.0cm,height=3.0cm]{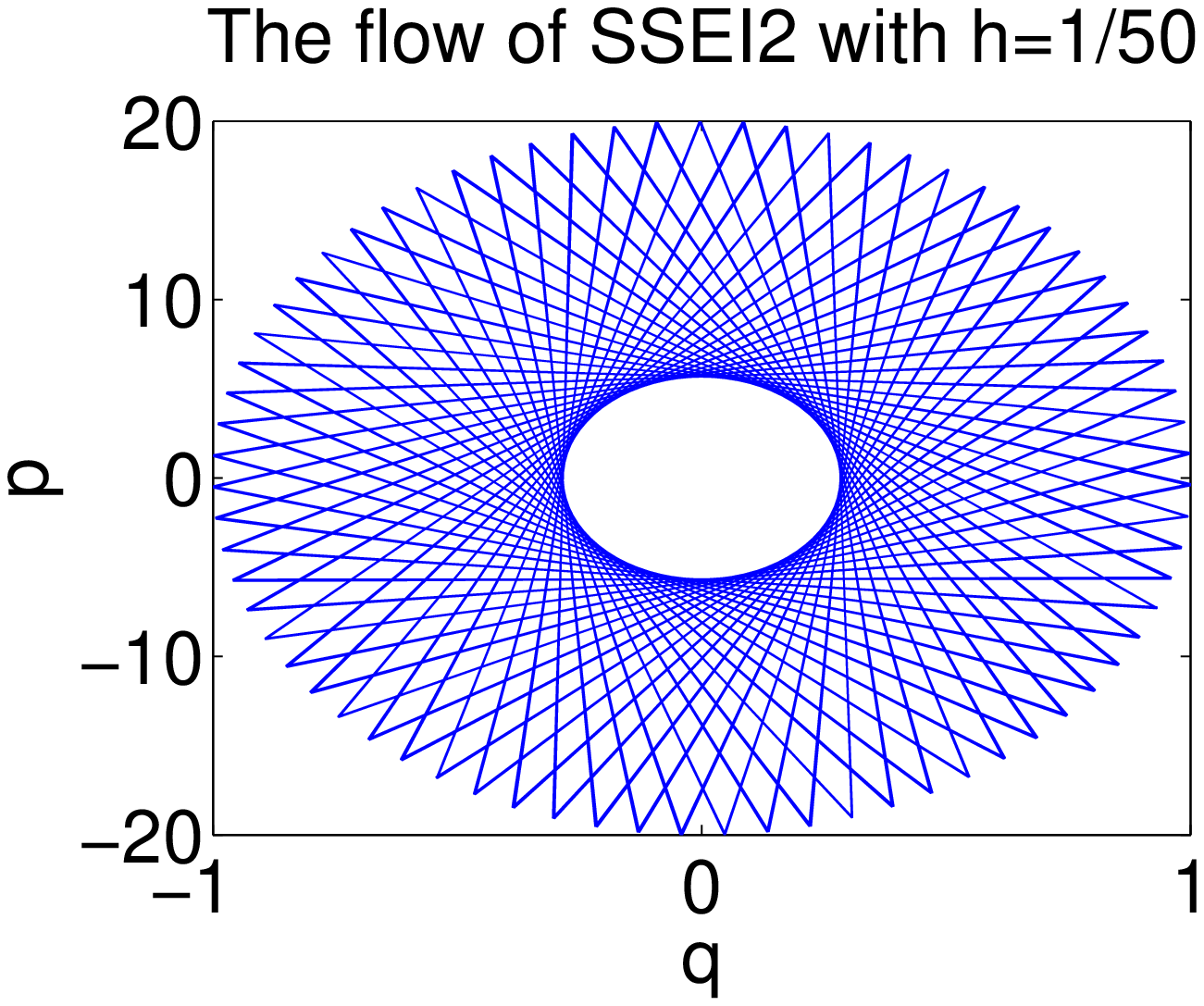}
\includegraphics[width=3.0cm,height=3.0cm]{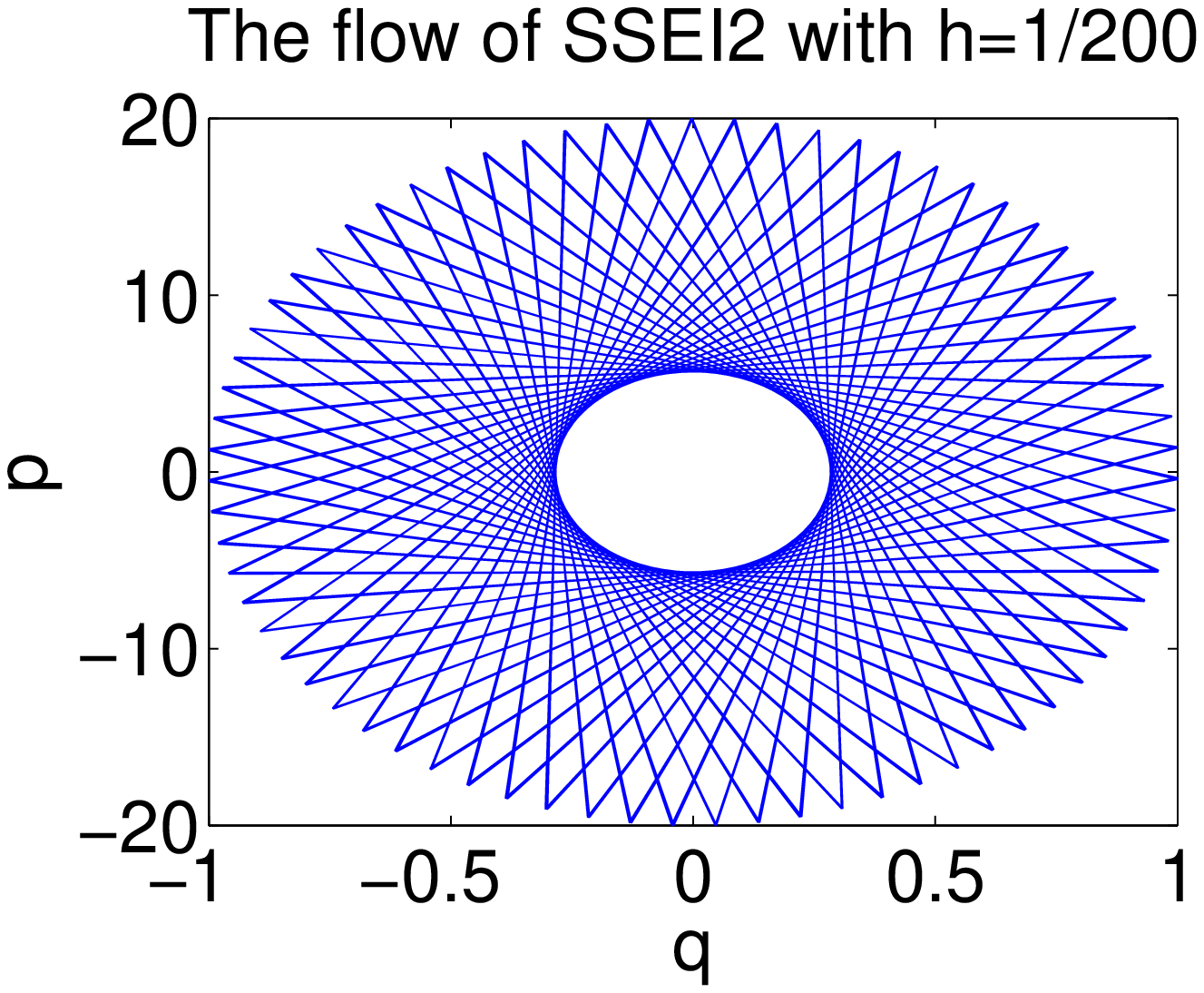}\\
\includegraphics[width=3.0cm,height=3.0cm]{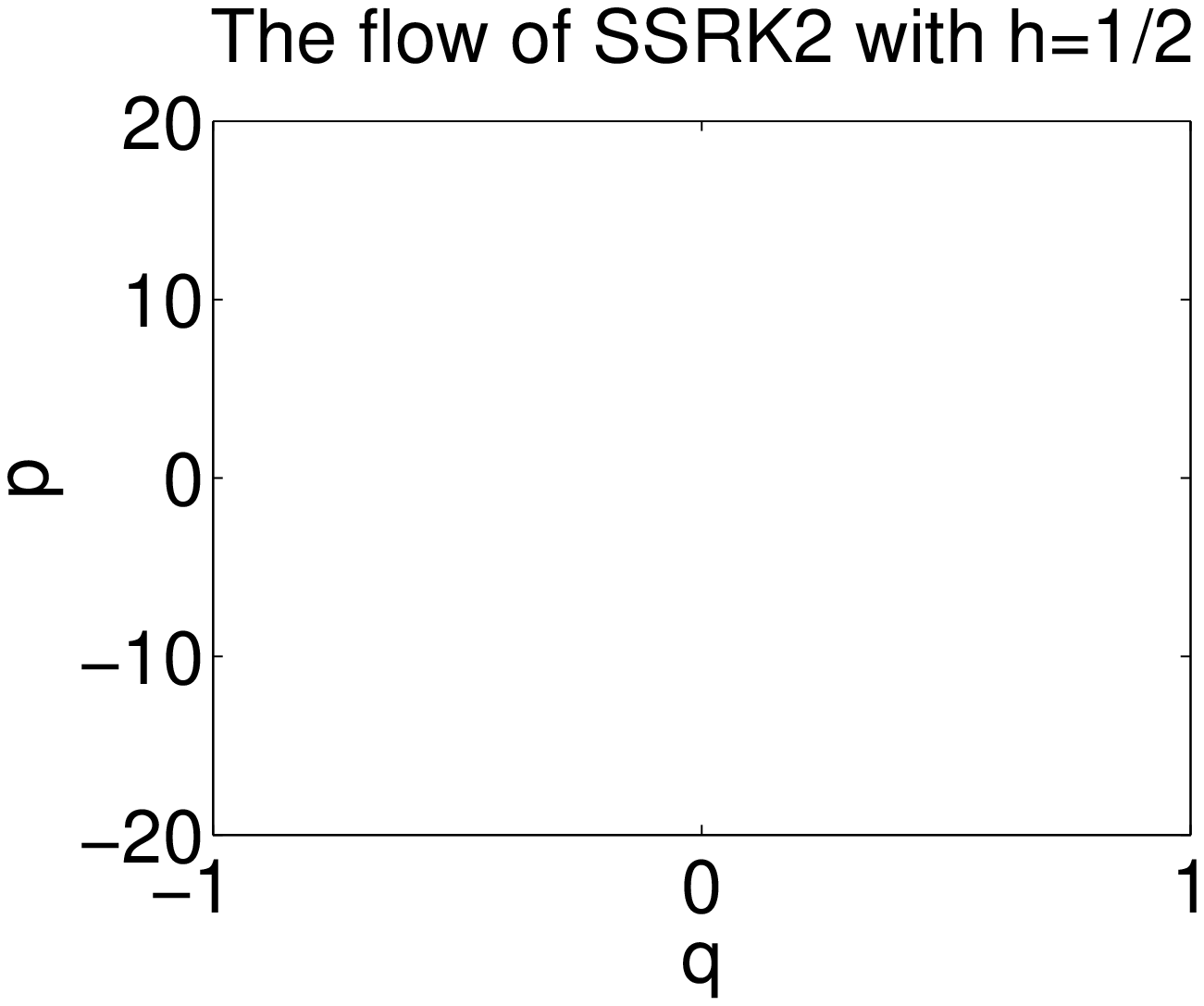}
\includegraphics[width=3.0cm,height=3.0cm]{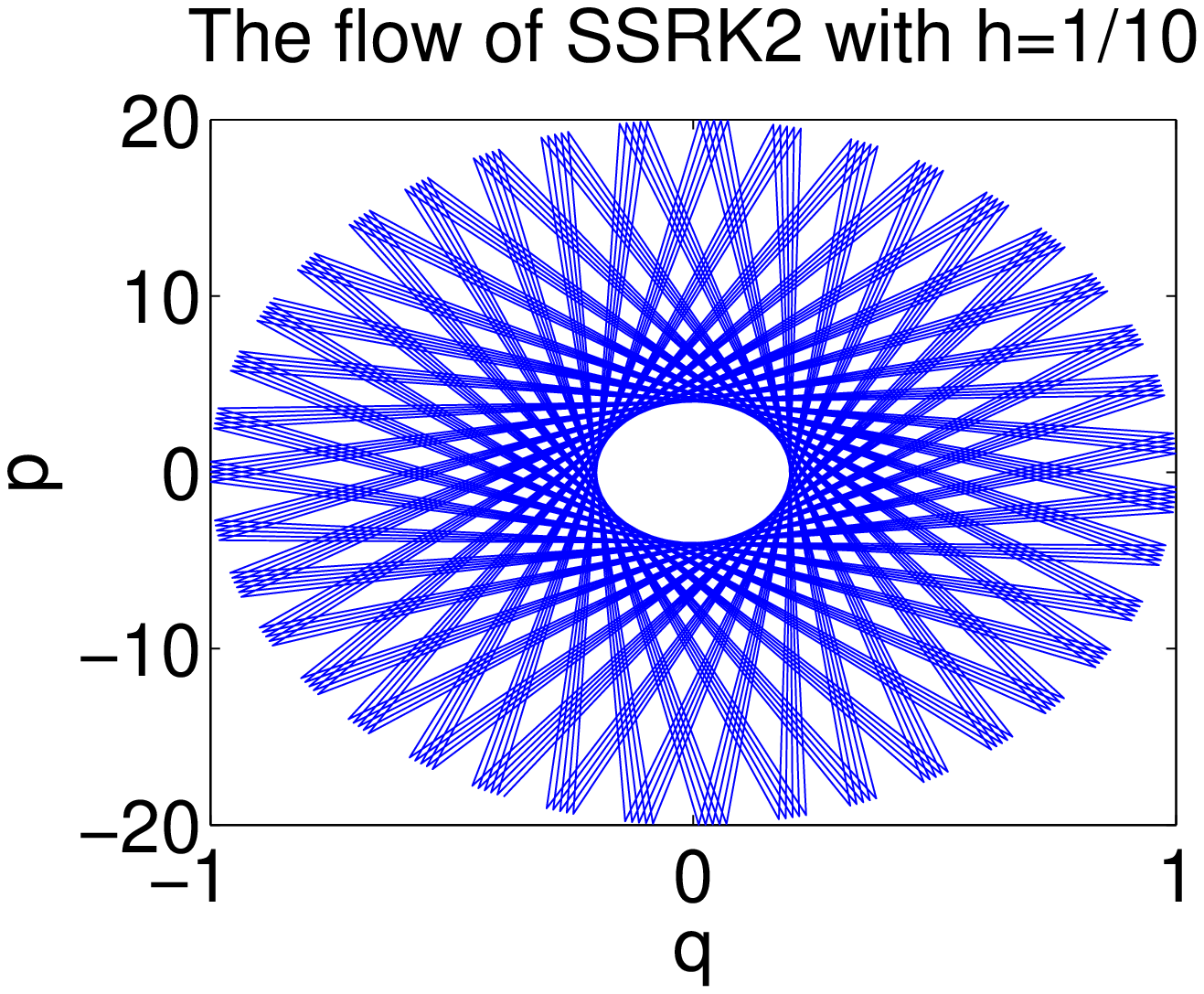}
\includegraphics[width=3.0cm,height=3.0cm]{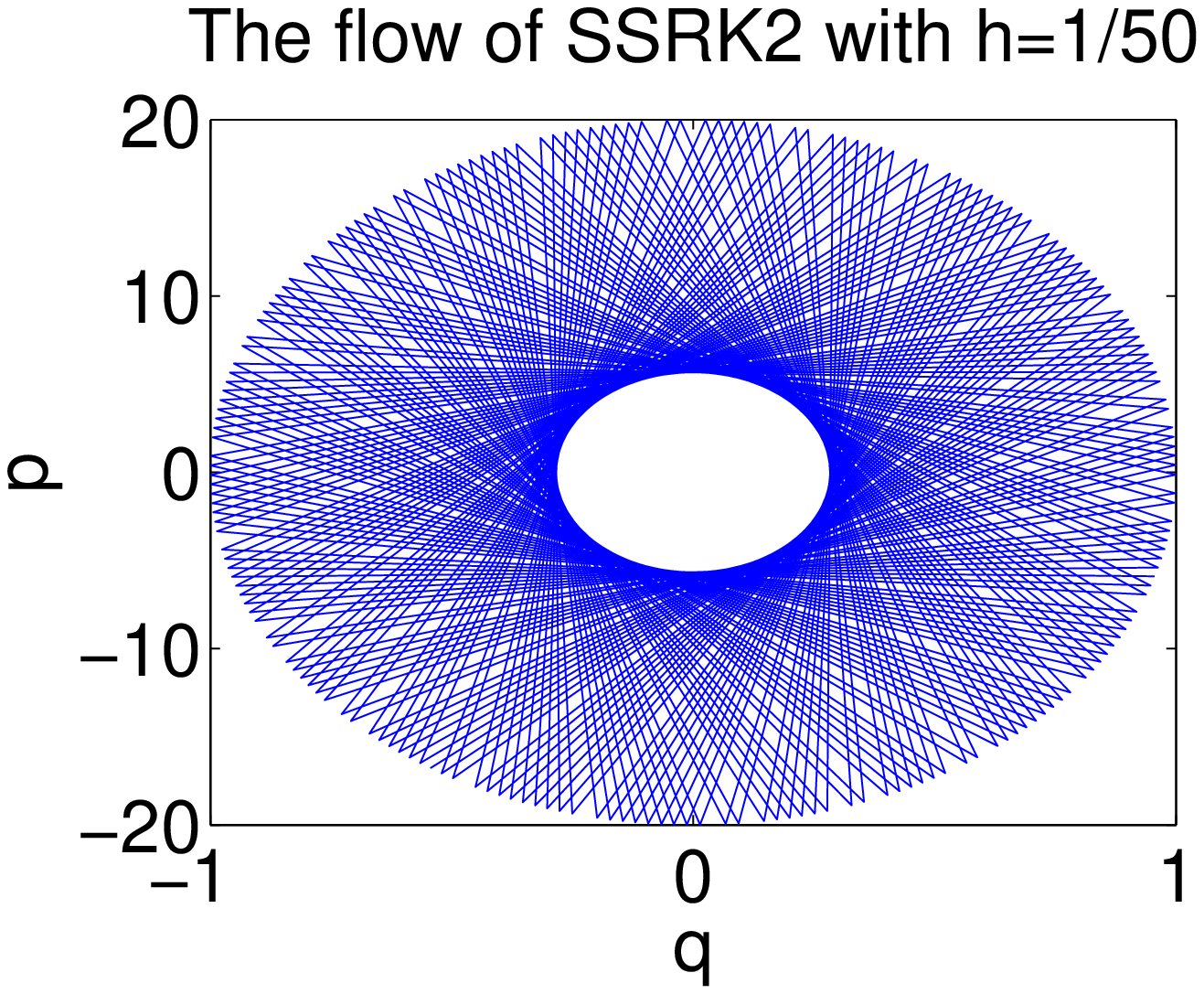}
\includegraphics[width=3.0cm,height=3.0cm]{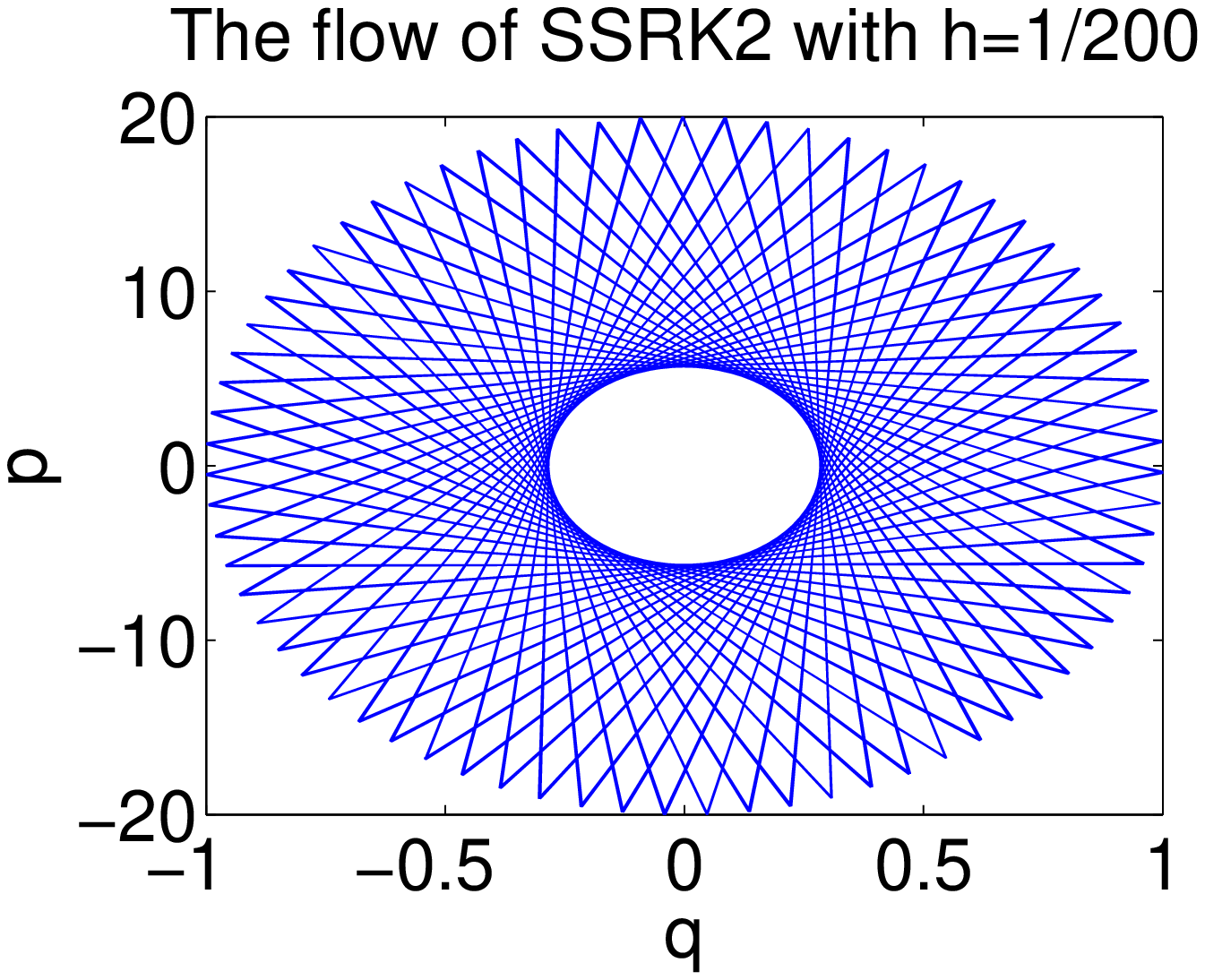}
\caption{Problem 1: the flows of different methods.} \label{p1}
\end{figure}
 \begin{figure}[ptb]
\centering
\includegraphics[width=3.8cm,height=4cm]{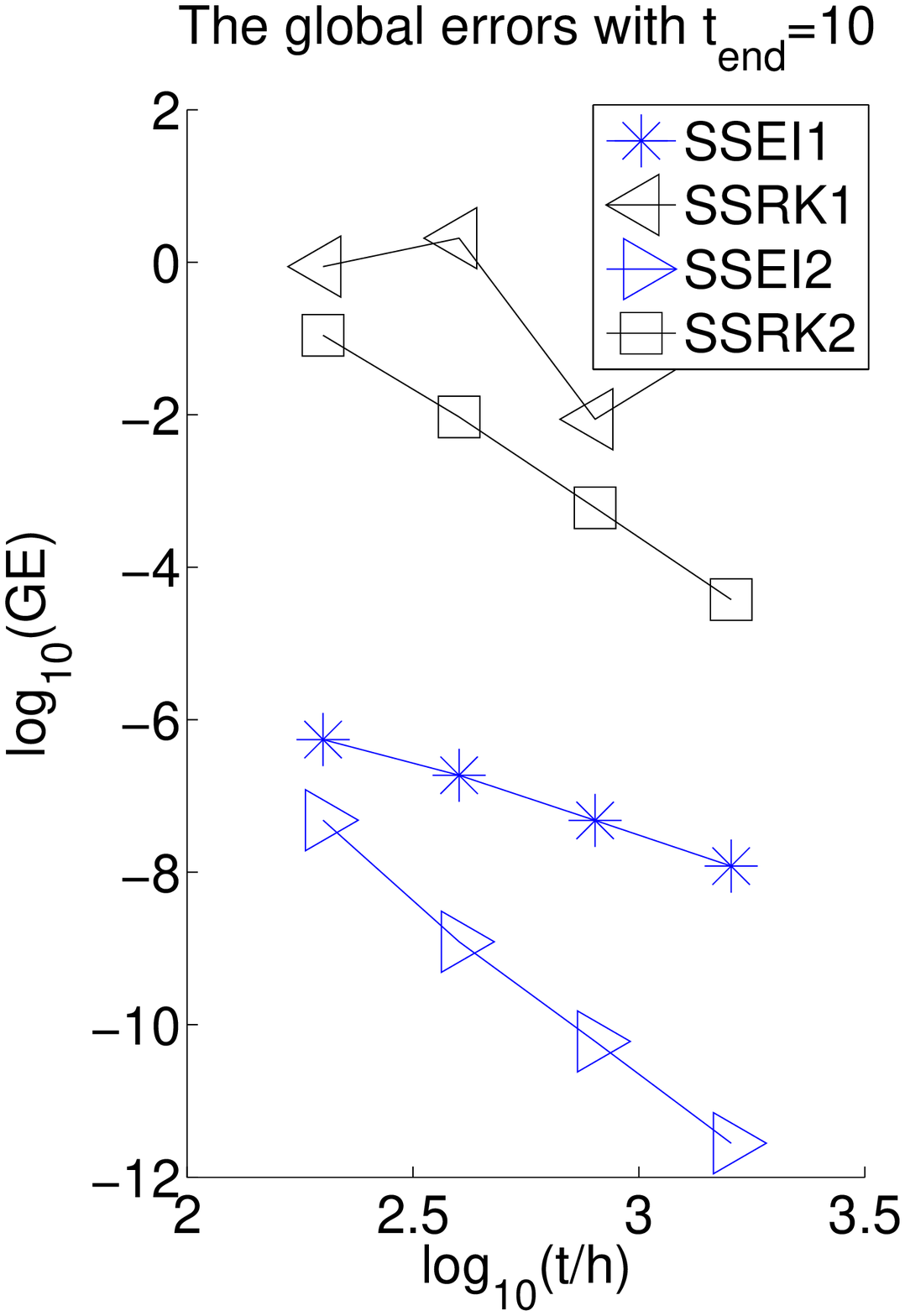}
\includegraphics[width=3.8cm,height=4cm]{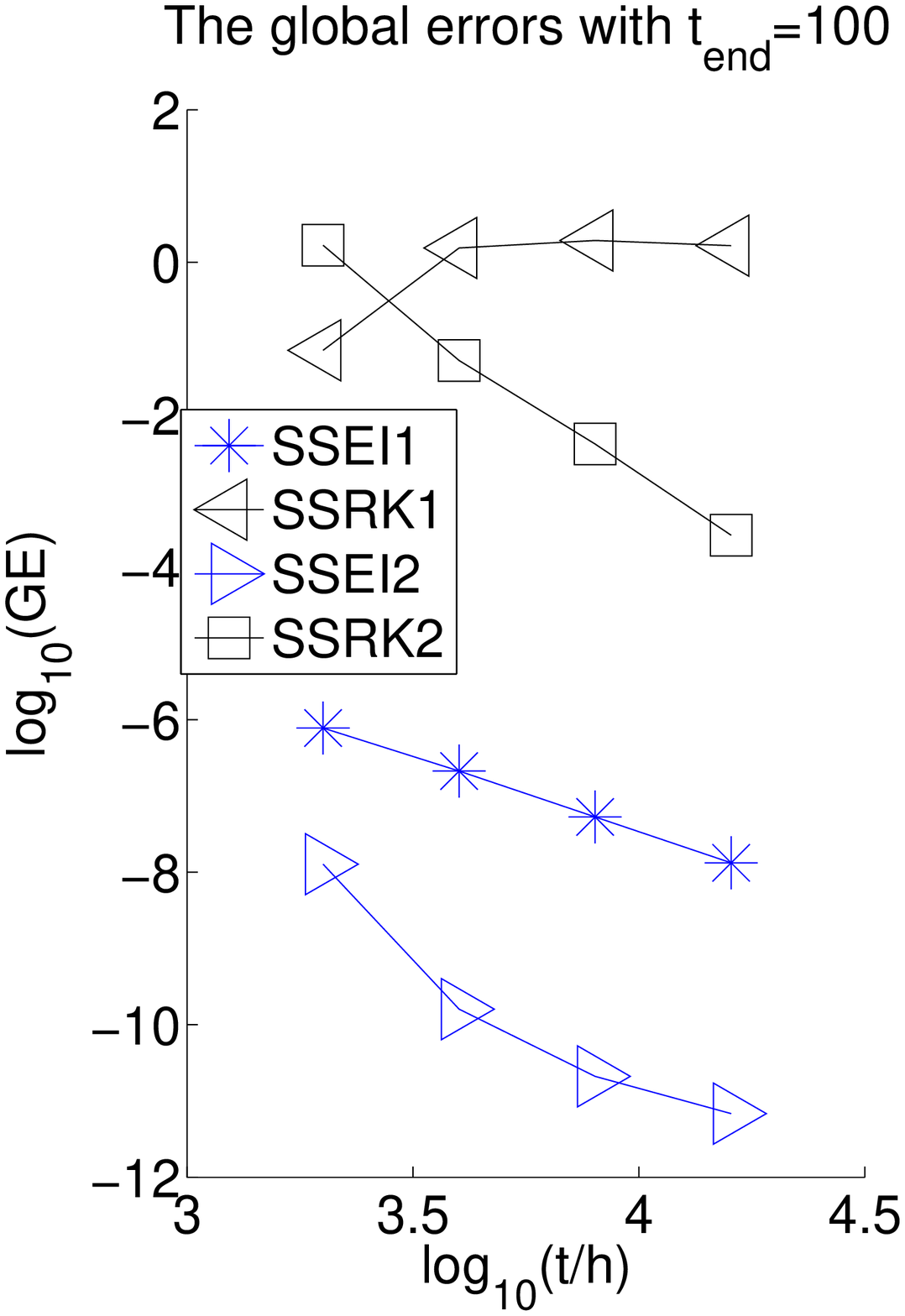}
\includegraphics[width=3.8cm,height=4cm]{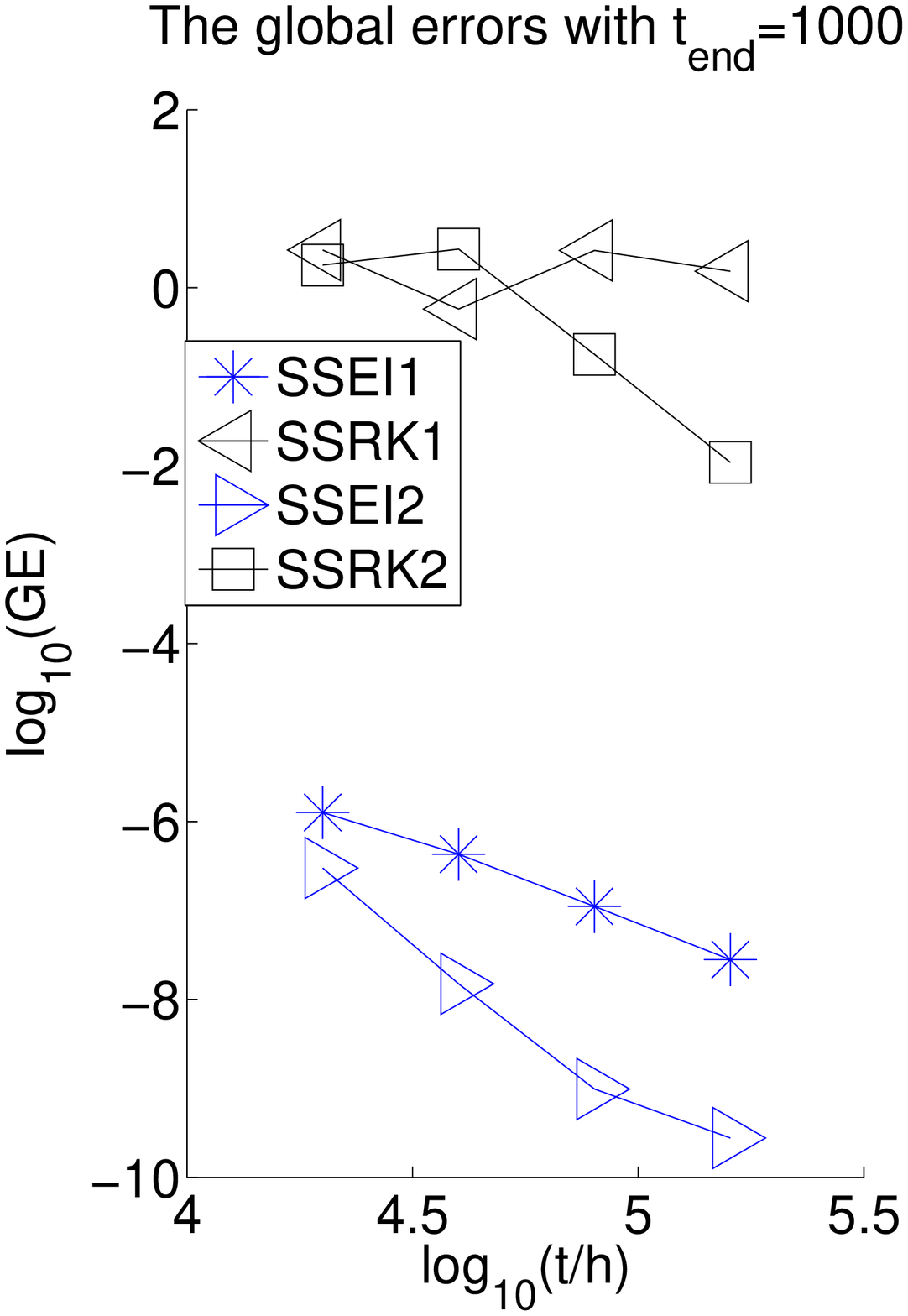}
\caption{Problem 1: the relative global errors.} \label{p2}
\end{figure}
\vskip2mm \noindent\textbf{Problem 2.} Consider the following
divergence free ODEs
$$\left(
    \begin{array}{c}
      x \\
      y \\
      z \\
    \end{array}
  \right)'=\left(
             \begin{array}{ccc}
               0 & -\omega & 0 \\
               \omega & 0 & -\omega \\
               0 & \omega & 0 \\
             \end{array}
           \right)\left(
    \begin{array}{c}
      x \\
      y \\
      z \\
    \end{array}
  \right)+\left(
            \begin{array}{c}
              \sin(x-z) \\
              0 \\
               \sin(x-z) \\
            \end{array}
          \right).
$$
By choosing $P=\left(
             \begin{array}{ccc}
               0 & 0 & 1 \\
               0 & 1 & 0 \\
              1 & 0 & 0 \\
             \end{array}
           \right),$
it can be checked that the vector field  of   this problem  falls
into $\mathcal{S}$. We consider $\omega=100$ and
$(0.5,0.5,0.5)^{\intercal}$ as the initial value. This problem is
firstly integrated  on $[0,100]$ with $h=1/50,1/100,1/200,1/400$ and
  the numerical flows $x$ and $y$   at the  time points
$\{\frac{1}{2}i\}_{i=1,\ldots,200}$ are shown in  Figure \ref{p2-1}.
Then the relative global errors for different $t_{\textmd{end}}$
with $h= 0.1/2^{i}$ for $i=2,\ldots,5$ are given in Figure
\ref{p2-2}. These results demonstrate clearly again that SSEI
methods perform better than SSRK methods.
 \begin{figure}[ptb]
\centering
\includegraphics[width=3.0cm,height=3.0cm]{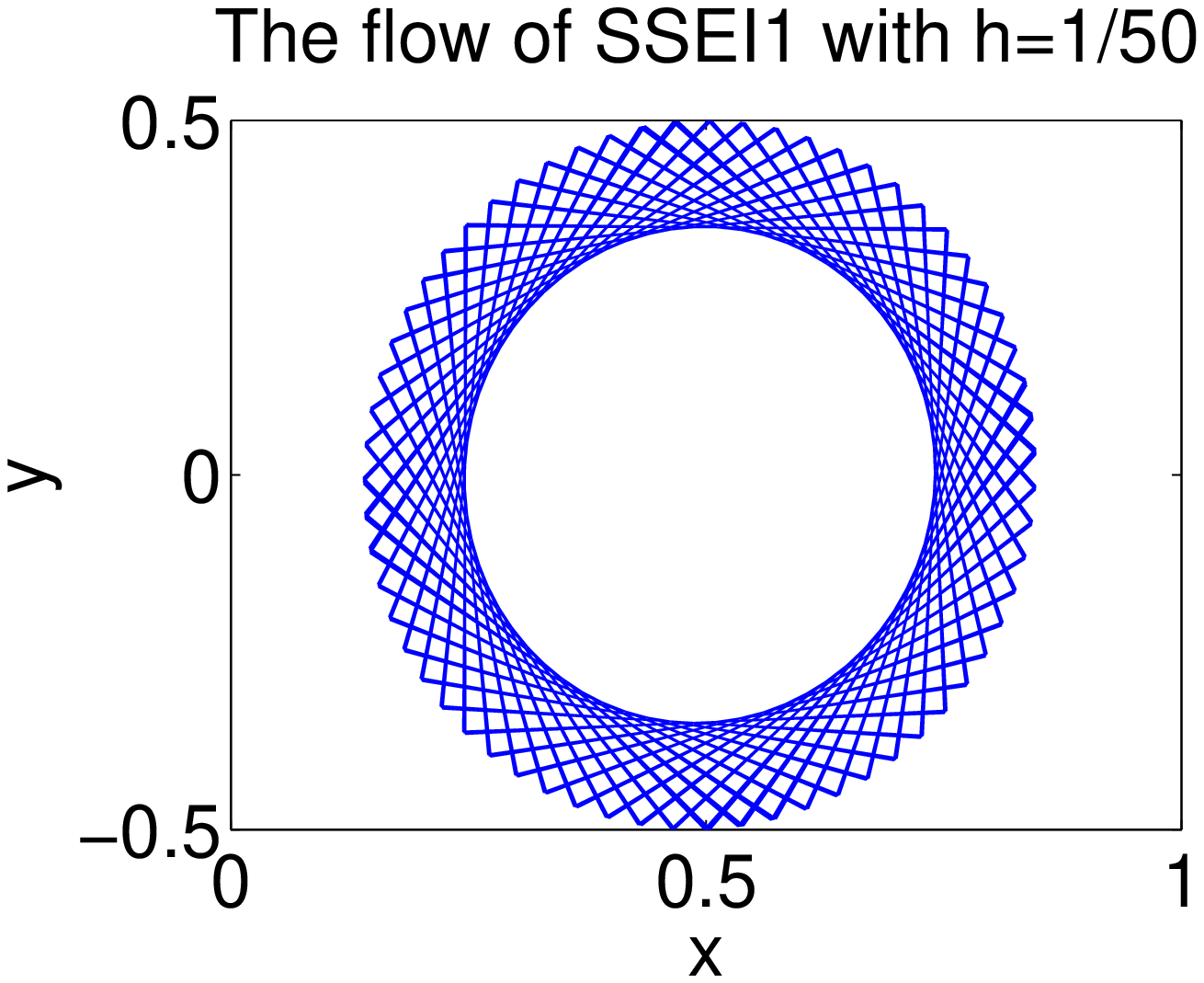}
\includegraphics[width=3.0cm,height=3.0cm]{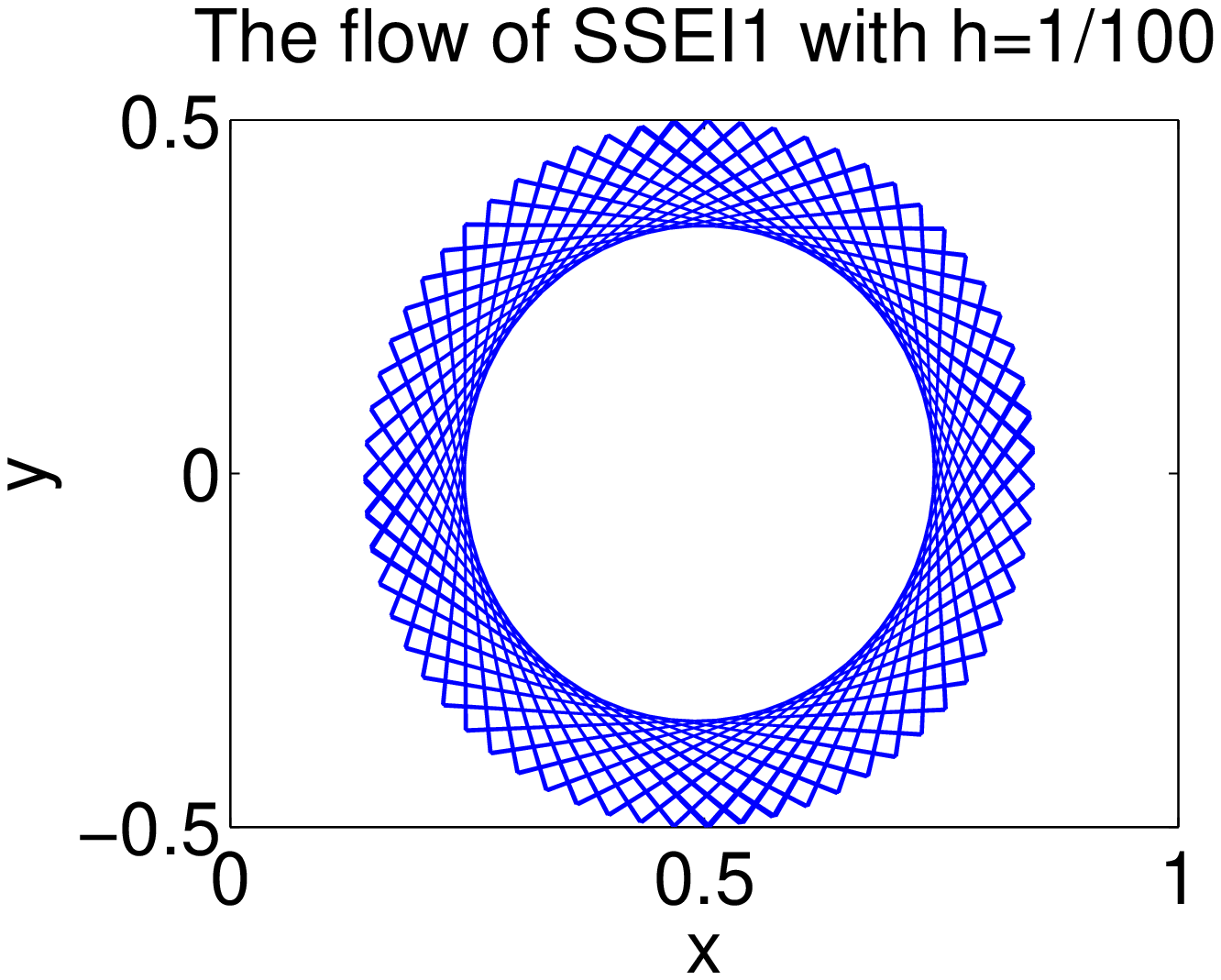}
\includegraphics[width=3.0cm,height=3.0cm]{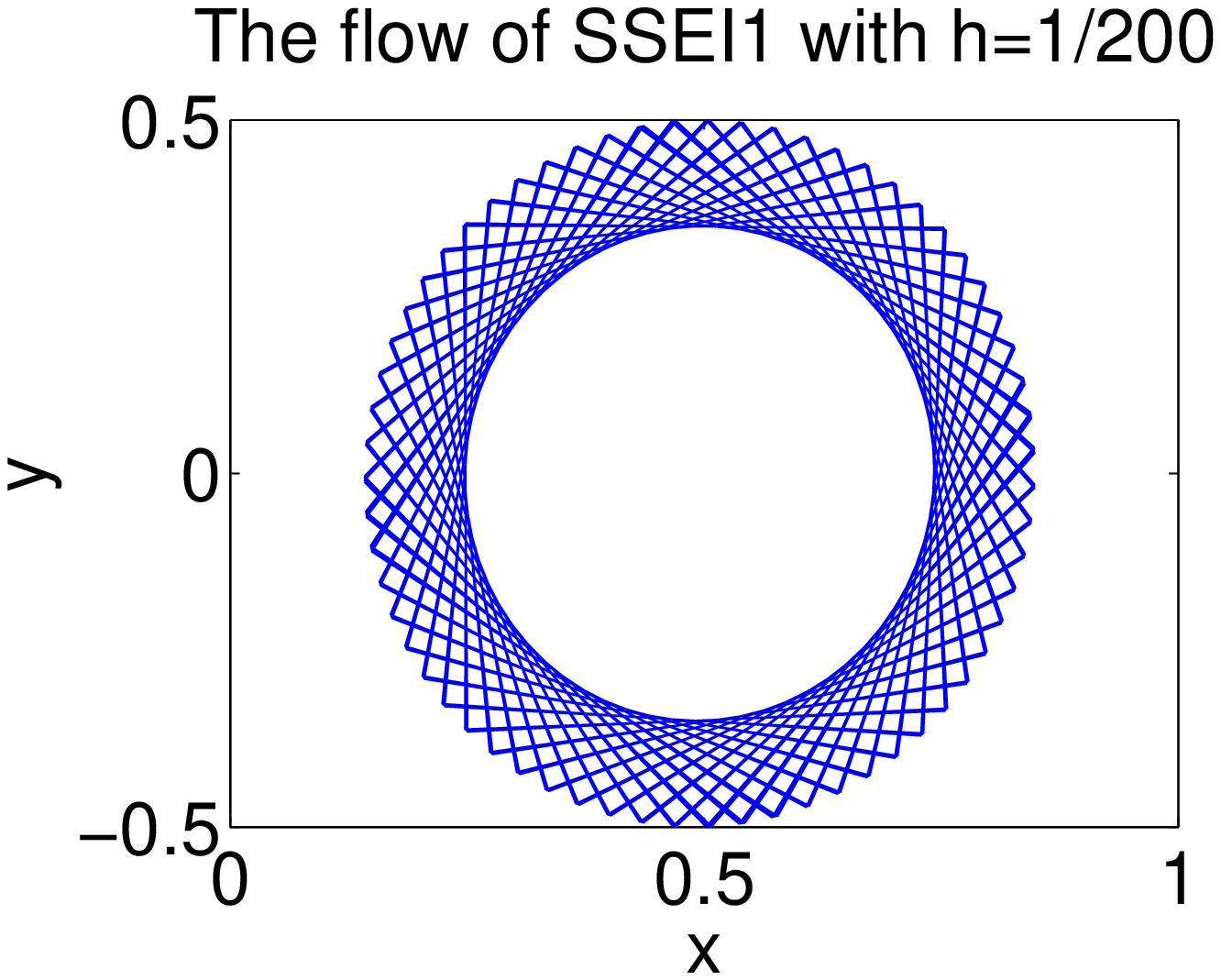}
\includegraphics[width=3.0cm,height=3.0cm]{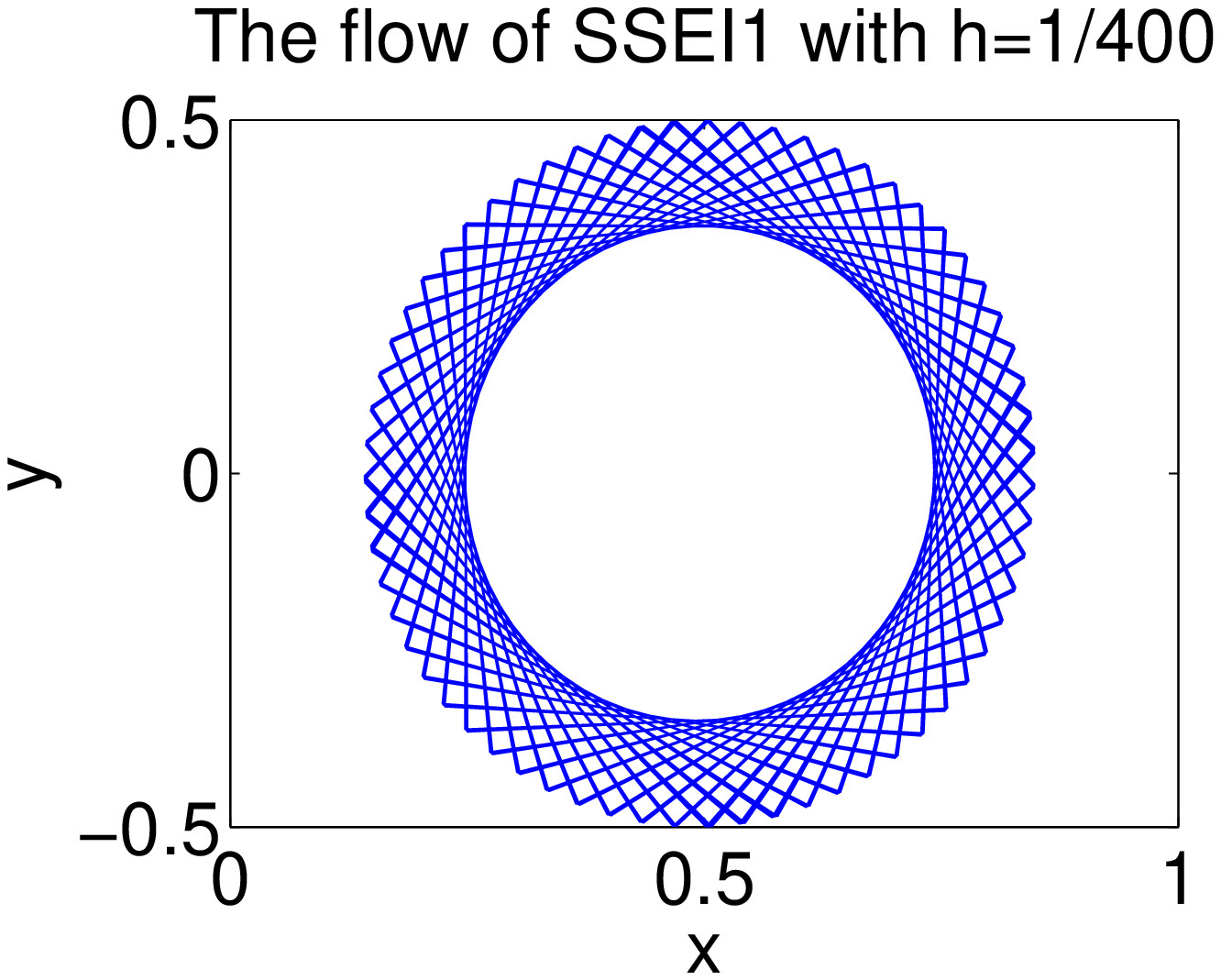}\\
\includegraphics[width=3.0cm,height=3.0cm]{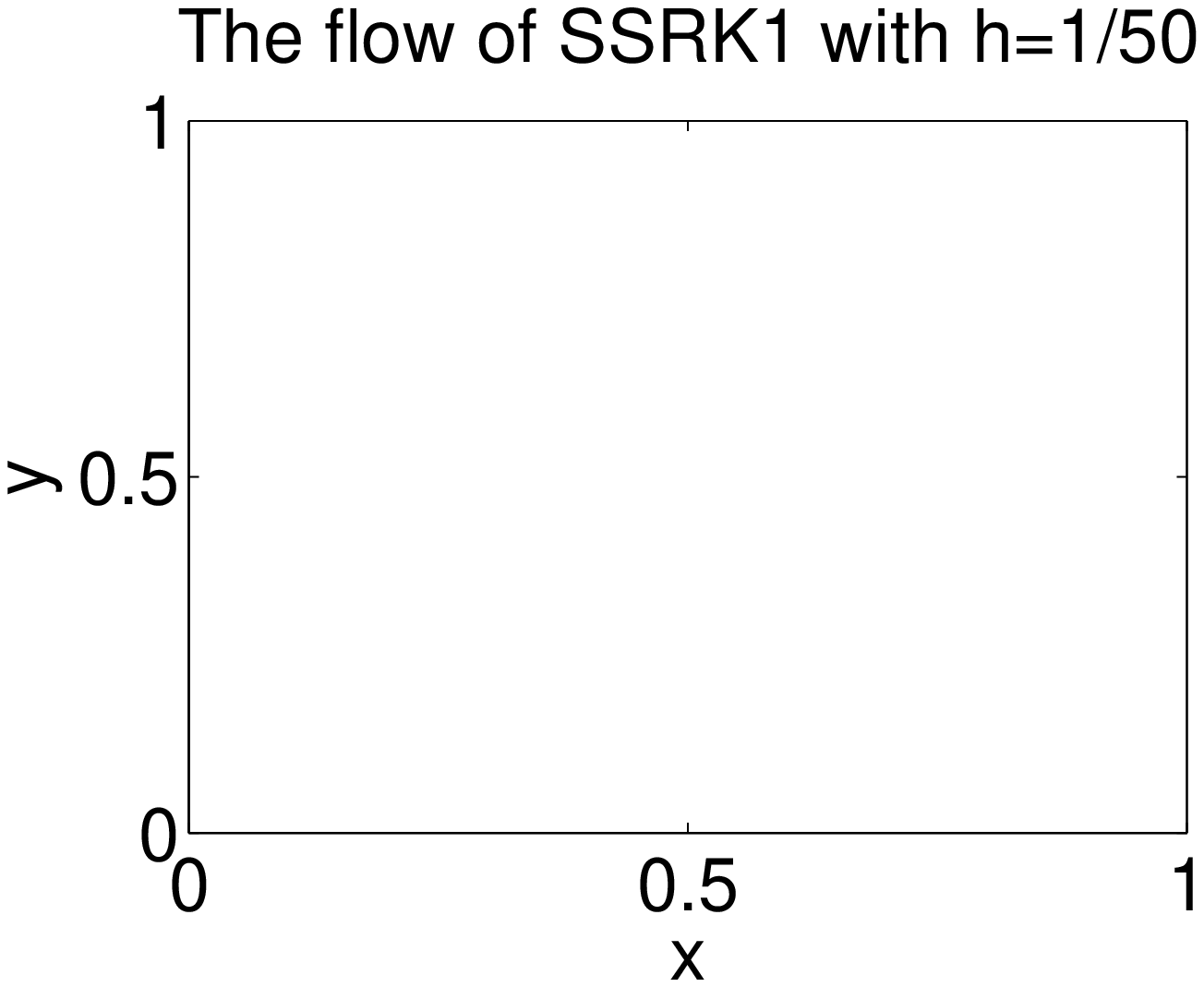}
\includegraphics[width=3.0cm,height=3.0cm]{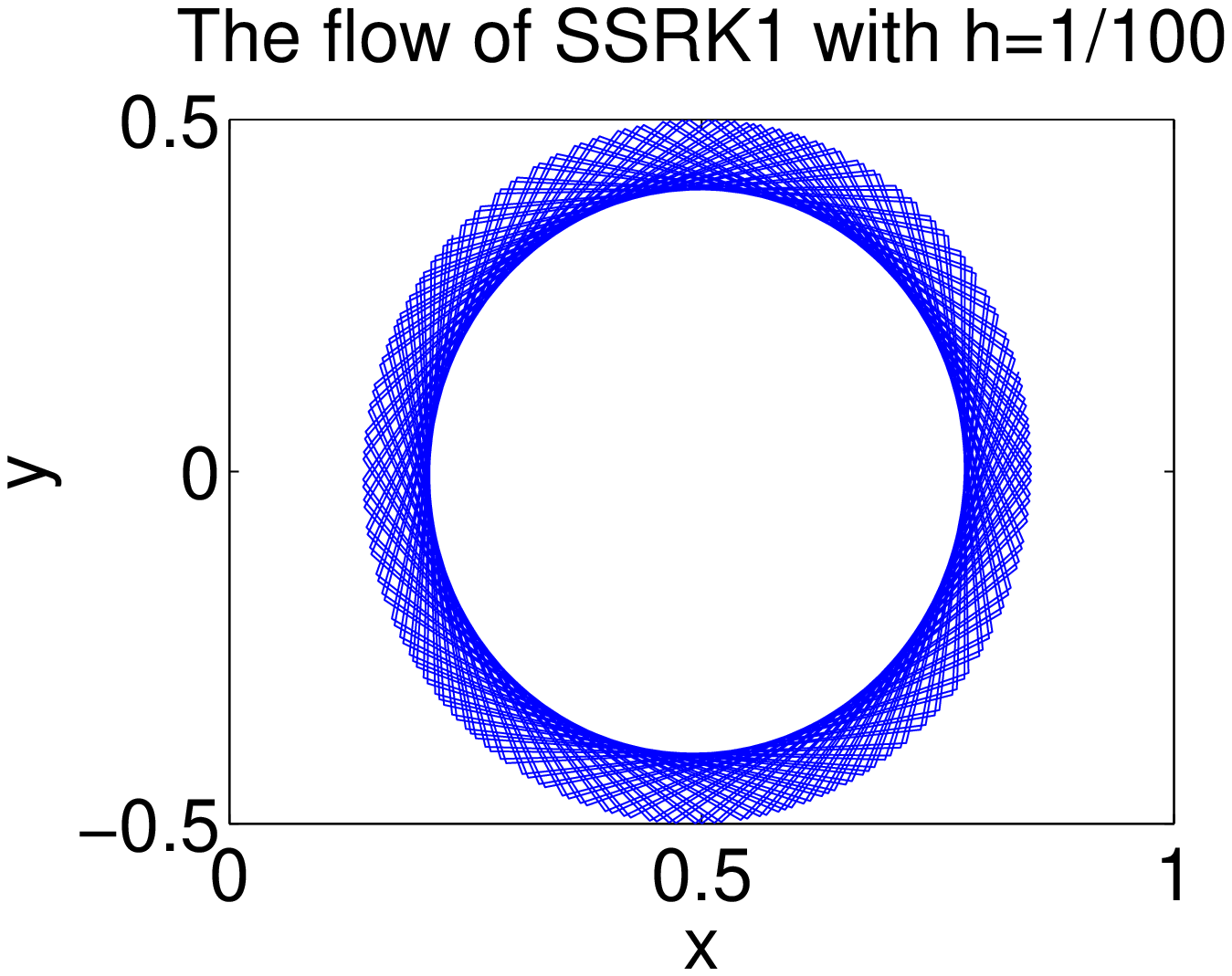}
\includegraphics[width=3.0cm,height=3.0cm]{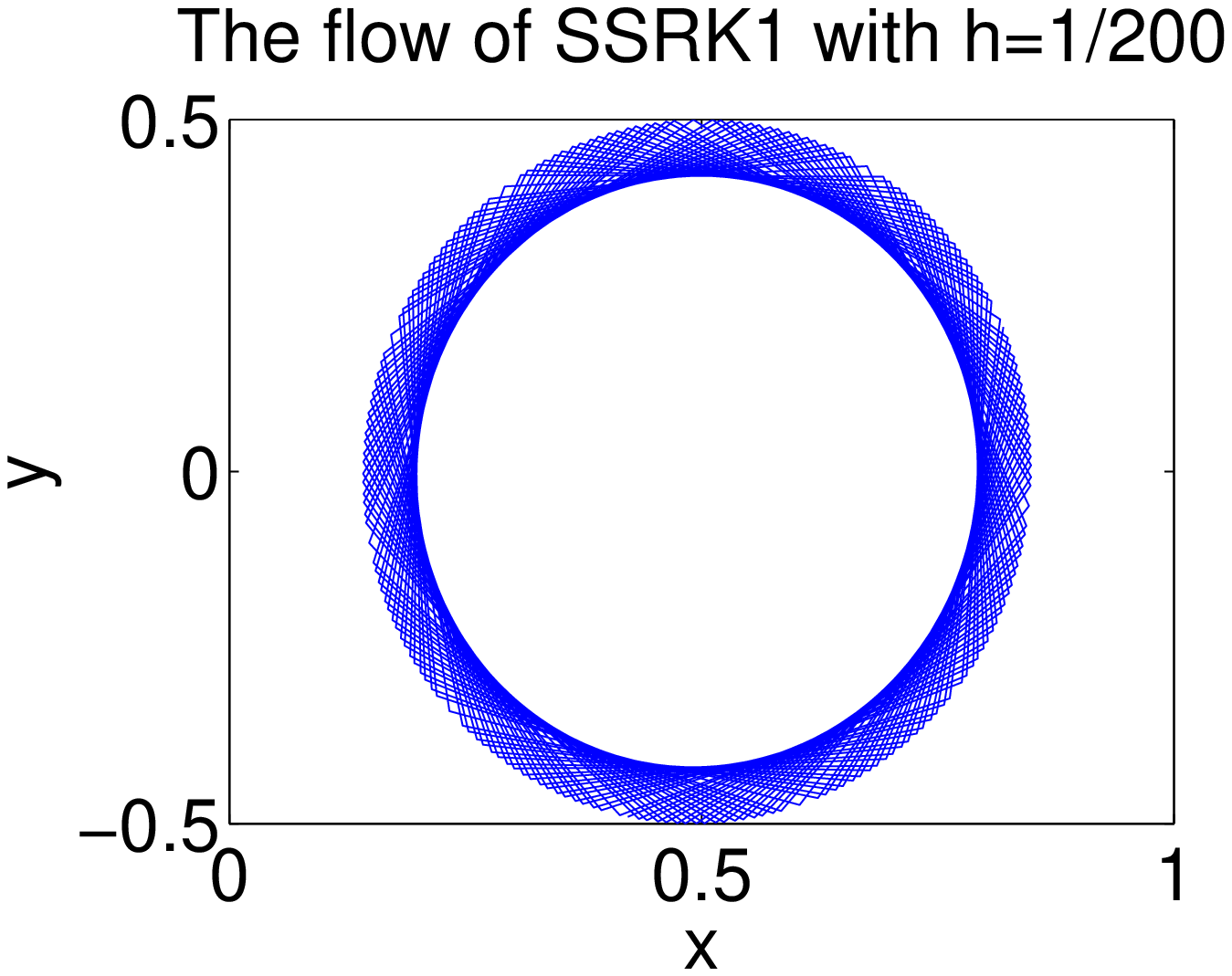}
\includegraphics[width=3.0cm,height=3.0cm]{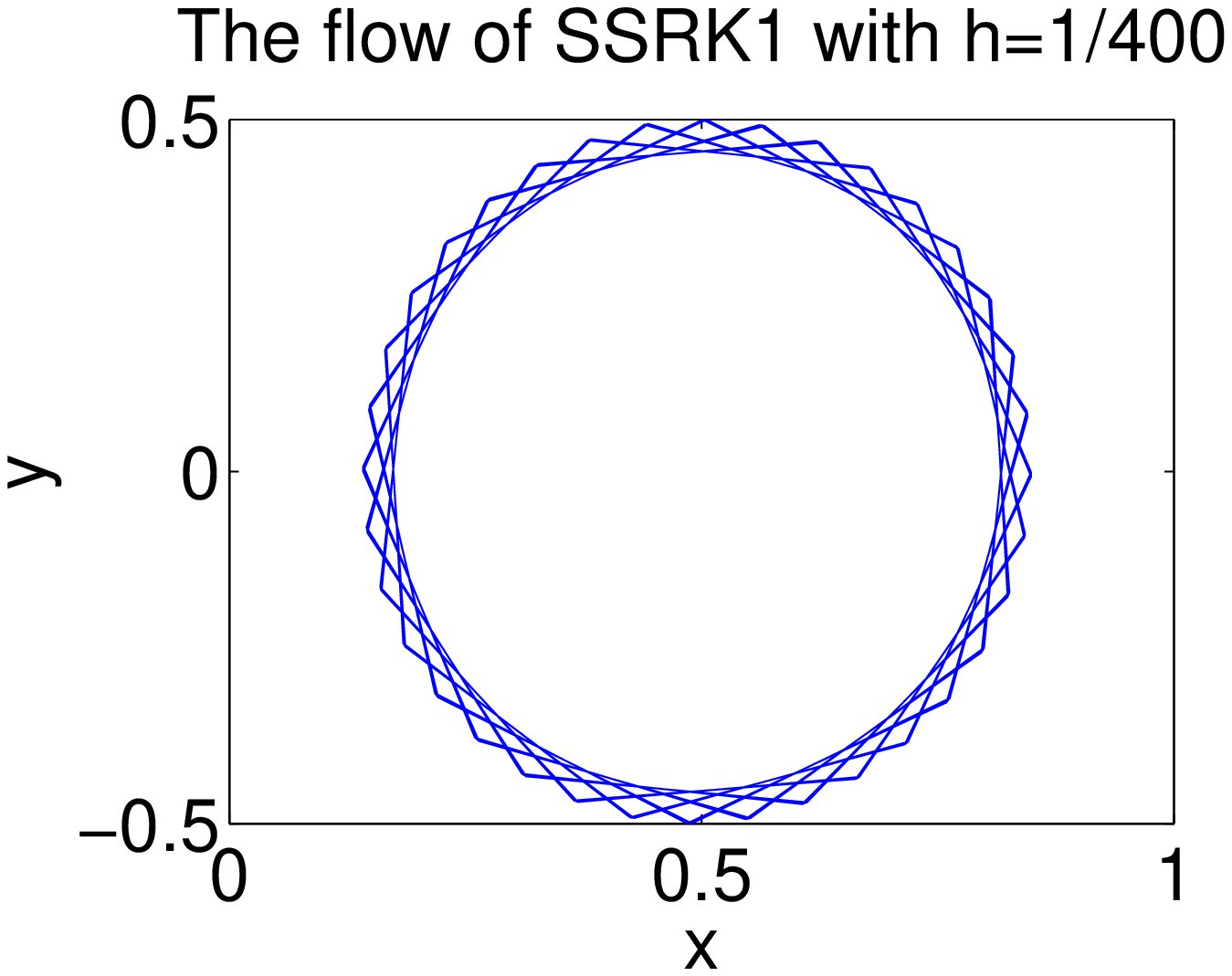}\\
\includegraphics[width=3.0cm,height=3.0cm]{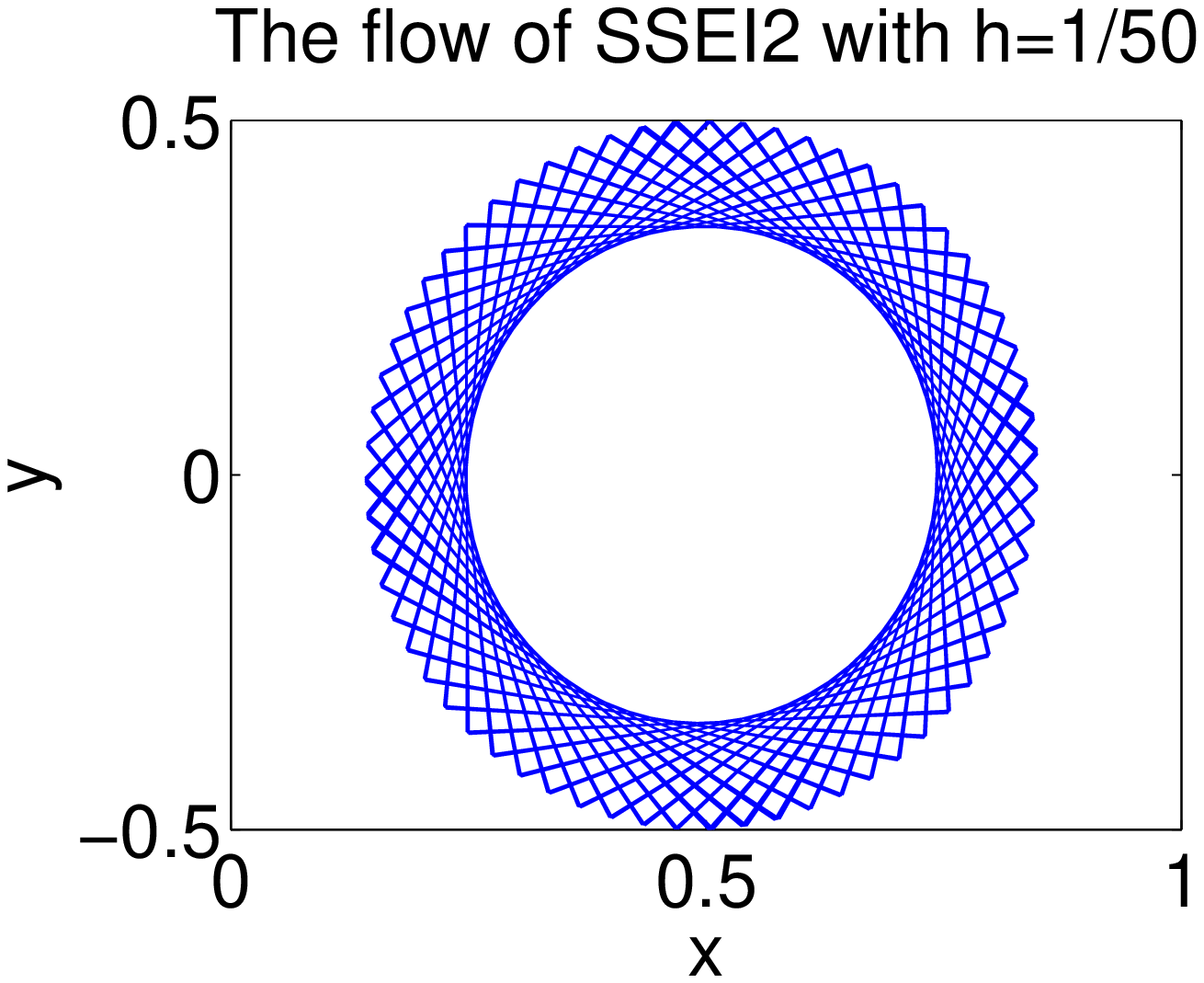}
\includegraphics[width=3.0cm,height=3.0cm]{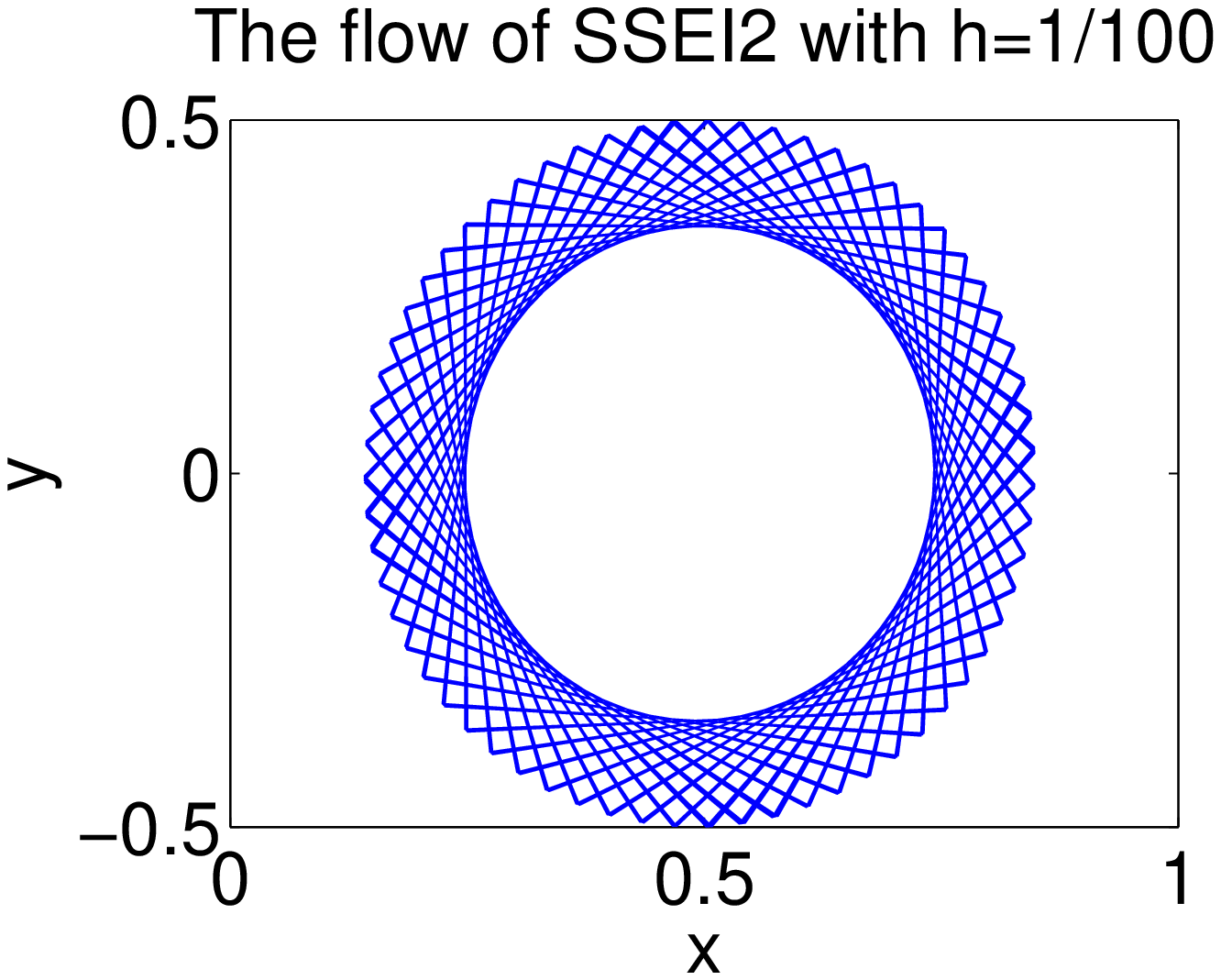}
\includegraphics[width=3.0cm,height=3.0cm]{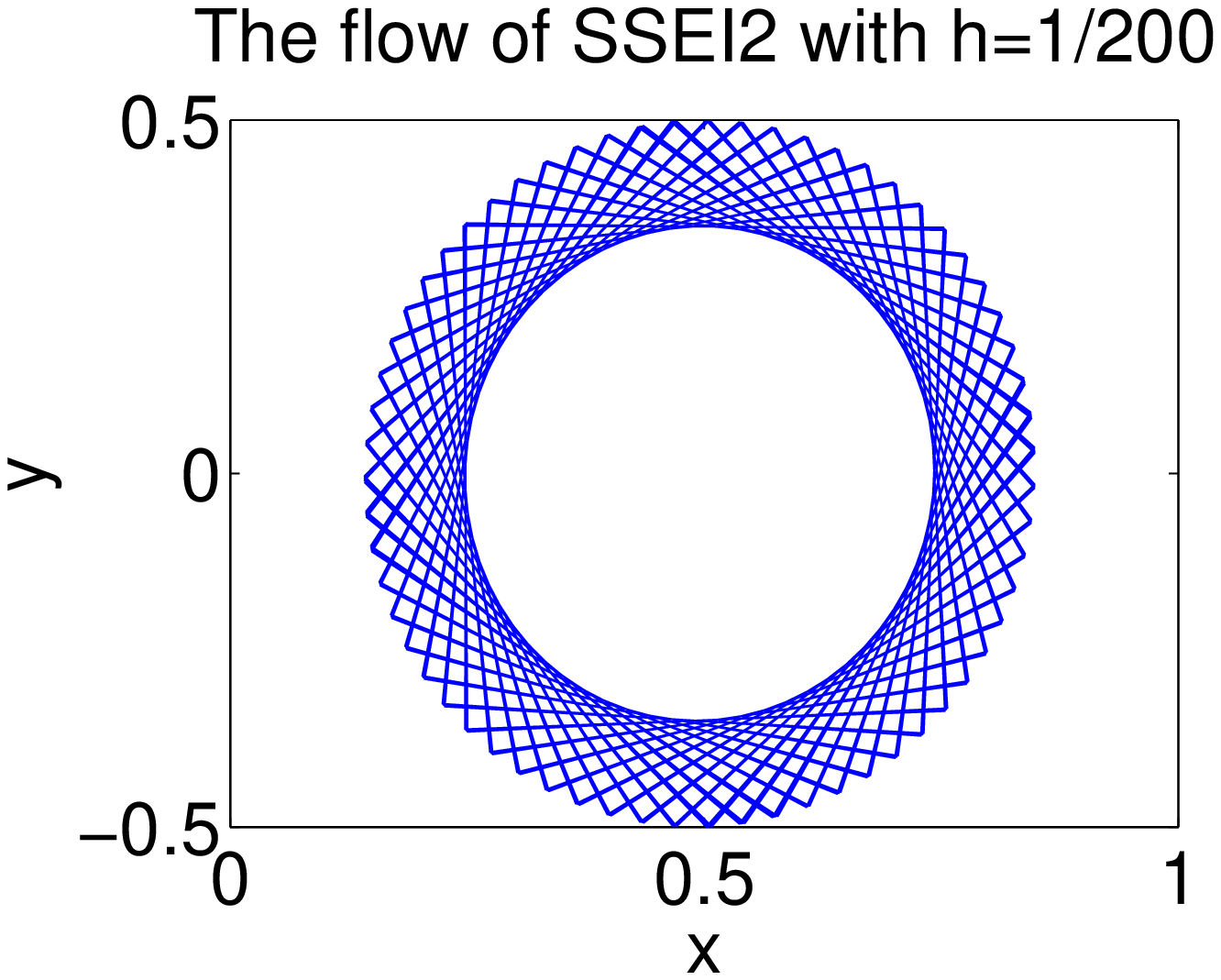}
\includegraphics[width=3.0cm,height=3.0cm]{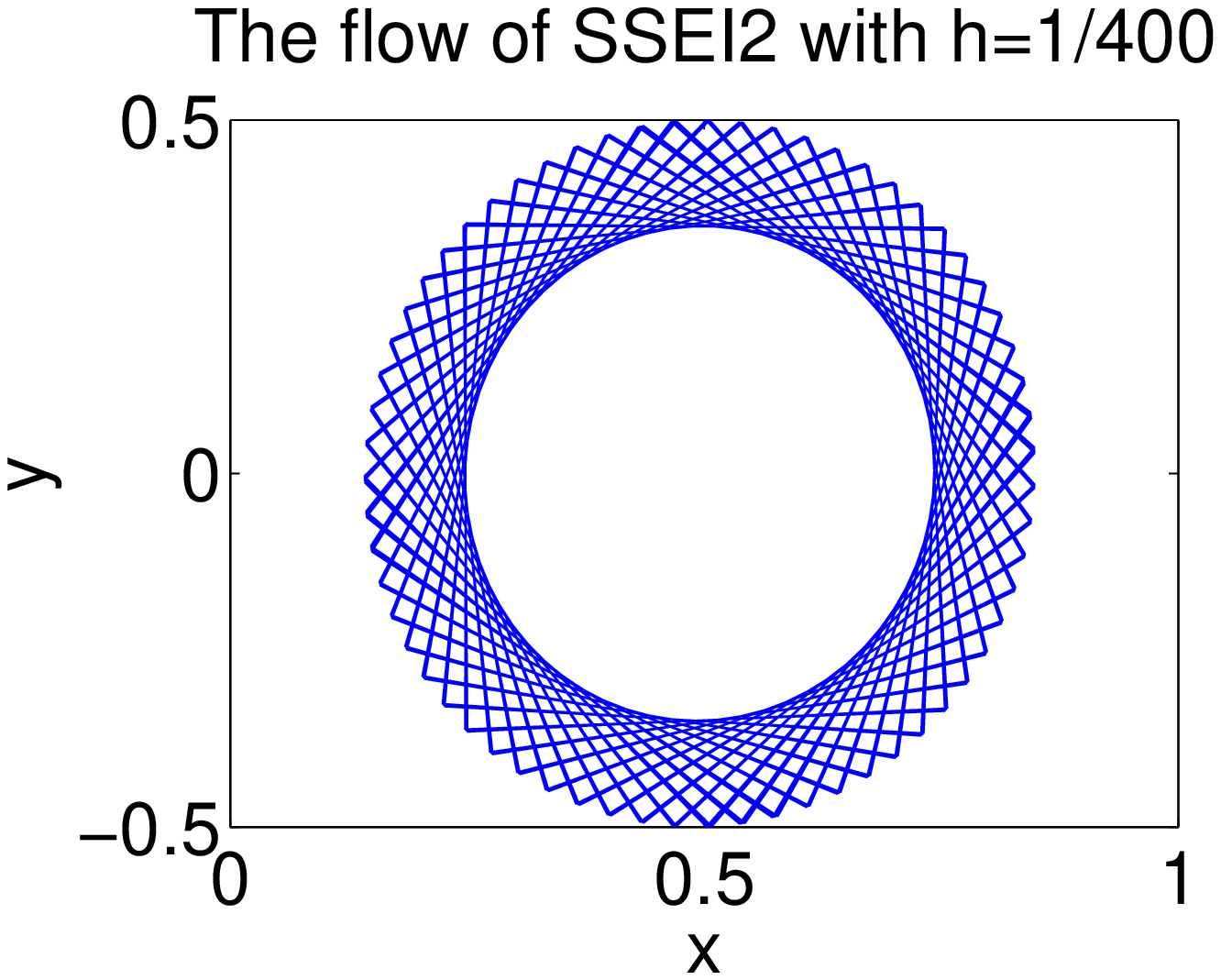}\\
\includegraphics[width=3.0cm,height=3.0cm]{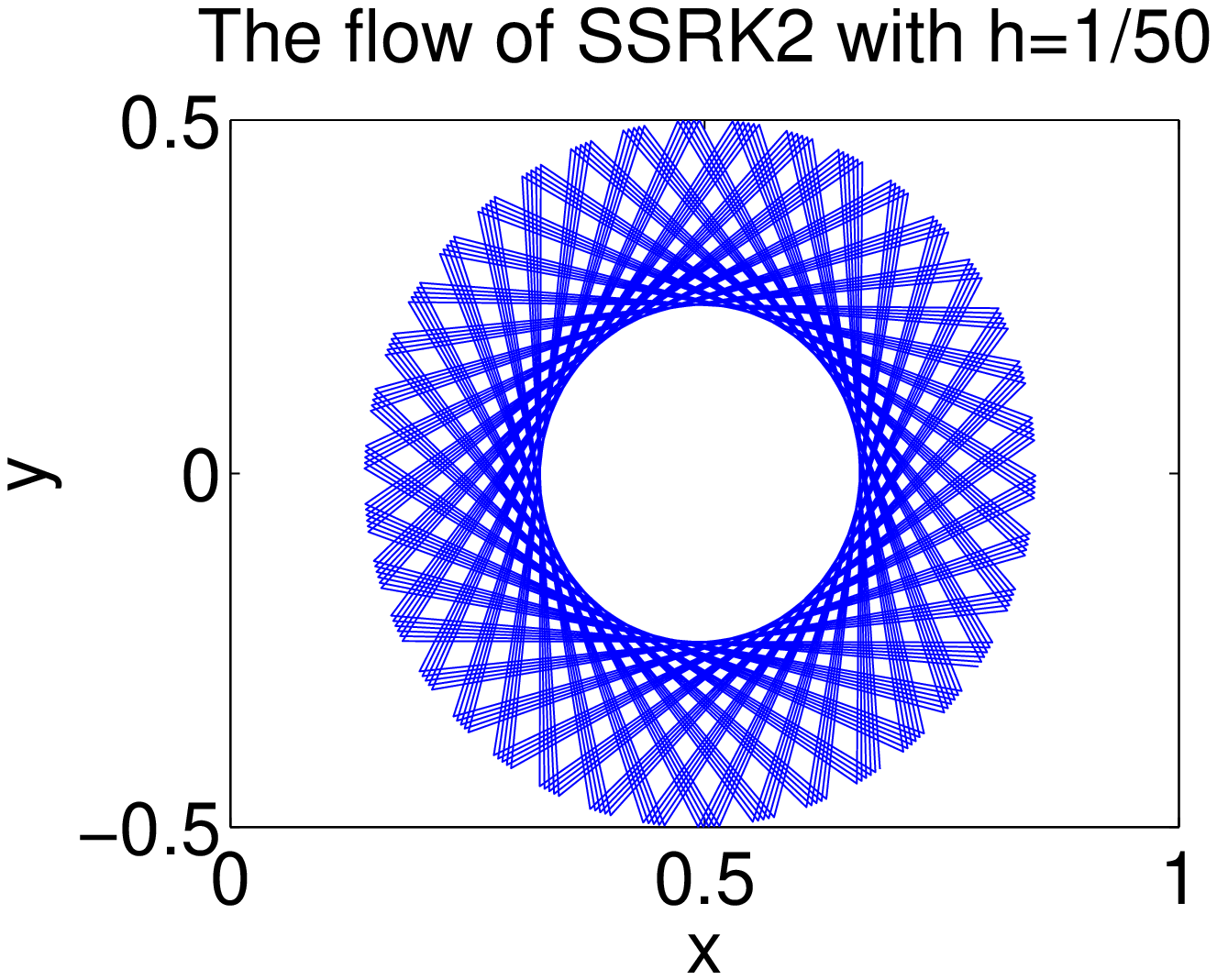}
\includegraphics[width=3.0cm,height=3.0cm]{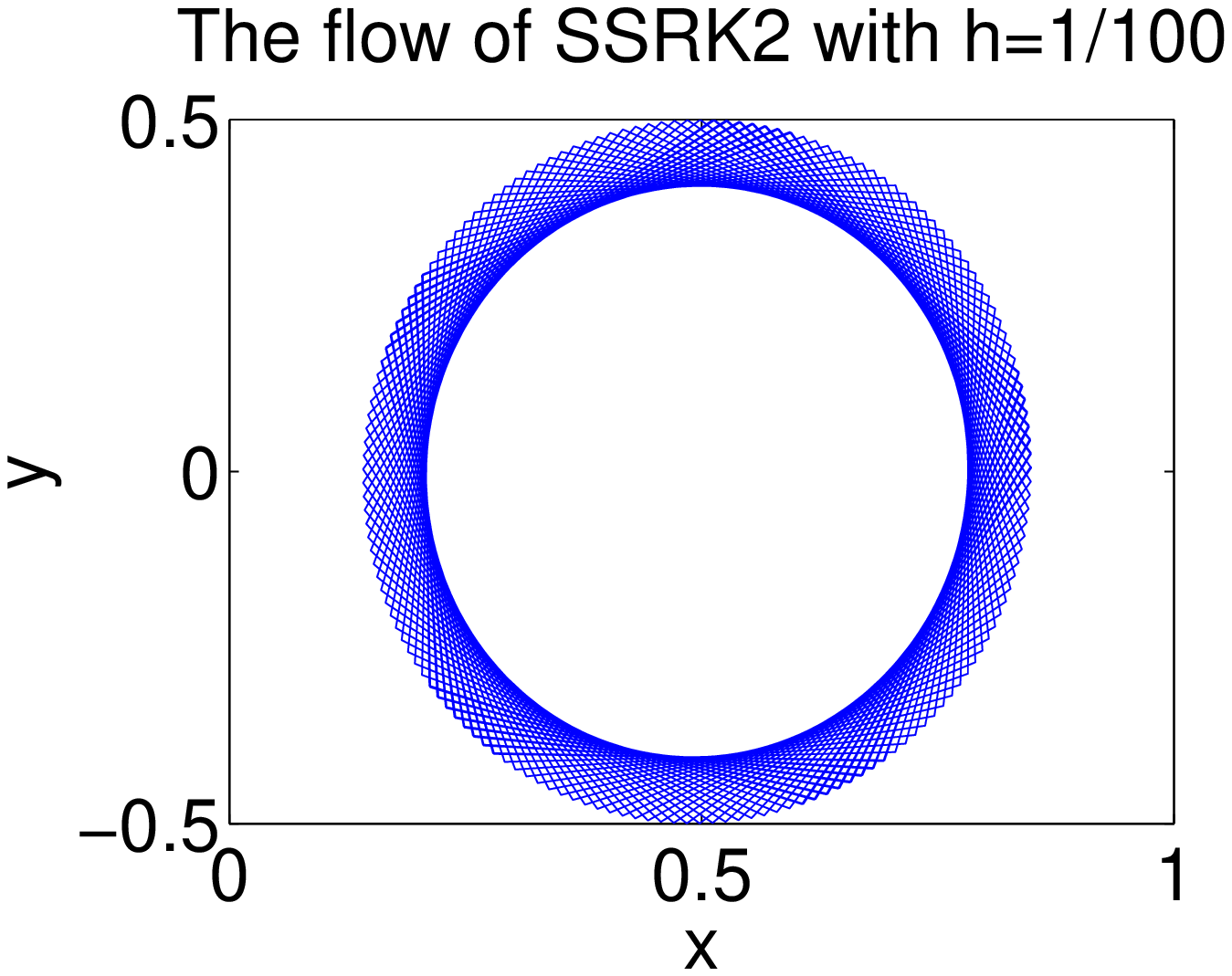}
\includegraphics[width=3.0cm,height=3.0cm]{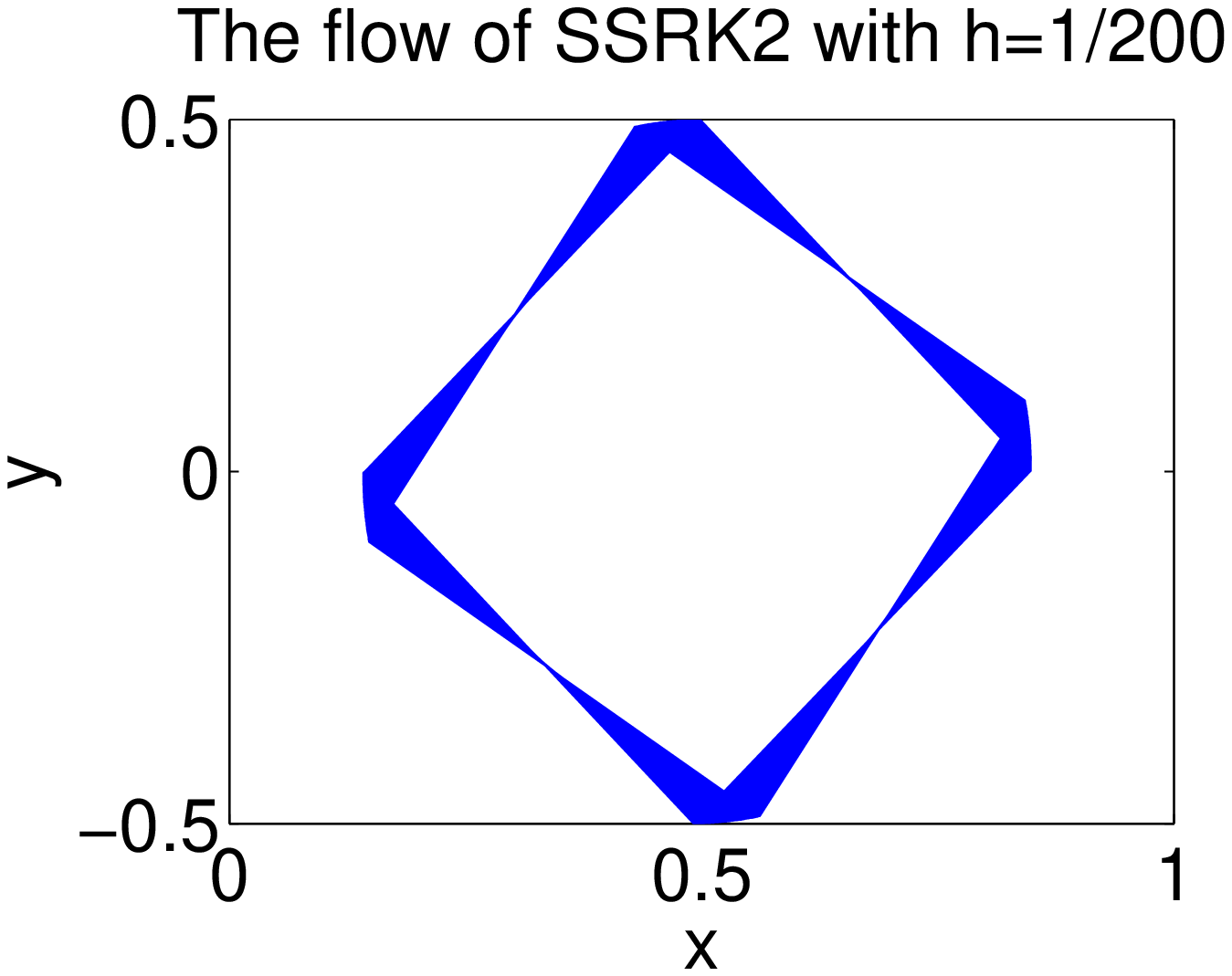}
\includegraphics[width=3.0cm,height=3.0cm]{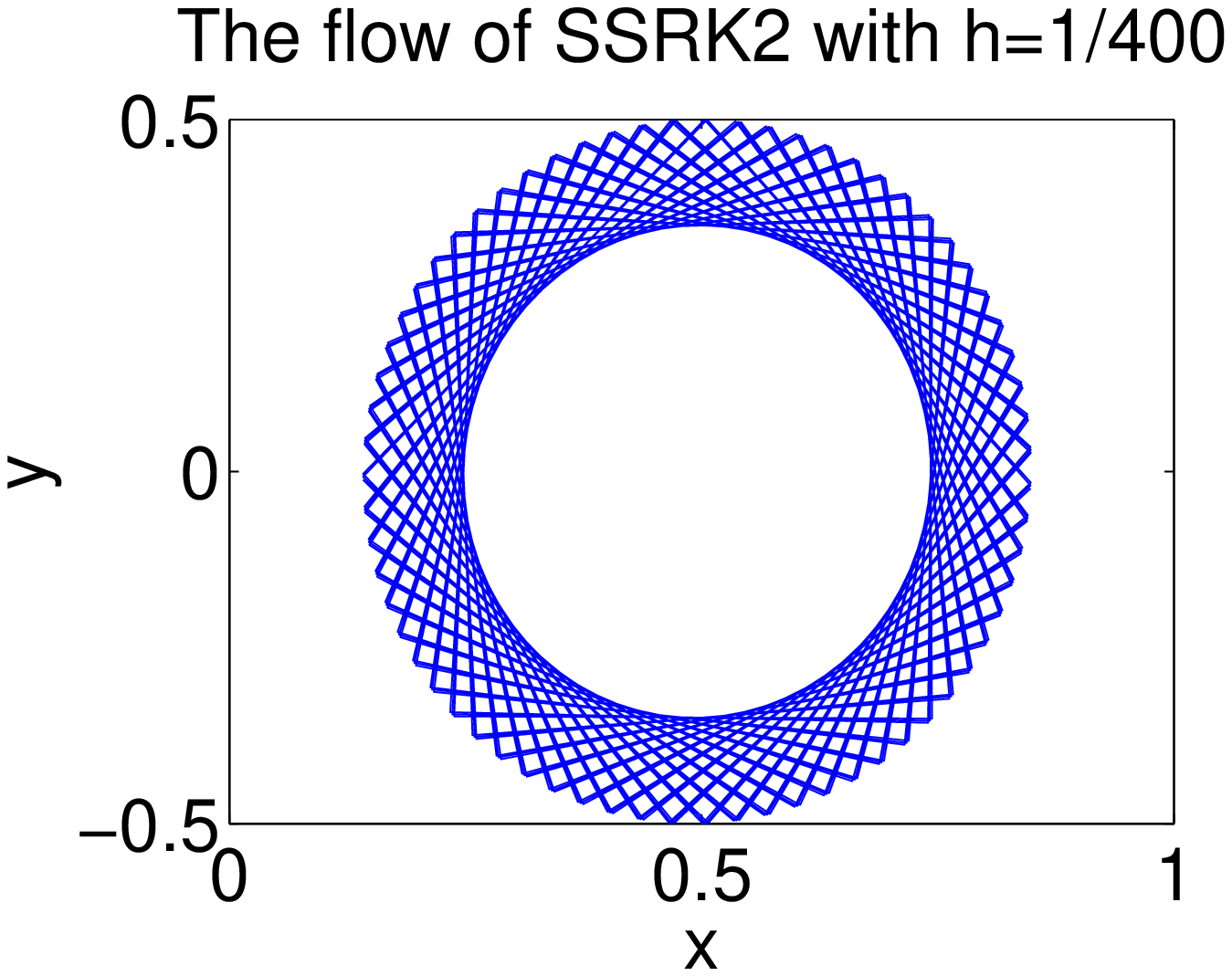}
\caption{Problem 2: the flows of different methods.} \label{p2-1}
\end{figure}
 \begin{figure}[ptb]
\centering
\includegraphics[width=3.8cm,height=4cm]{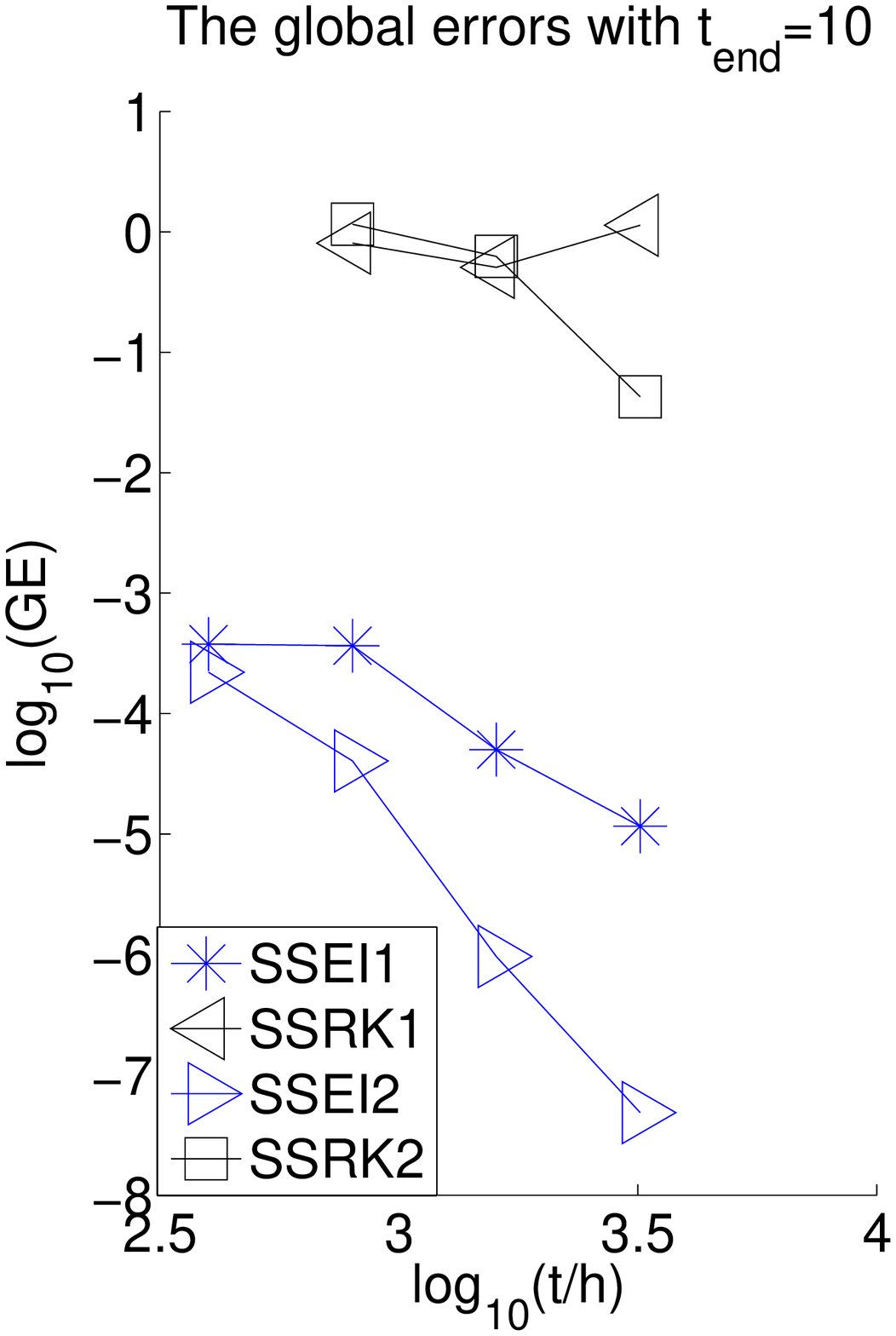}
\includegraphics[width=3.8cm,height=4cm]{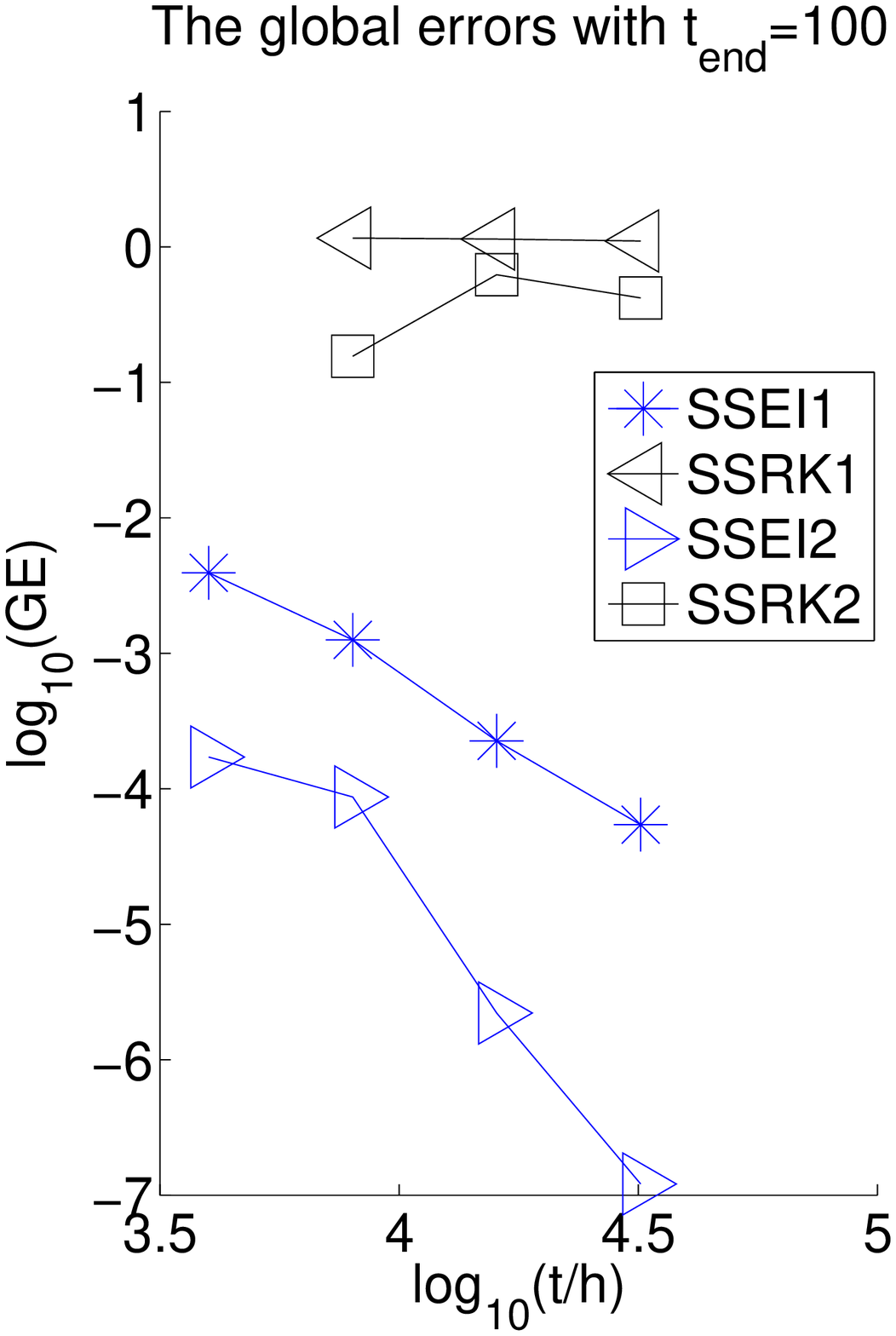}
\includegraphics[width=3.8cm,height=4cm]{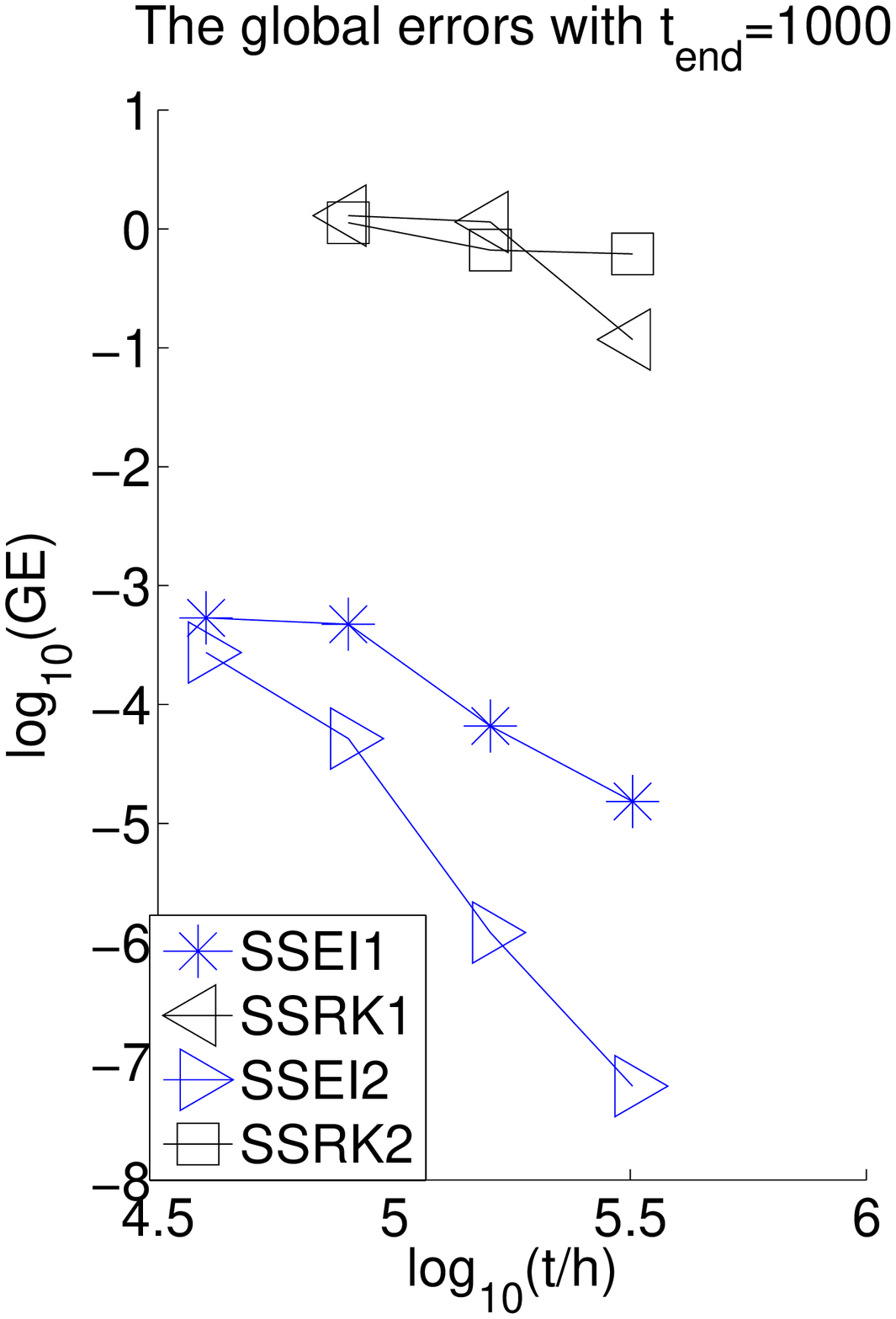}
\caption{Problem 2: the relative global errors.} \label{p2-2}
\end{figure}

 \vskip2mm\noindent\textbf{Problem 3.} Consider the damped Helmholtz-Duffing
oscillator (see \cite{AML-2012})
$$q''+2\upsilon q'+Aq=-Bq^2-\varepsilon q^3,$$ where
$q$ denotes the displacement of the system, $A$ is the natural
frequency, $\varepsilon$ is a non-linear system parameter,
$\upsilon$ is the damping factor, and $B$ is a system parameter
independent of time. It is  well known  that the dynamical behavior
of
 eardrum oscillations, elasto-magnetic suspensions, thin laminated plates, graded beams,
and other physical phenomena all  fall into this kind of equations.
We choose the parameters
$$\upsilon=0.01,\ A=200,\ B=-0.5,\ \varepsilon=1$$
and the initial value $q(0)=1$ and $q'(0)=15.199$. This problem is
firstly integrated  on  $[0,200]$ with $h=1/2,1/10,1/50,1/200$.
  We present the numerical flows $q$ and $p=q'$   at the  time points
$\{\frac{1}{2}i\}_{i=1,\ldots,400}$ in  Figure \ref{p3-1}.  We then
solve the problem with different $t_{\textmd{end}}=10,100, 1000$ and
$h= 0.1/2^{i}$ for $i=0,\ldots,3$.  The relative global errors are
shown in Figure \ref{p3-2}.  It follows again from the results  that
SSEI methods perform  much better  than SSRK methods.
 \begin{figure}[ptb]
\centering
\includegraphics[width=3.0cm,height=3.0cm]{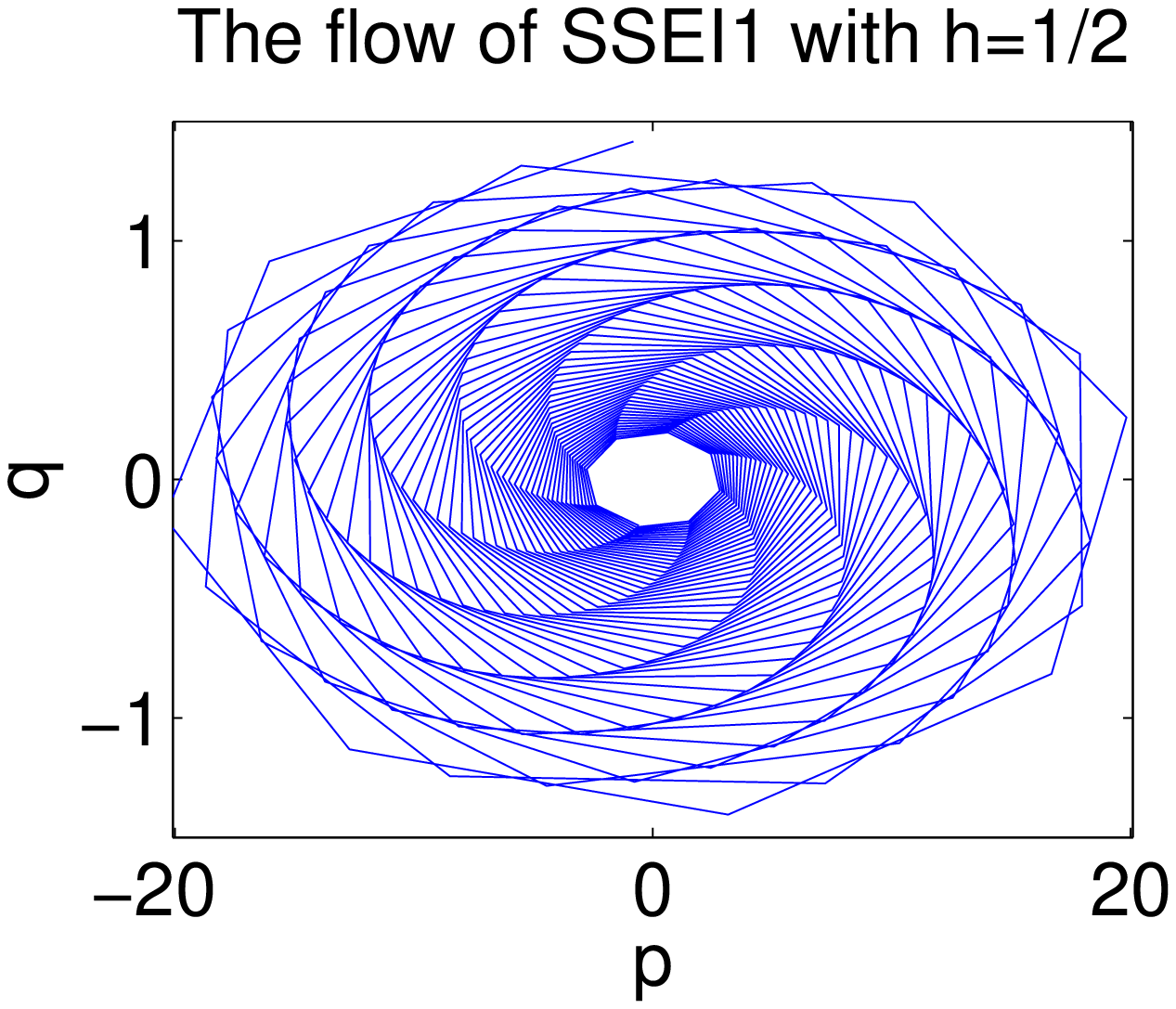}
\includegraphics[width=3.0cm,height=3.0cm]{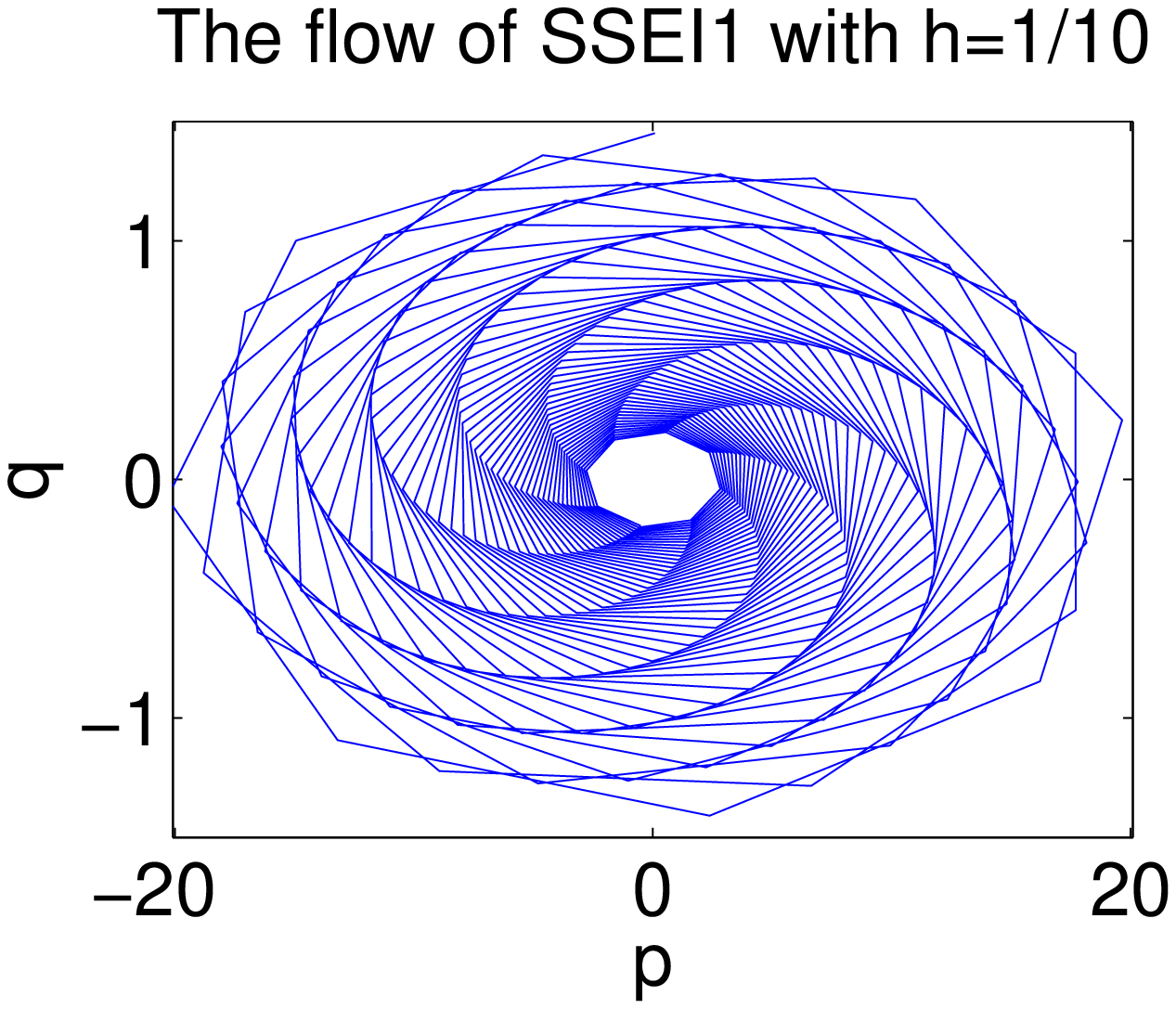}
\includegraphics[width=3.0cm,height=3.0cm]{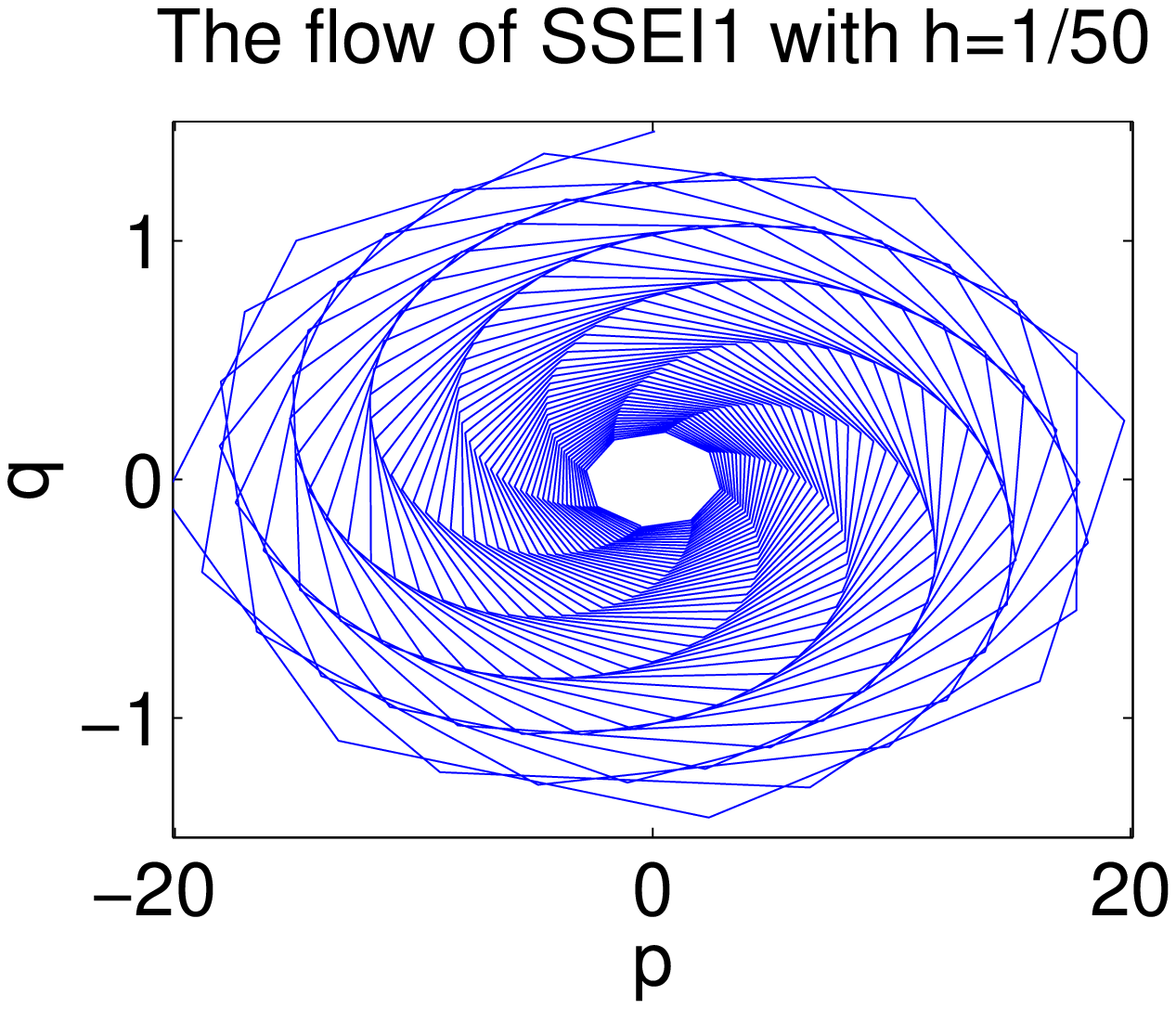}
\includegraphics[width=3.0cm,height=3.0cm]{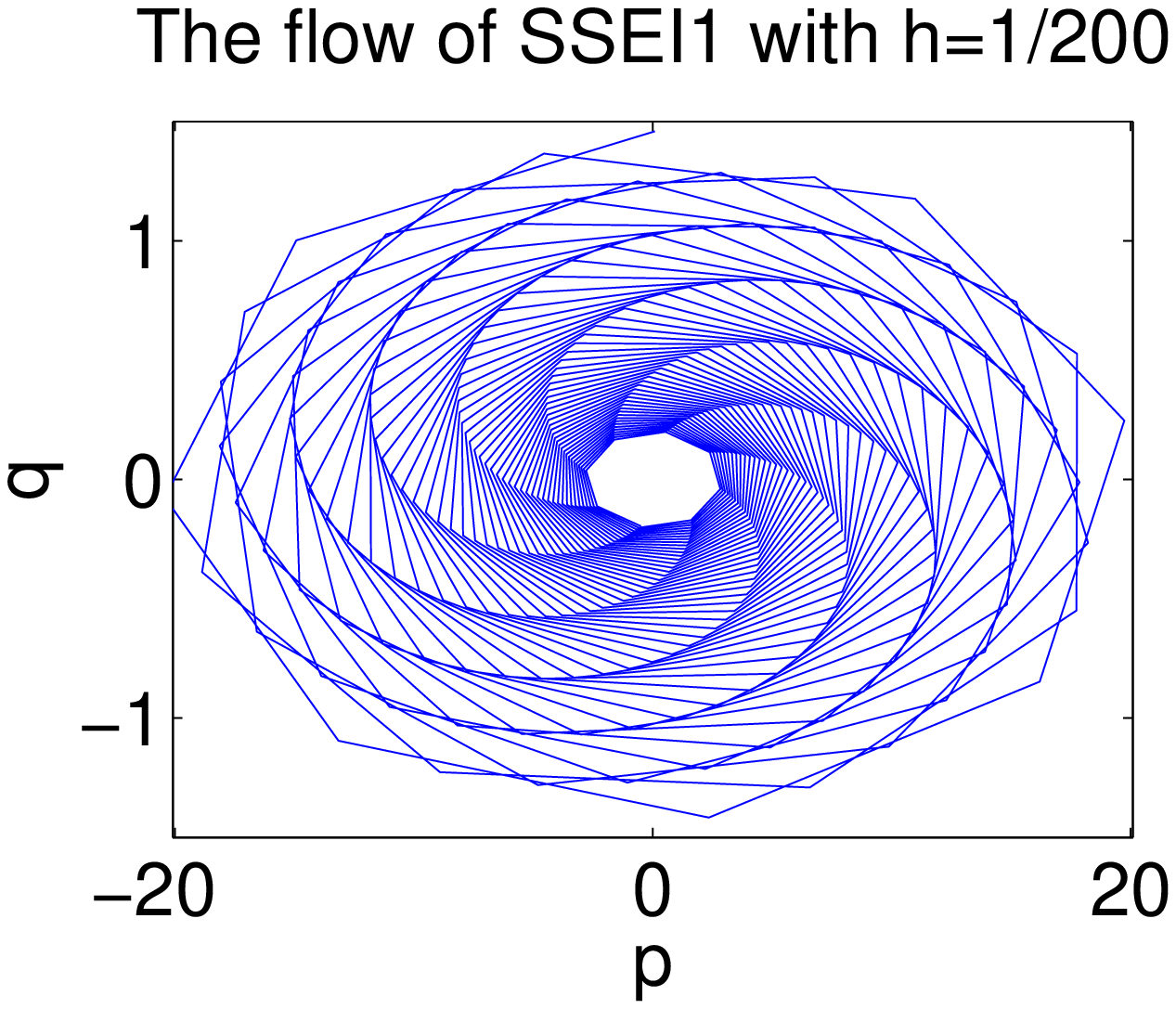}\\
\includegraphics[width=3.0cm,height=3.0cm]{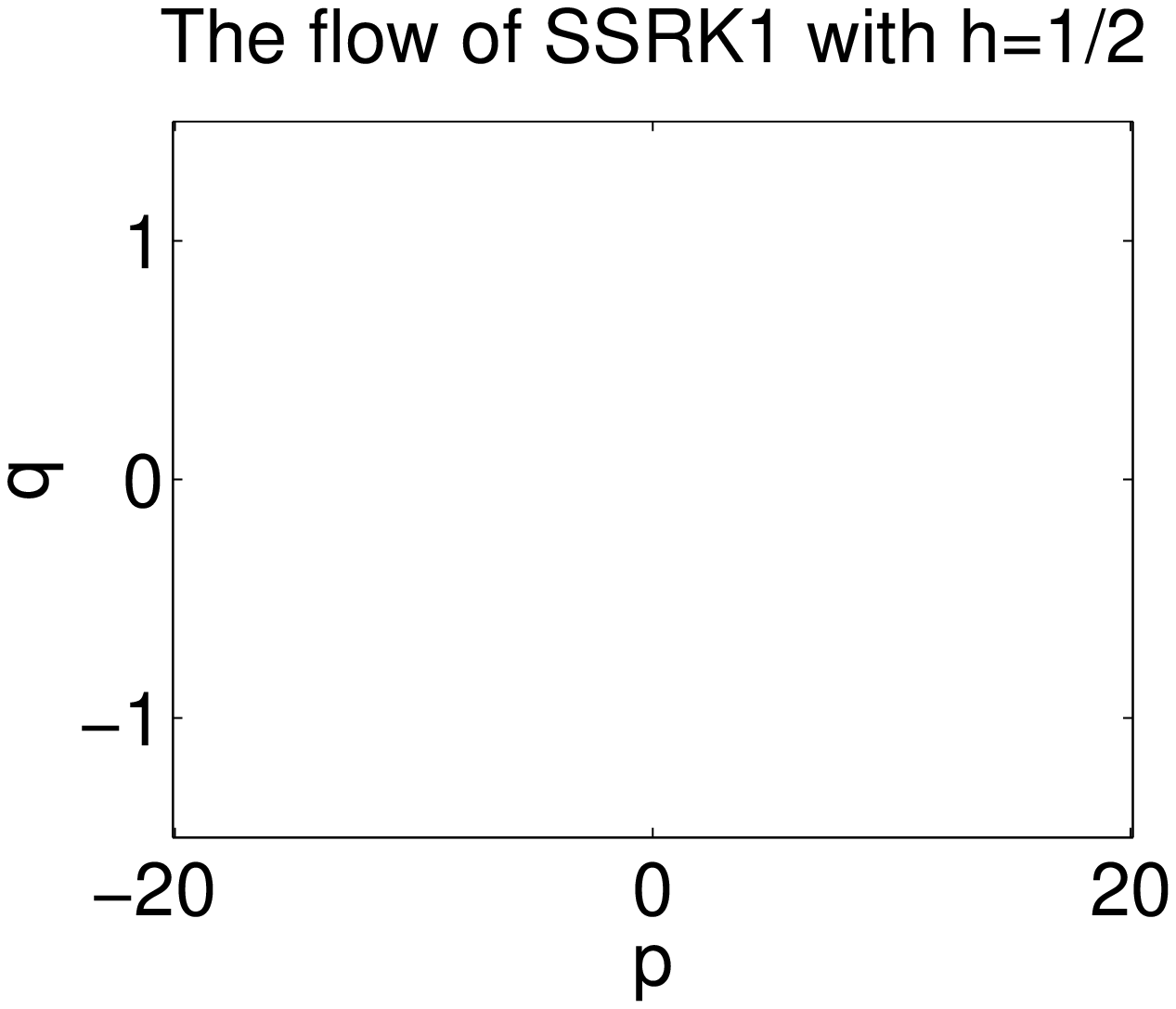}
\includegraphics[width=3.0cm,height=3.0cm]{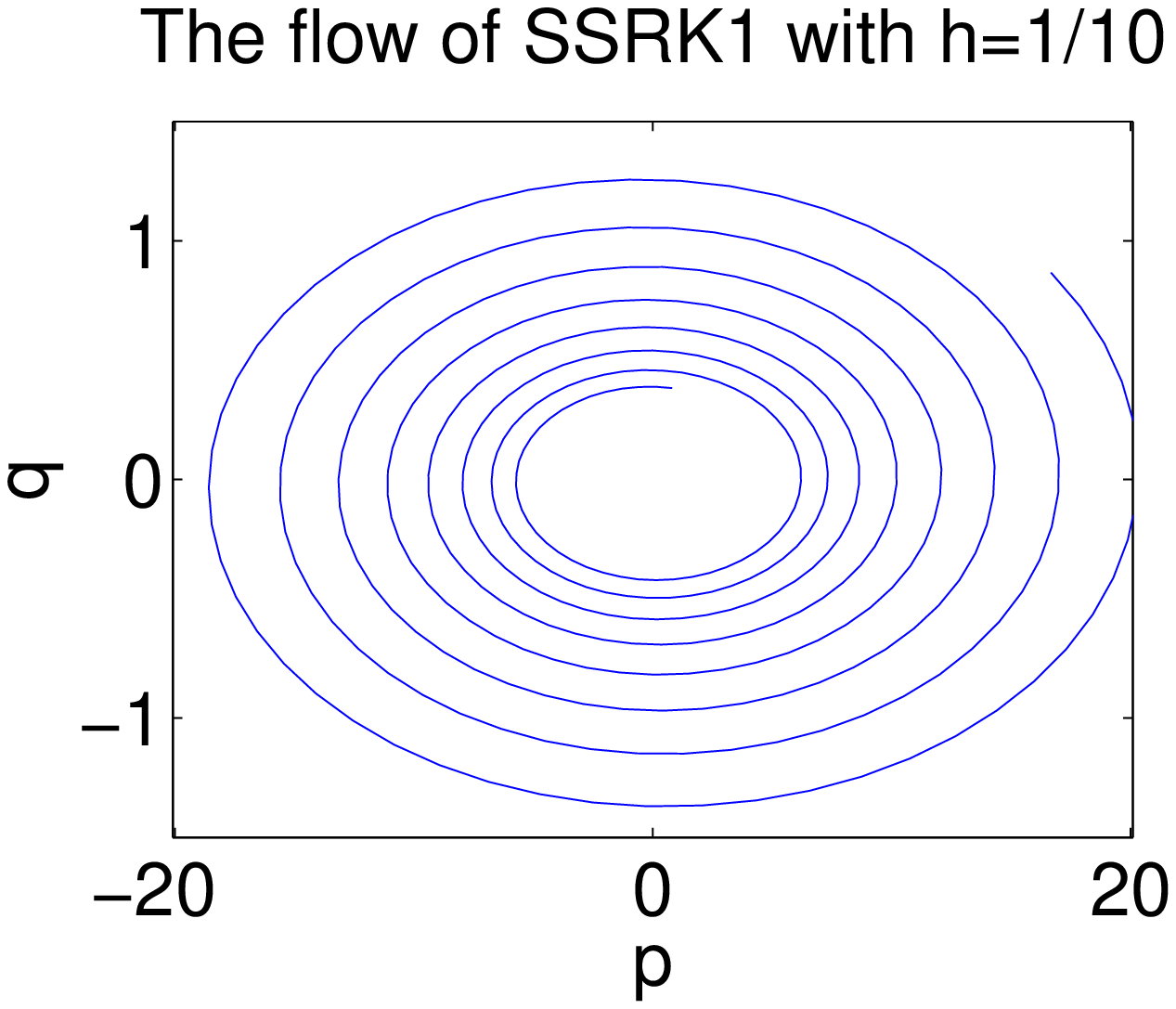}
\includegraphics[width=3.0cm,height=3.0cm]{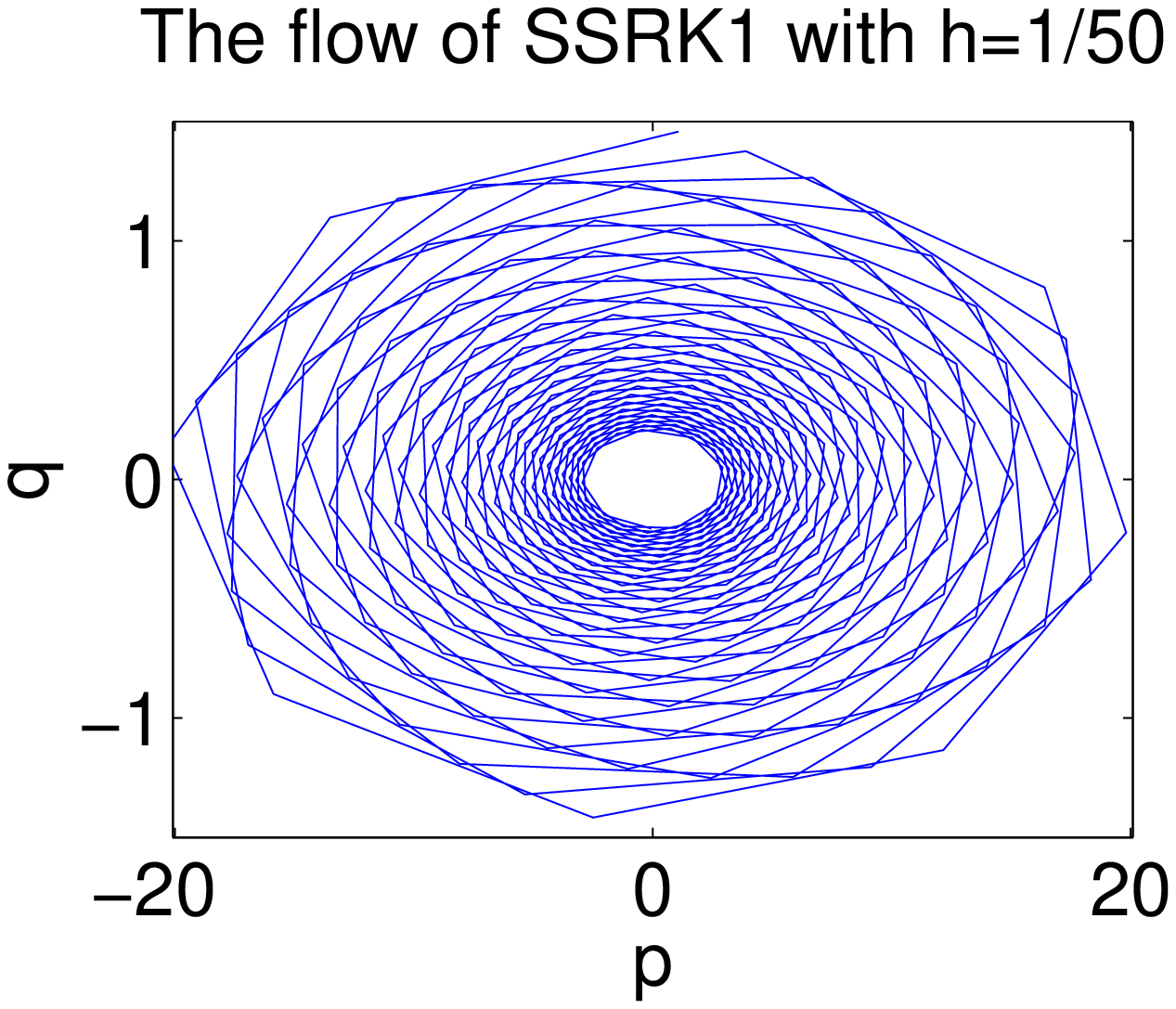}
\includegraphics[width=3.0cm,height=3.0cm]{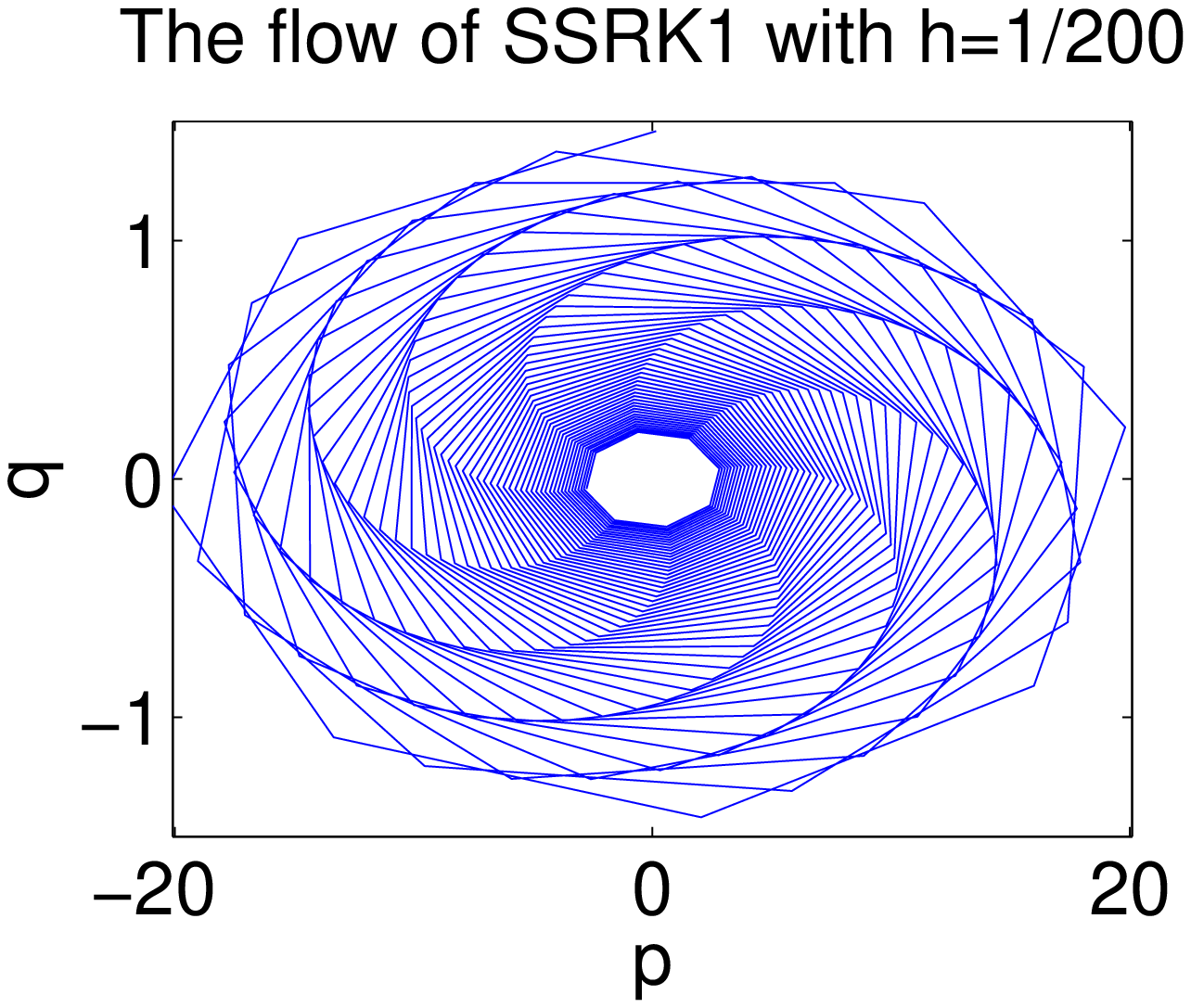}\\
\includegraphics[width=3.0cm,height=3.0cm]{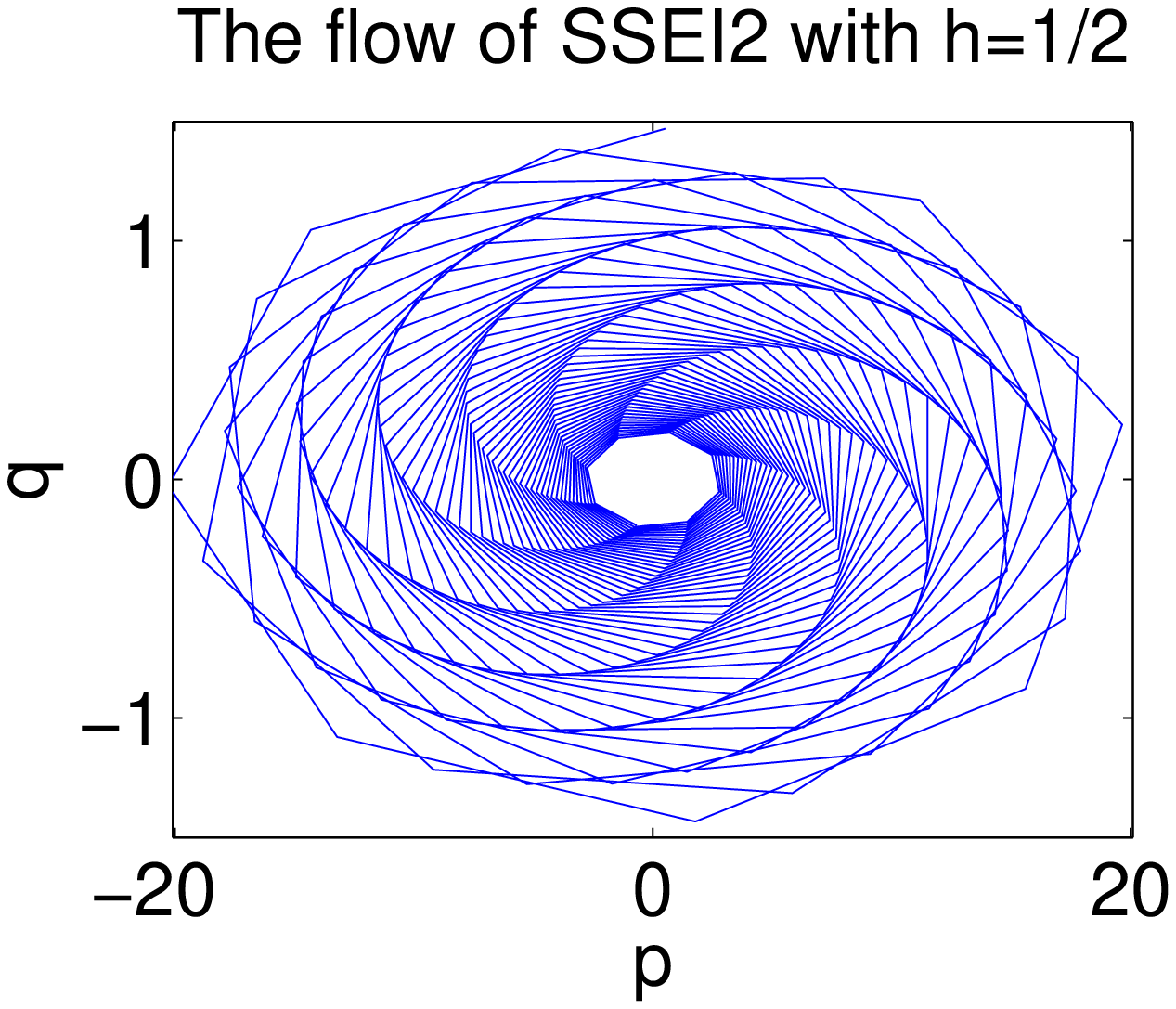}
\includegraphics[width=3.0cm,height=3.0cm]{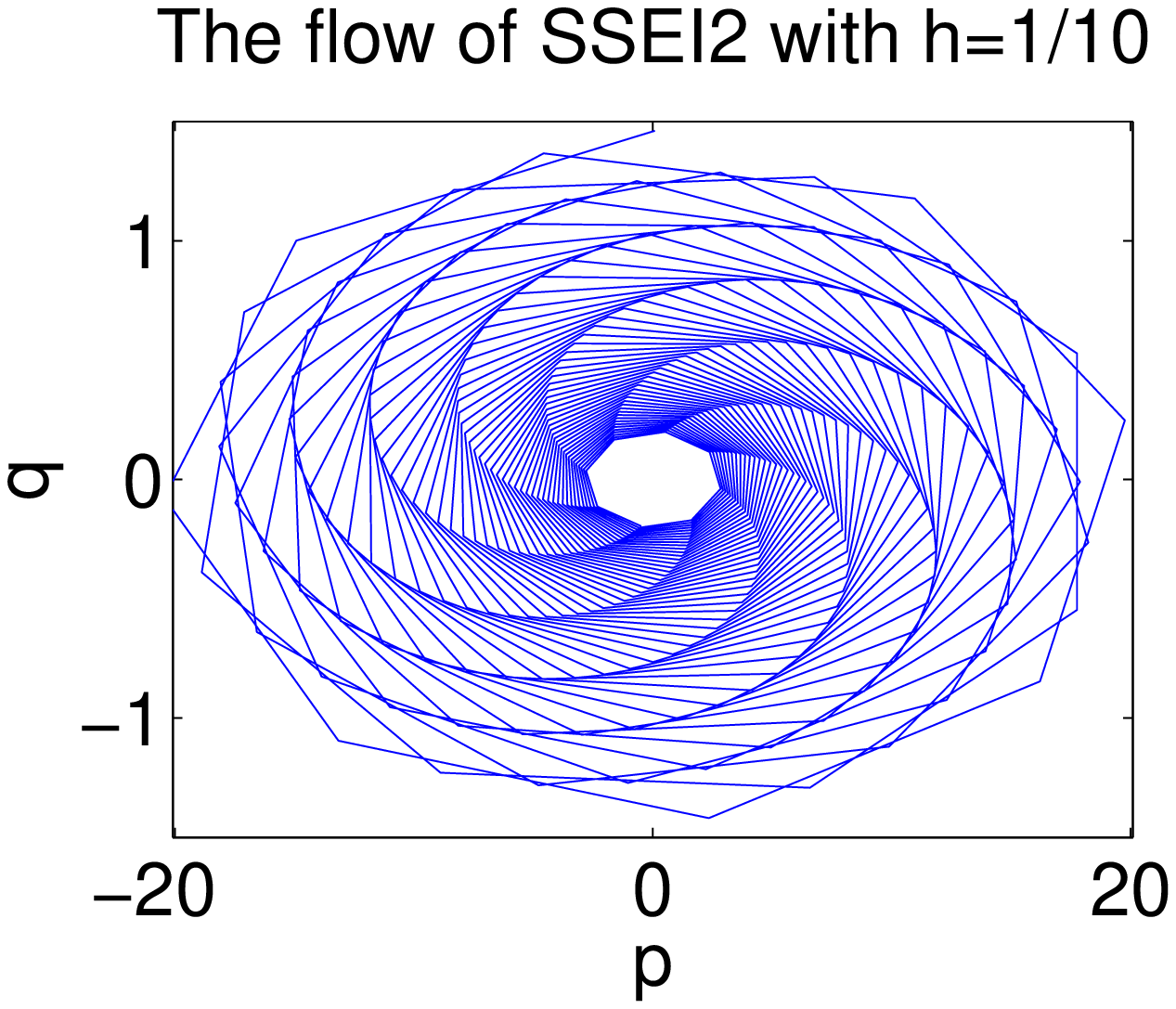}
\includegraphics[width=3.0cm,height=3.0cm]{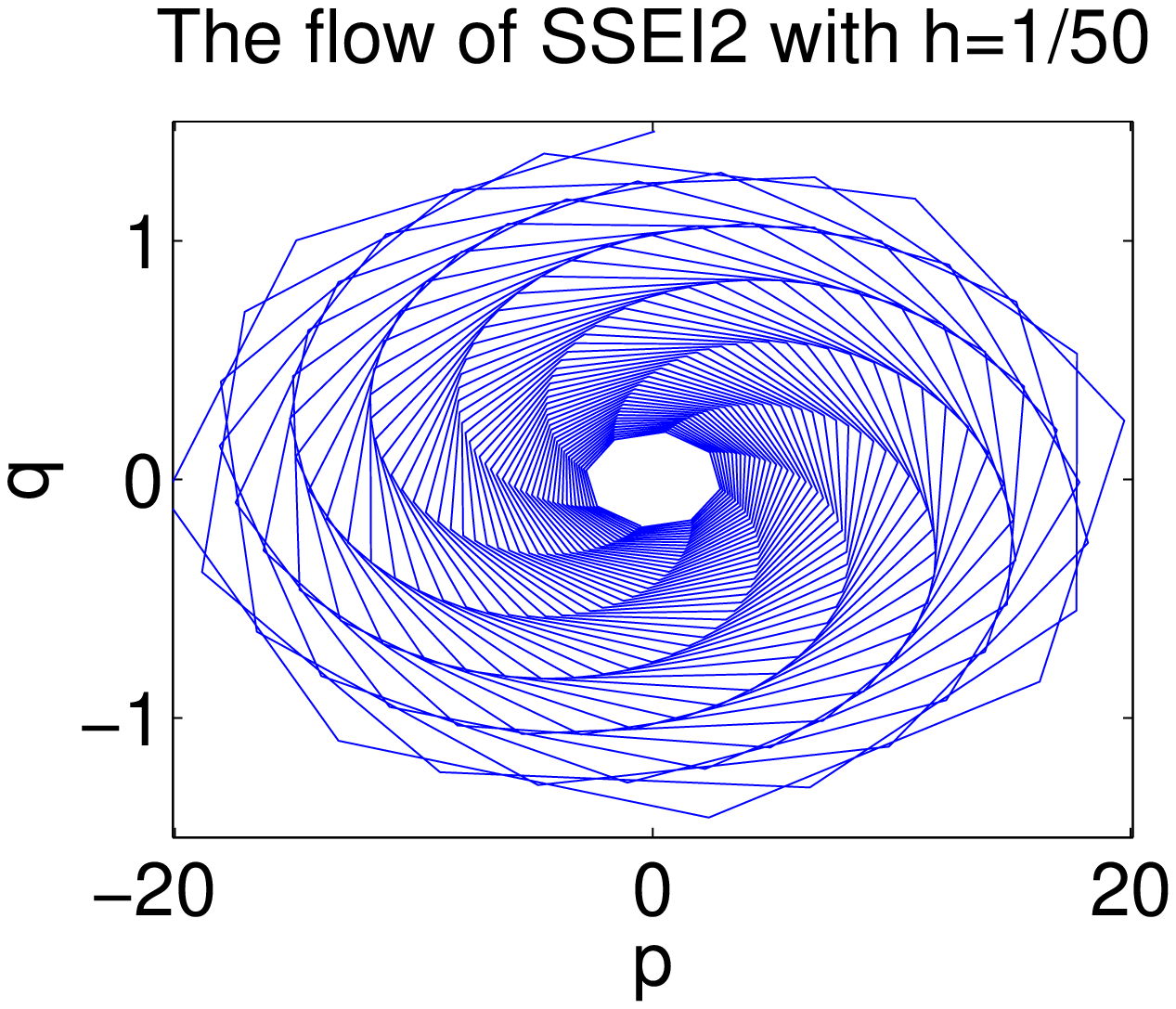}
\includegraphics[width=3.0cm,height=3.0cm]{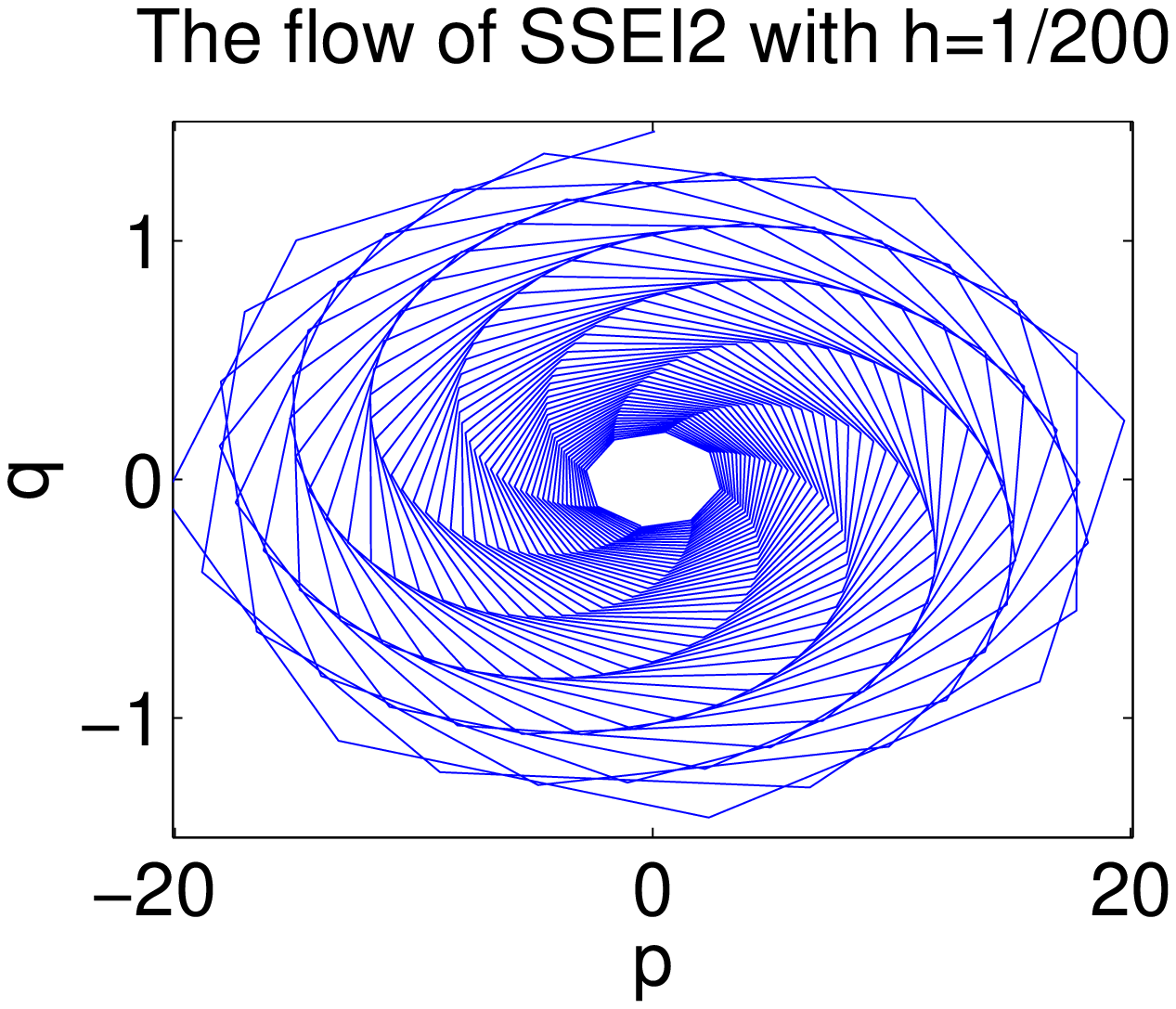}\\
\includegraphics[width=3.0cm,height=3.0cm]{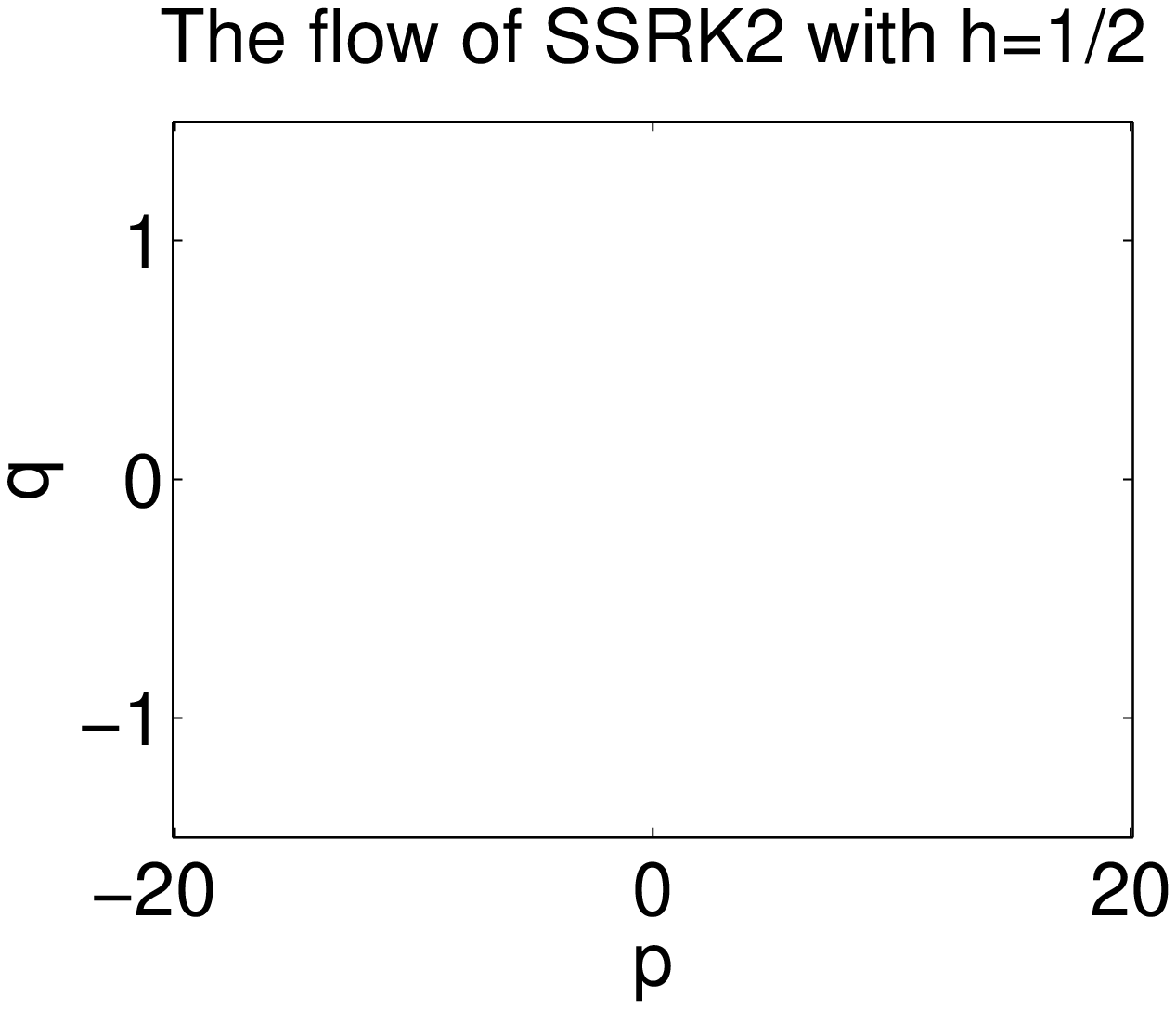}
\includegraphics[width=3.0cm,height=3.0cm]{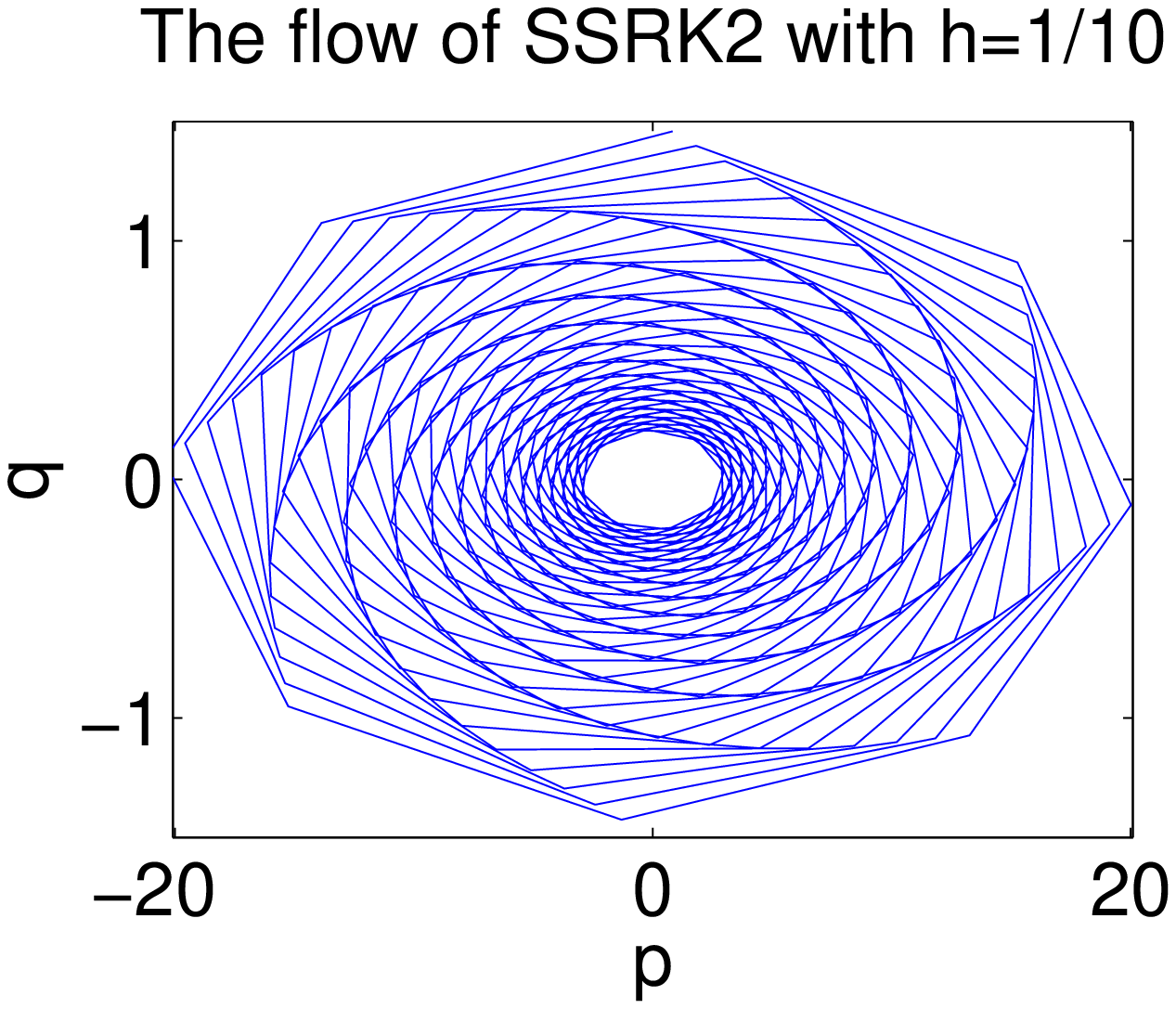}
\includegraphics[width=3.0cm,height=3.0cm]{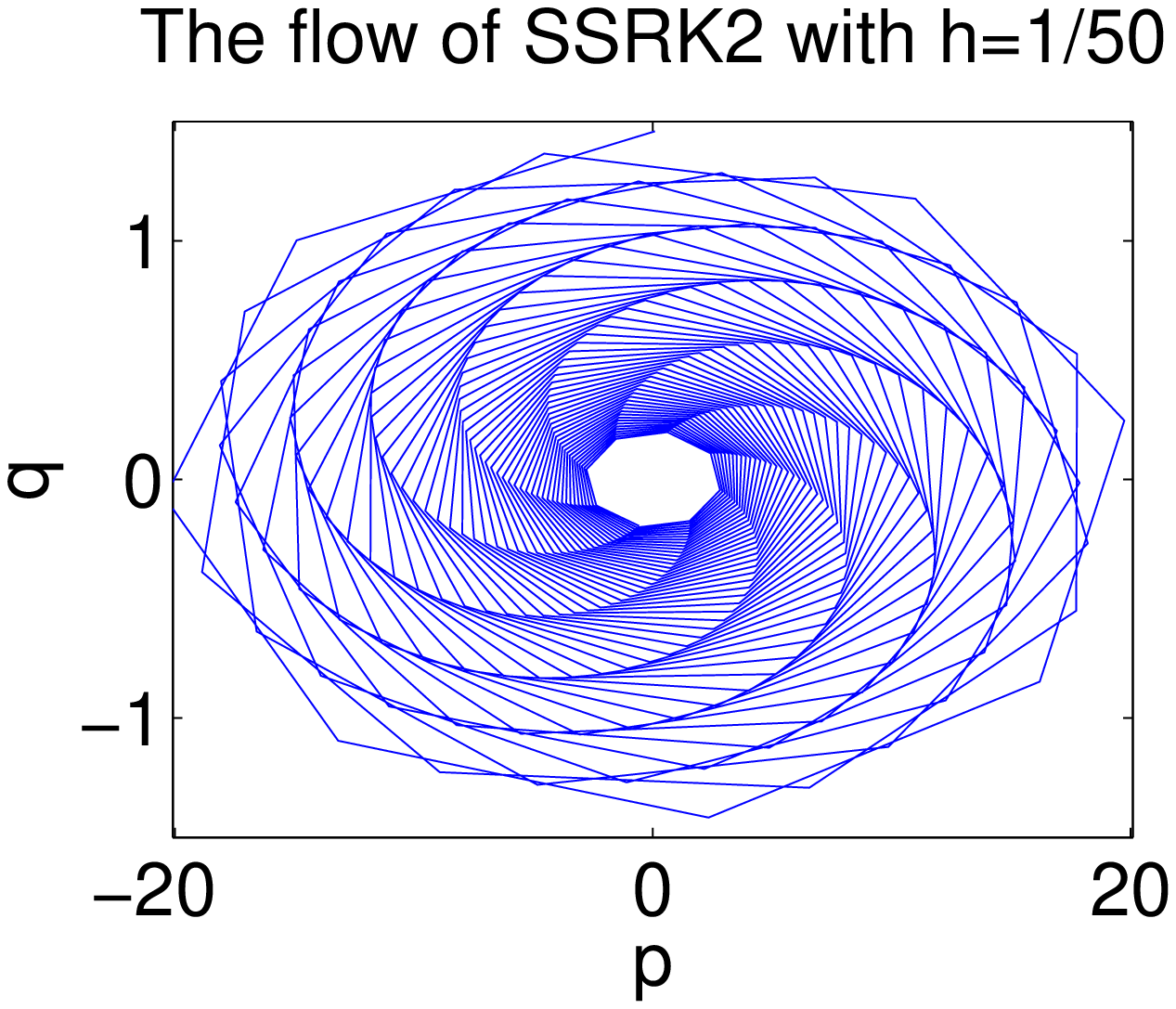}
\includegraphics[width=3.0cm,height=3.0cm]{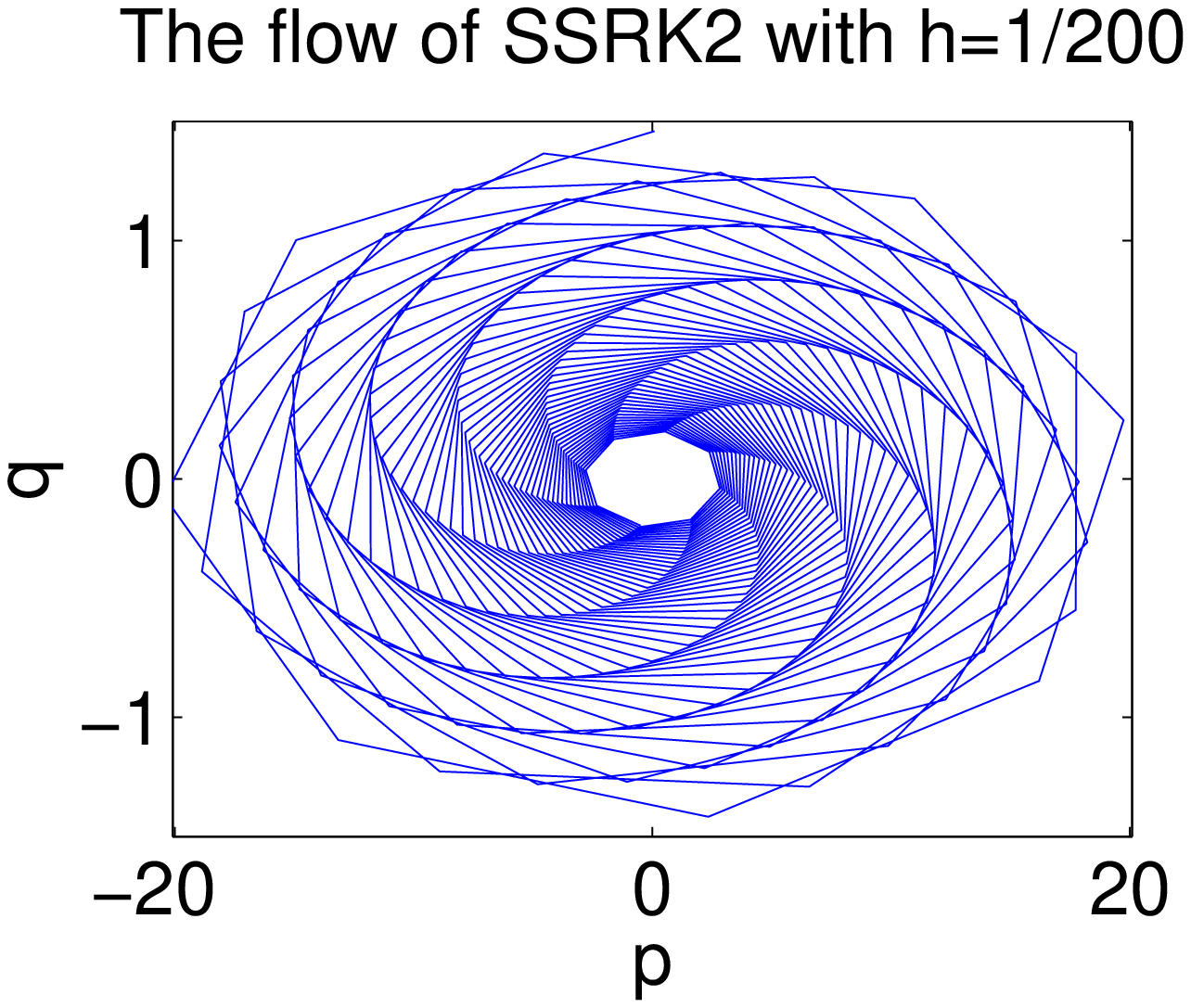}
\caption{Problem 3: the flows of different methods.} \label{p3-1}
\end{figure}
 \begin{figure}[ptb]
\centering
\includegraphics[width=3.8cm,height=4cm]{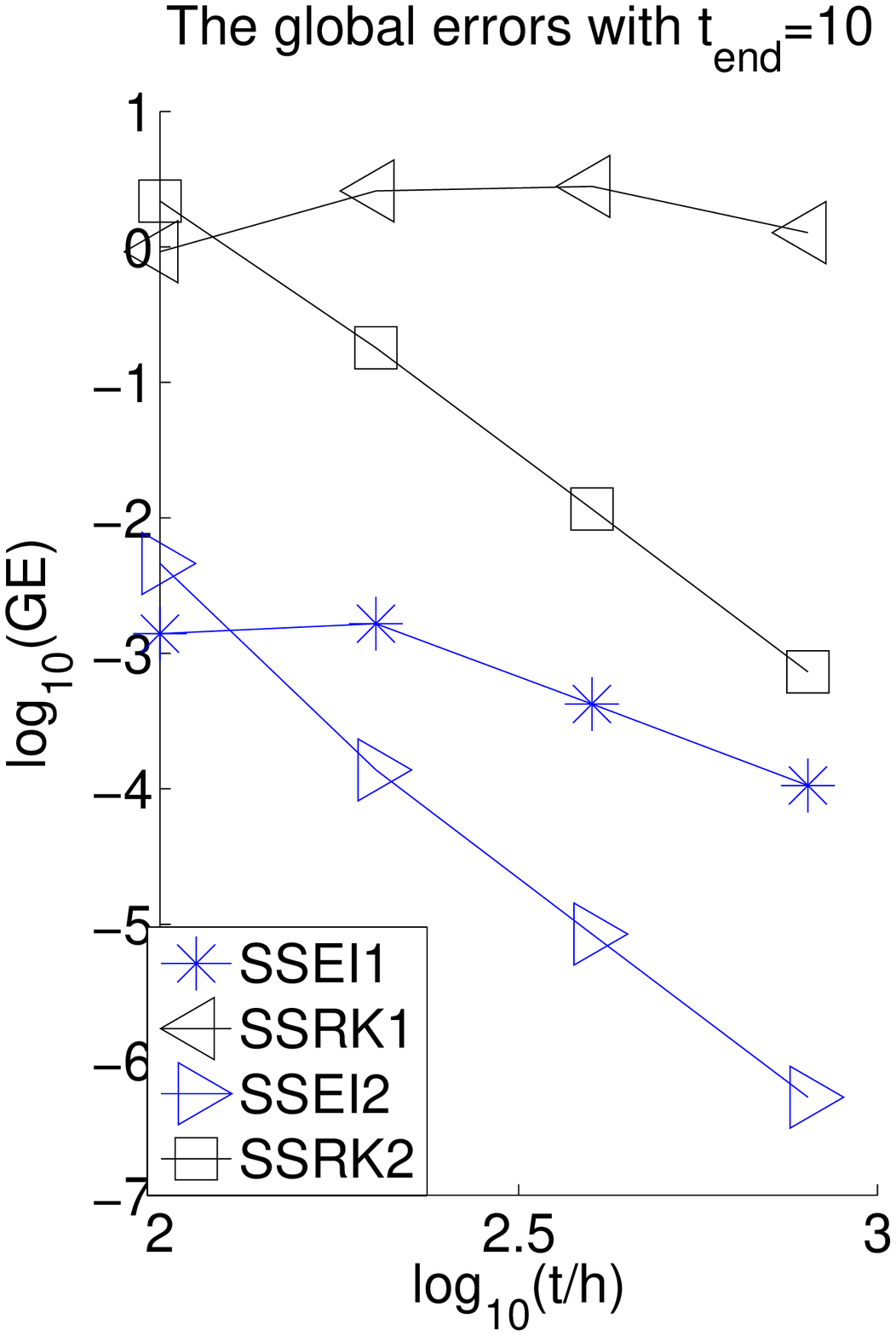}
\includegraphics[width=3.8cm,height=4cm]{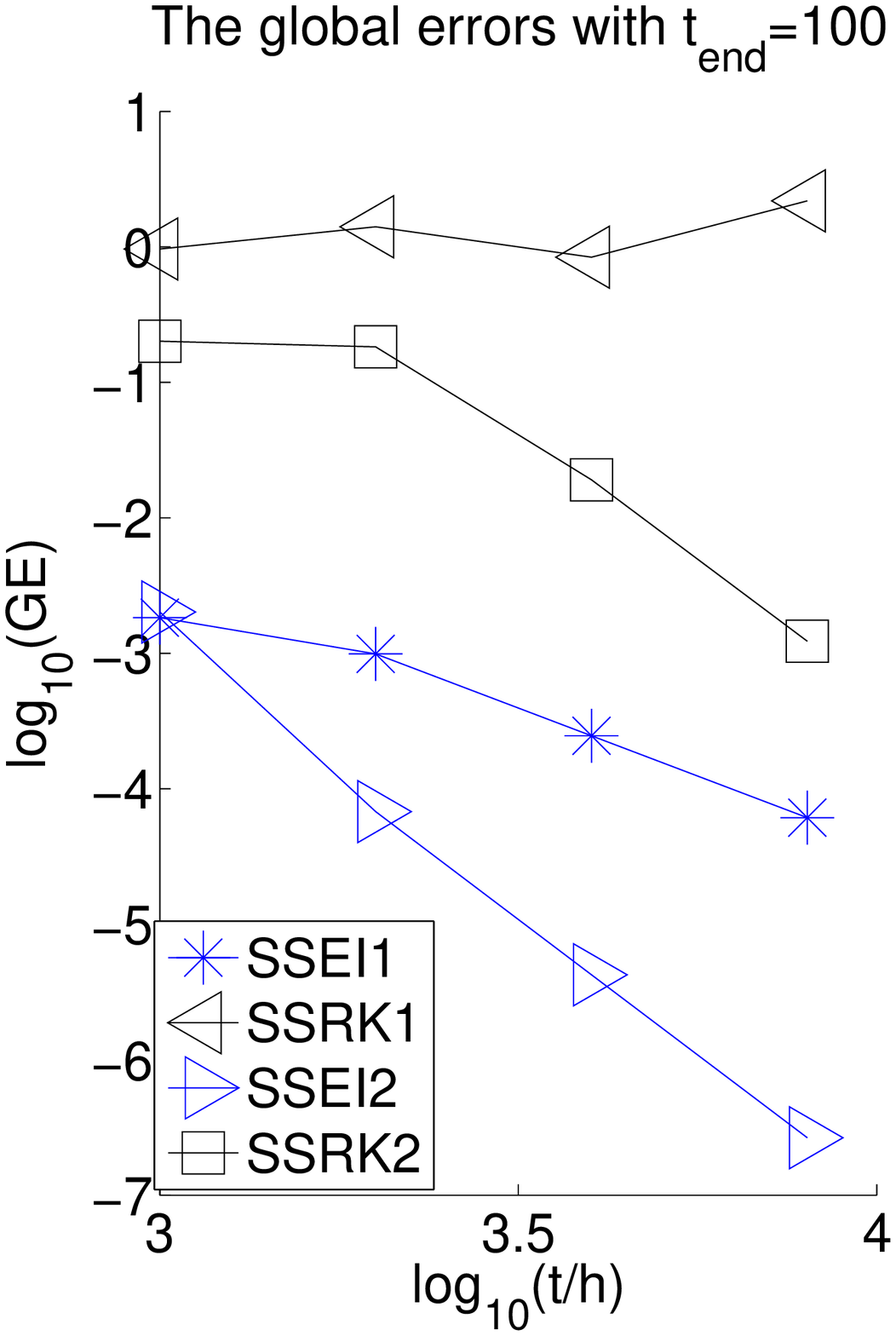}
\includegraphics[width=3.8cm,height=4cm]{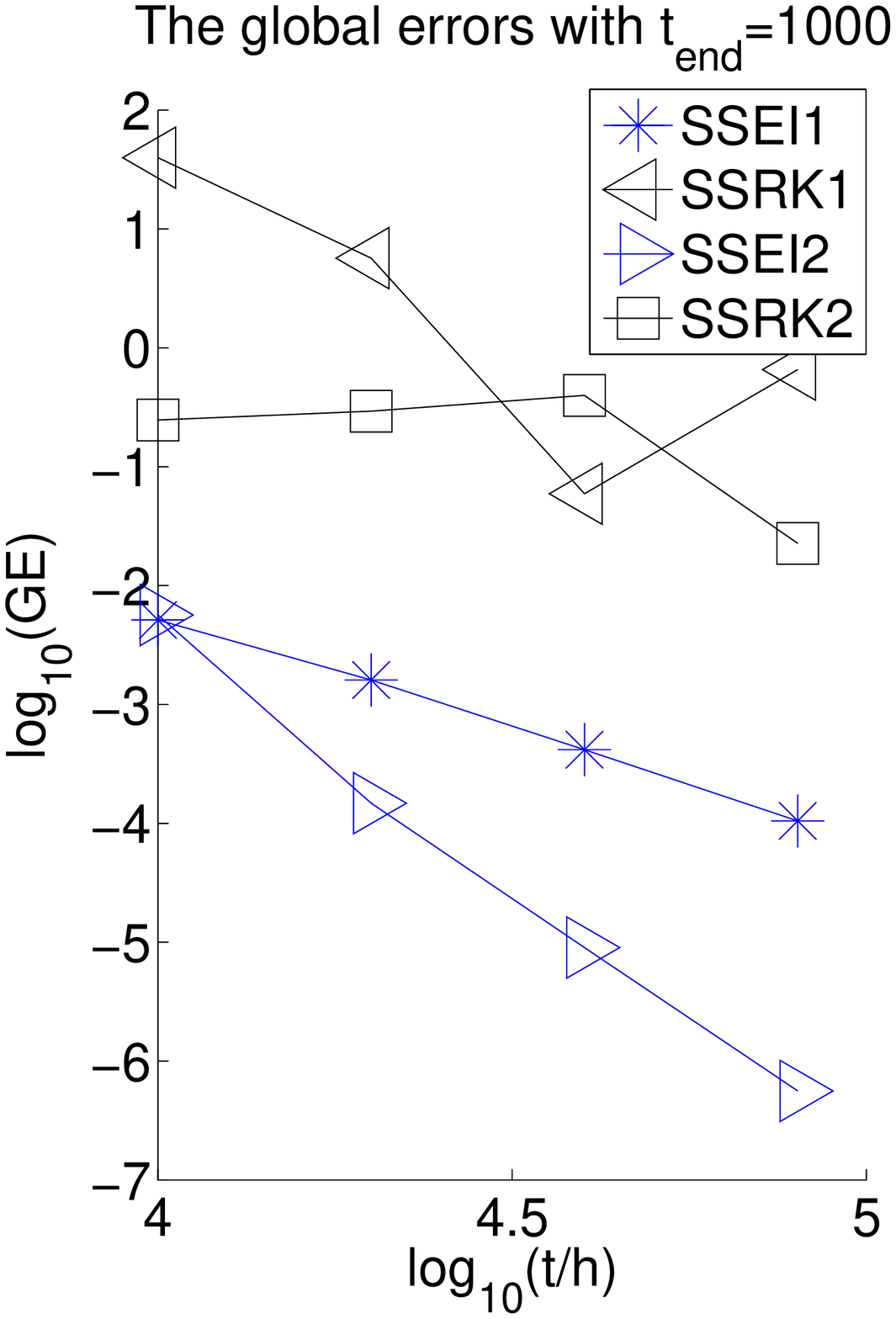}
\caption{Problem 3: the relative global errors.} \label{p3-2}
\end{figure}

 \vskip2mm\noindent\textbf{Problem 4.}
The last  numerical experiment is concerned with  the charged
particle system  with a constant magnetic field (see
\cite{Hairer2017-2}). The system can be given by \eqref{CPD} with
the potential $U(x)=\frac{1}{100\sqrt{x_1^2+x_2^2}}$ and the
constant magnetic field $B=(0,0,10)^{\intercal}.$  The initial
values are chosen as $x(0)=(0.7,1,0.1)^{\intercal}$ and
$x'(0)=(0.9,0.5,0.4)^{\intercal}.$ We firstly integrate this system
 on $[0,100]$ with $h=1/2,1/10,1/50,1/200$ and show the numerical
flows $x_2$ and $v_2=x_2'$   at the  time points
$\{\frac{1}{2}i\}_{i=1,\ldots,200}$  in  Figure \ref{p4-1}.  Then
 the problem is solved with   $t_{\textmd{end}}=10,100, 1000$ and
$h= 0.1/2^{i}$ for $i=0,\ldots,3$ and the relative global errors are
shown in Figure \ref{p4-2}. The SSEI methods are also shown to be
robust to this problem. Here, it is important to note that our SSEI1
method is explicit (see \eqref{spe EI for second}) when applied to
this problem, whereas, the SSRK1 method is implicit and the
iteration is required for solving this problem. This fact shows
another advantage of our volume-preserving exponential integrators
in comparison with volume-preserving RK methods.

 \begin{figure}[ptb]
\centering
\includegraphics[width=3.0cm,height=3.0cm]{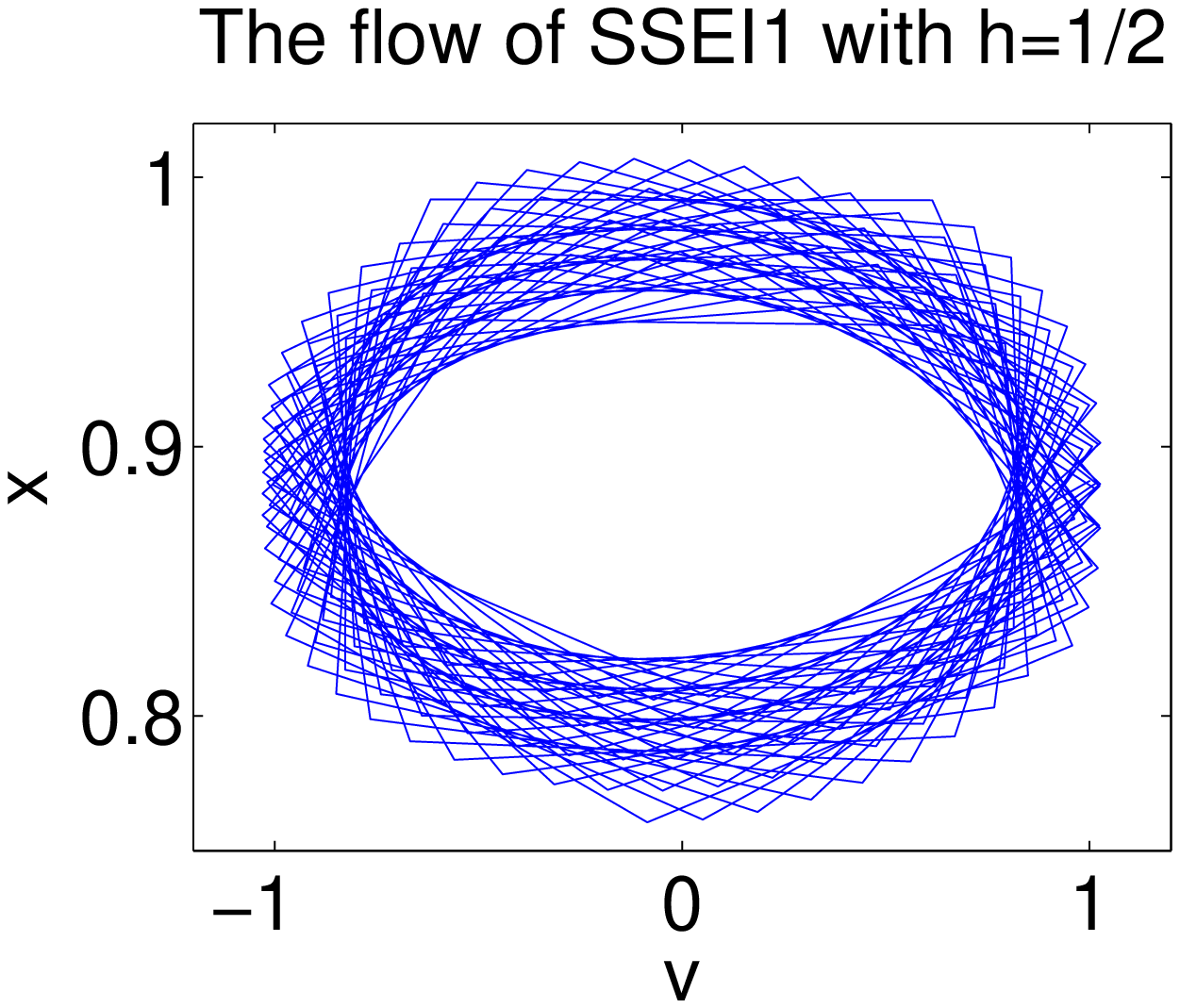}
\includegraphics[width=3.0cm,height=3.0cm]{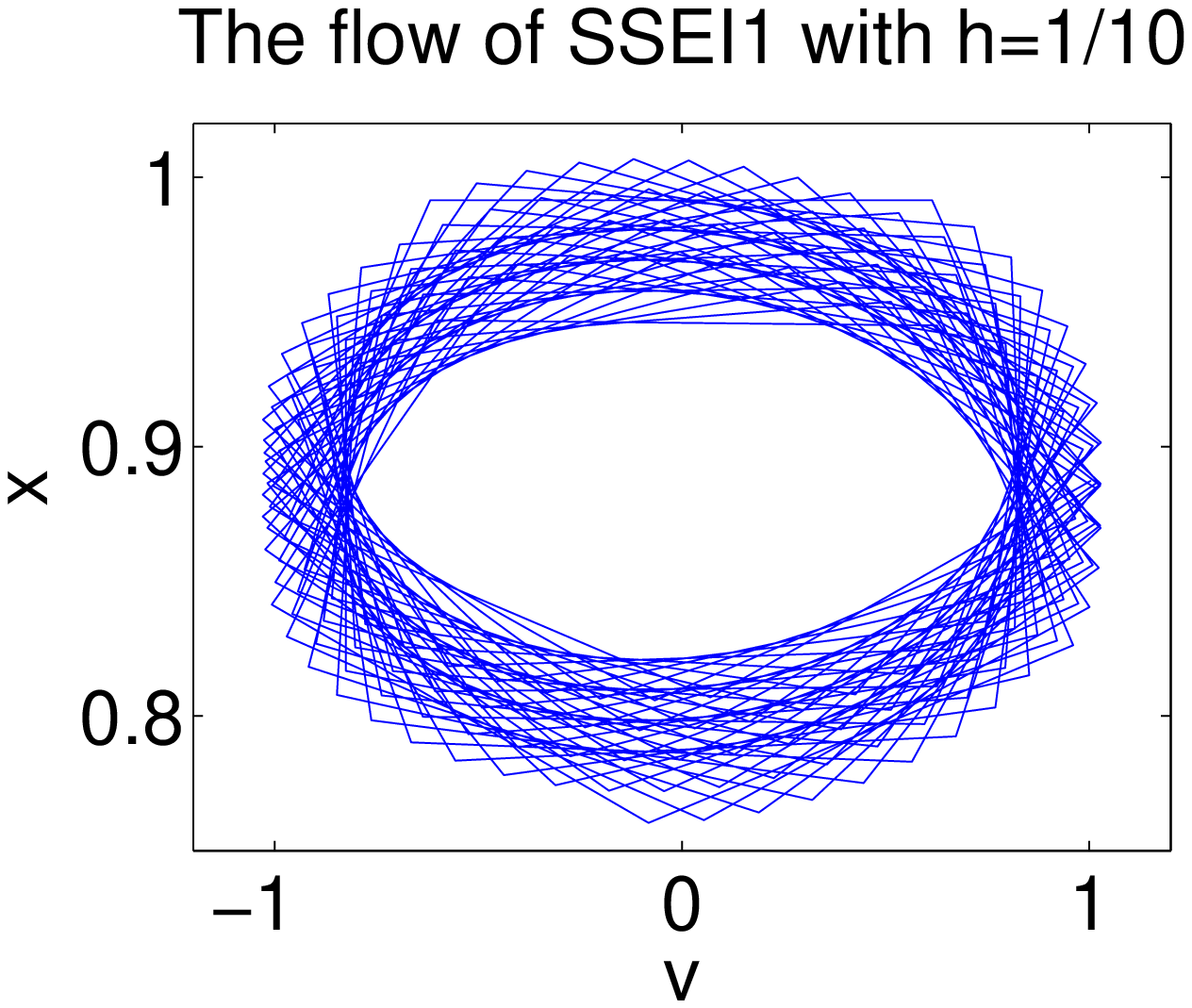}
\includegraphics[width=3.0cm,height=3.0cm]{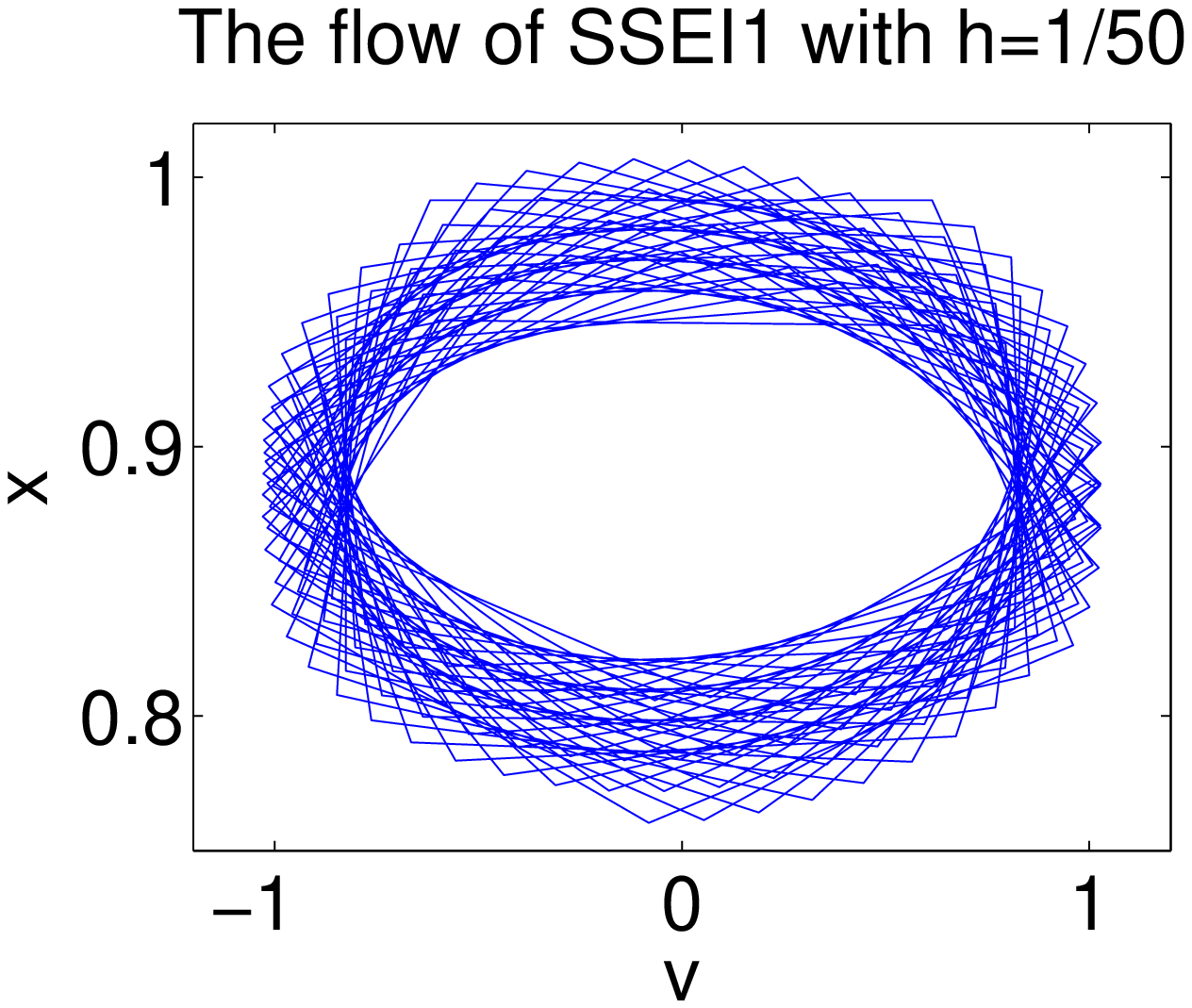}
\includegraphics[width=3.0cm,height=3.0cm]{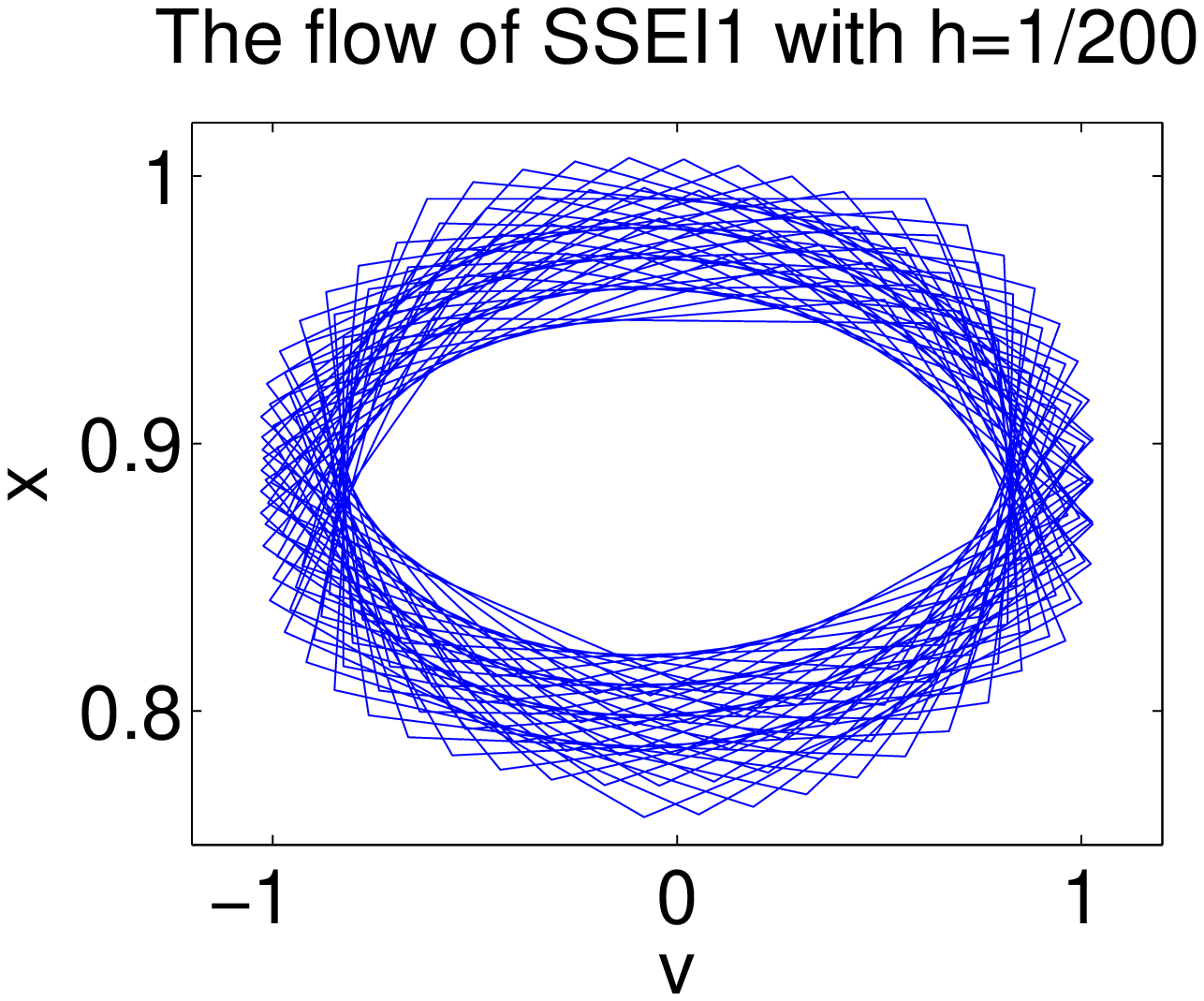}\\
\includegraphics[width=3.0cm,height=3.0cm]{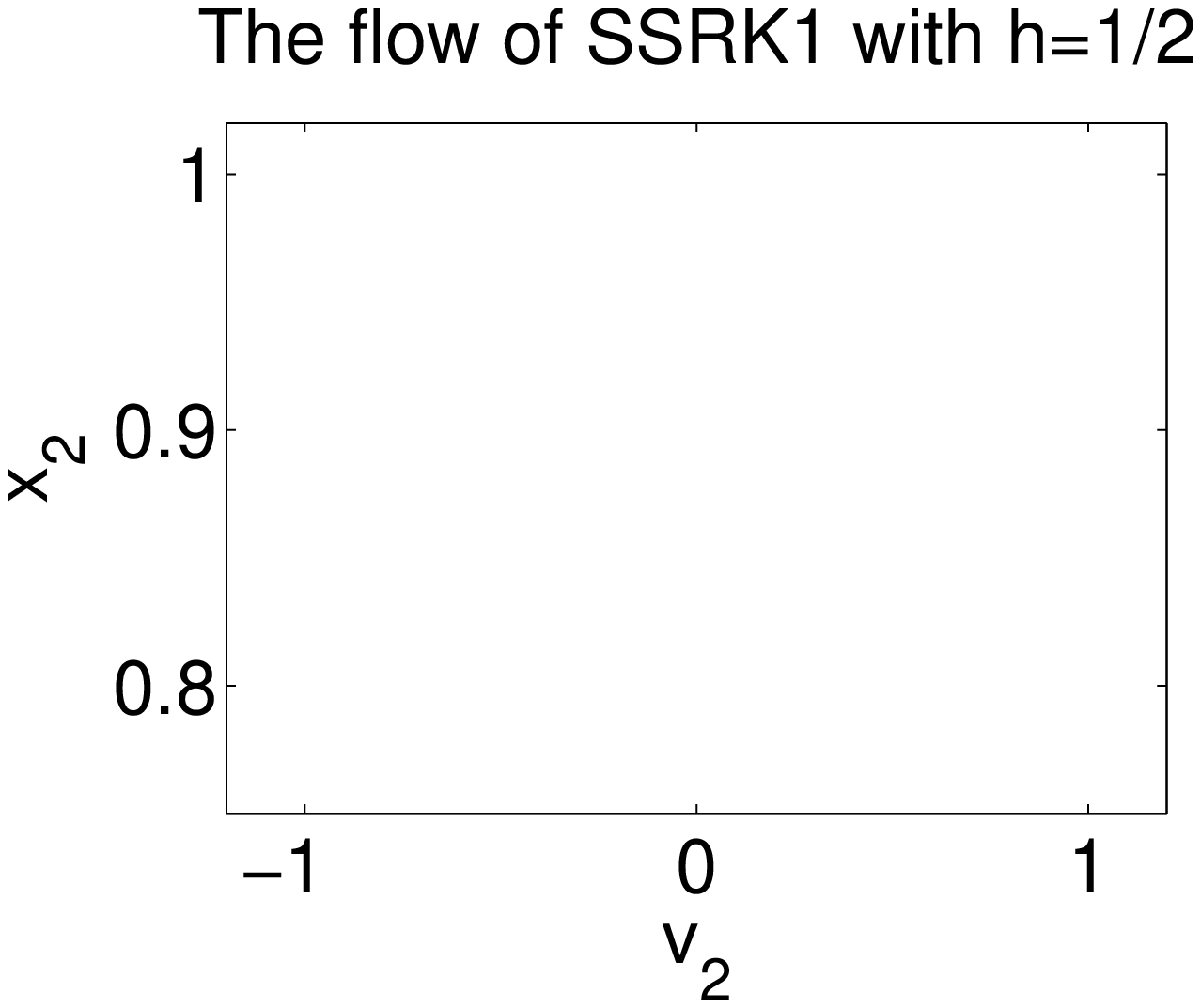}
\includegraphics[width=3.0cm,height=3.0cm]{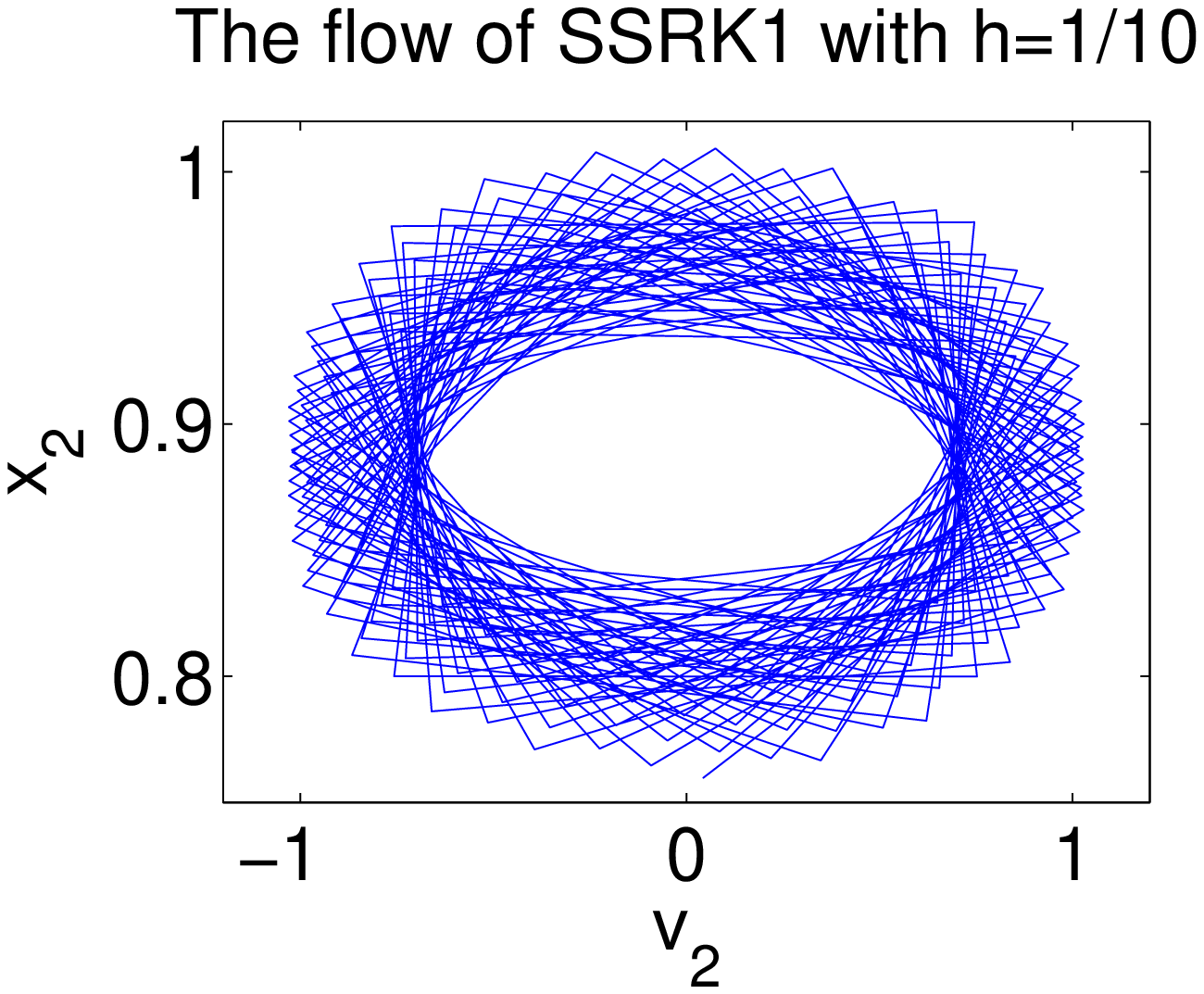}
\includegraphics[width=3.0cm,height=3.0cm]{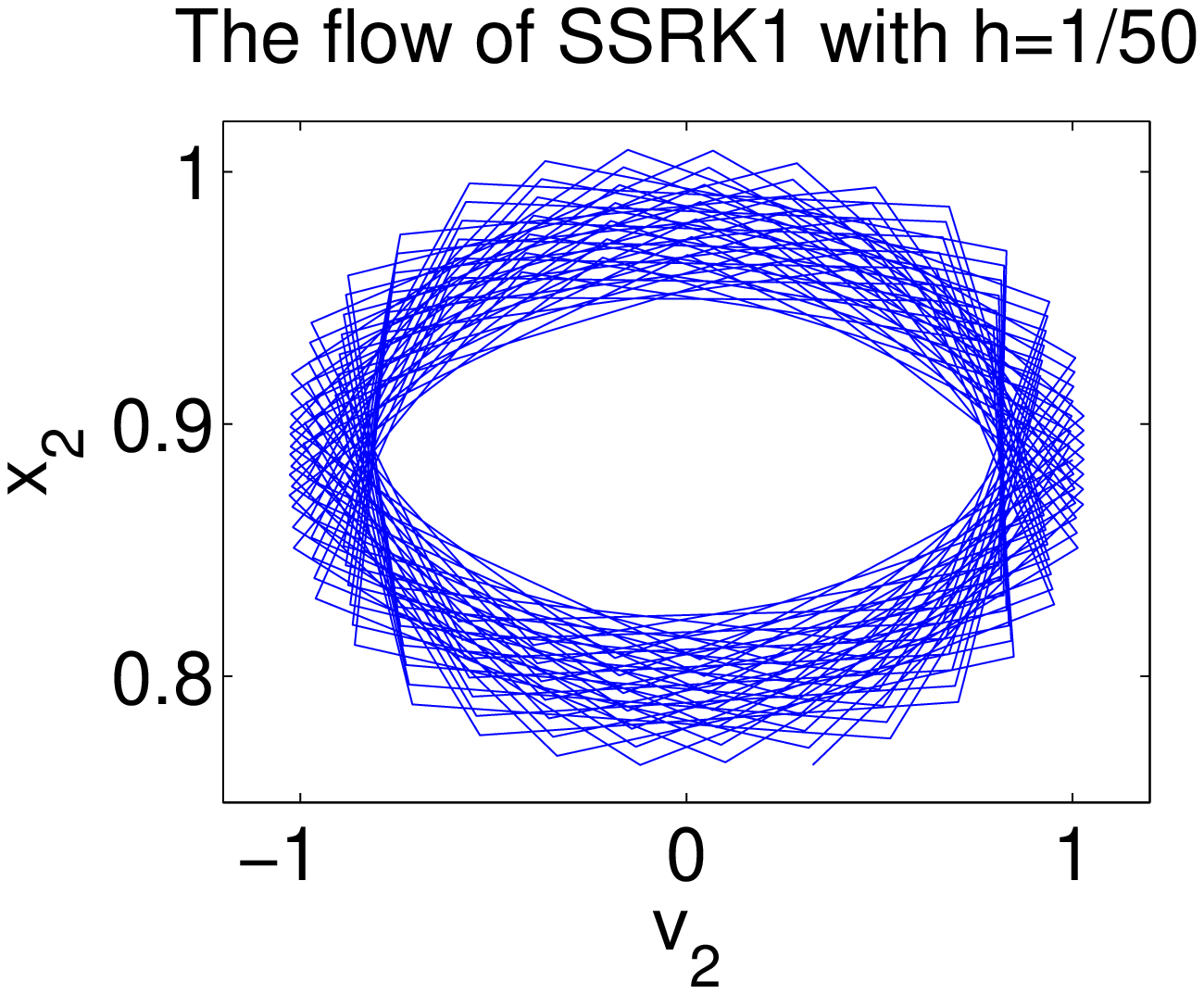}
\includegraphics[width=3.0cm,height=3.0cm]{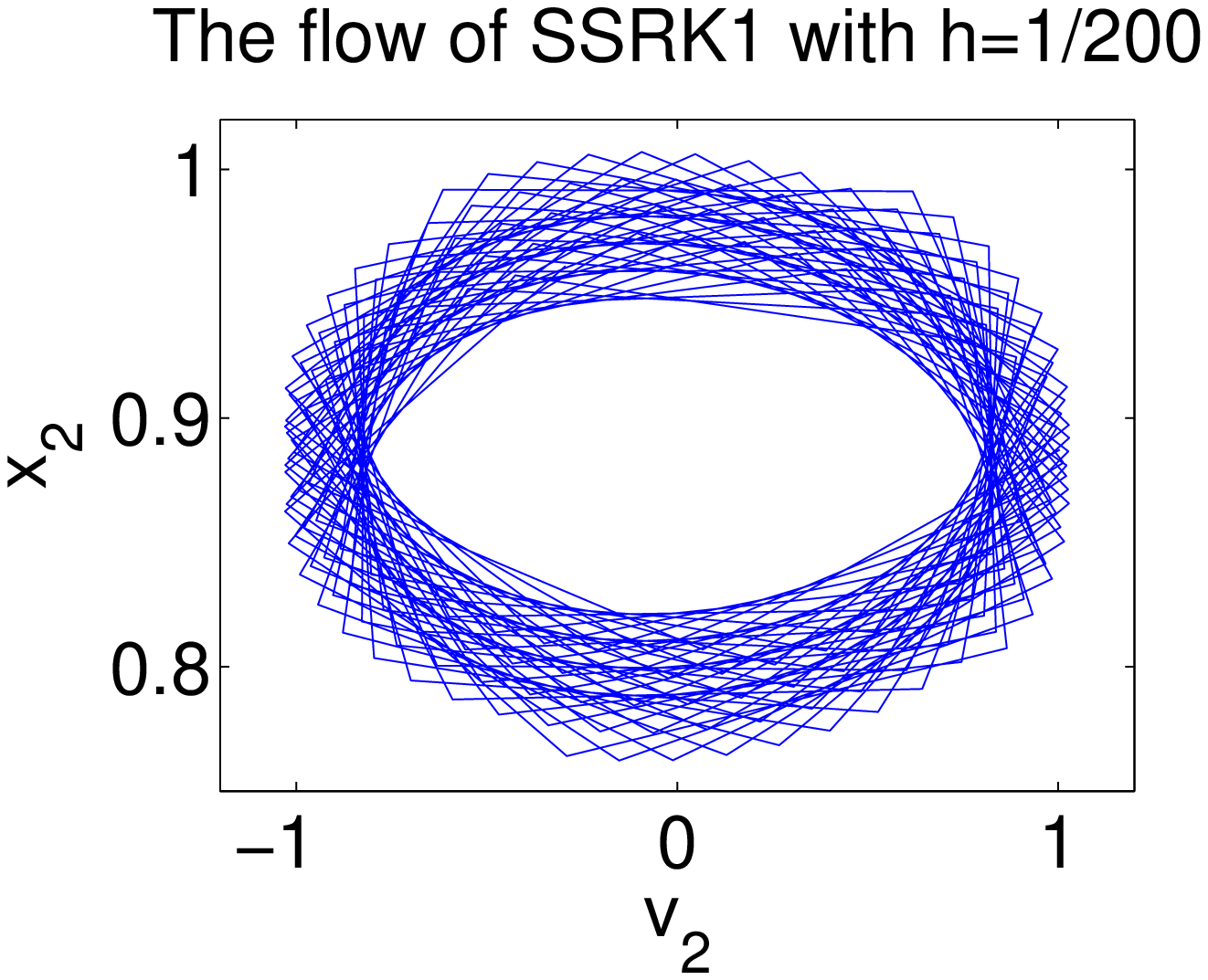}\\
\includegraphics[width=3.0cm,height=3.0cm]{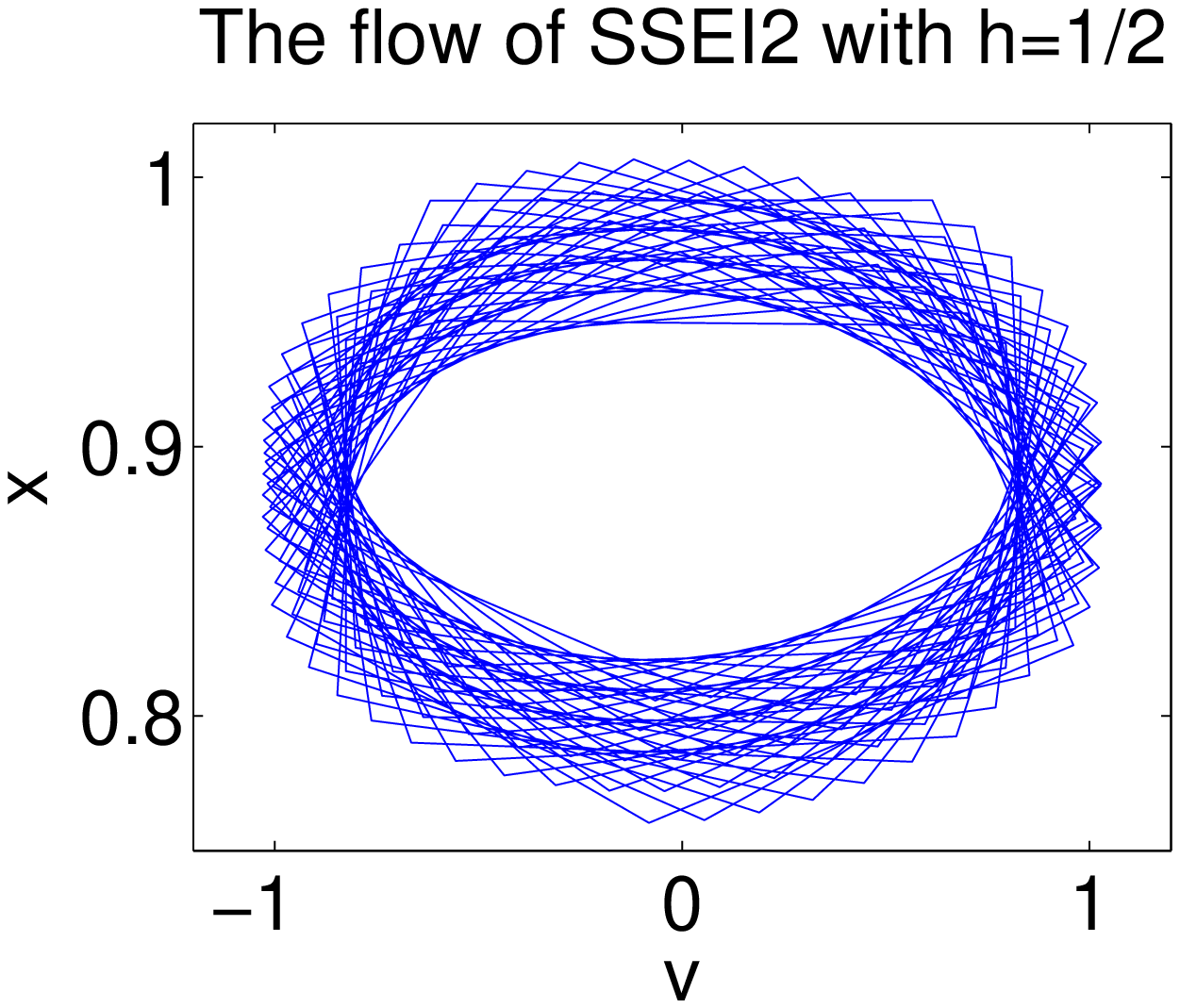}
\includegraphics[width=3.0cm,height=3.0cm]{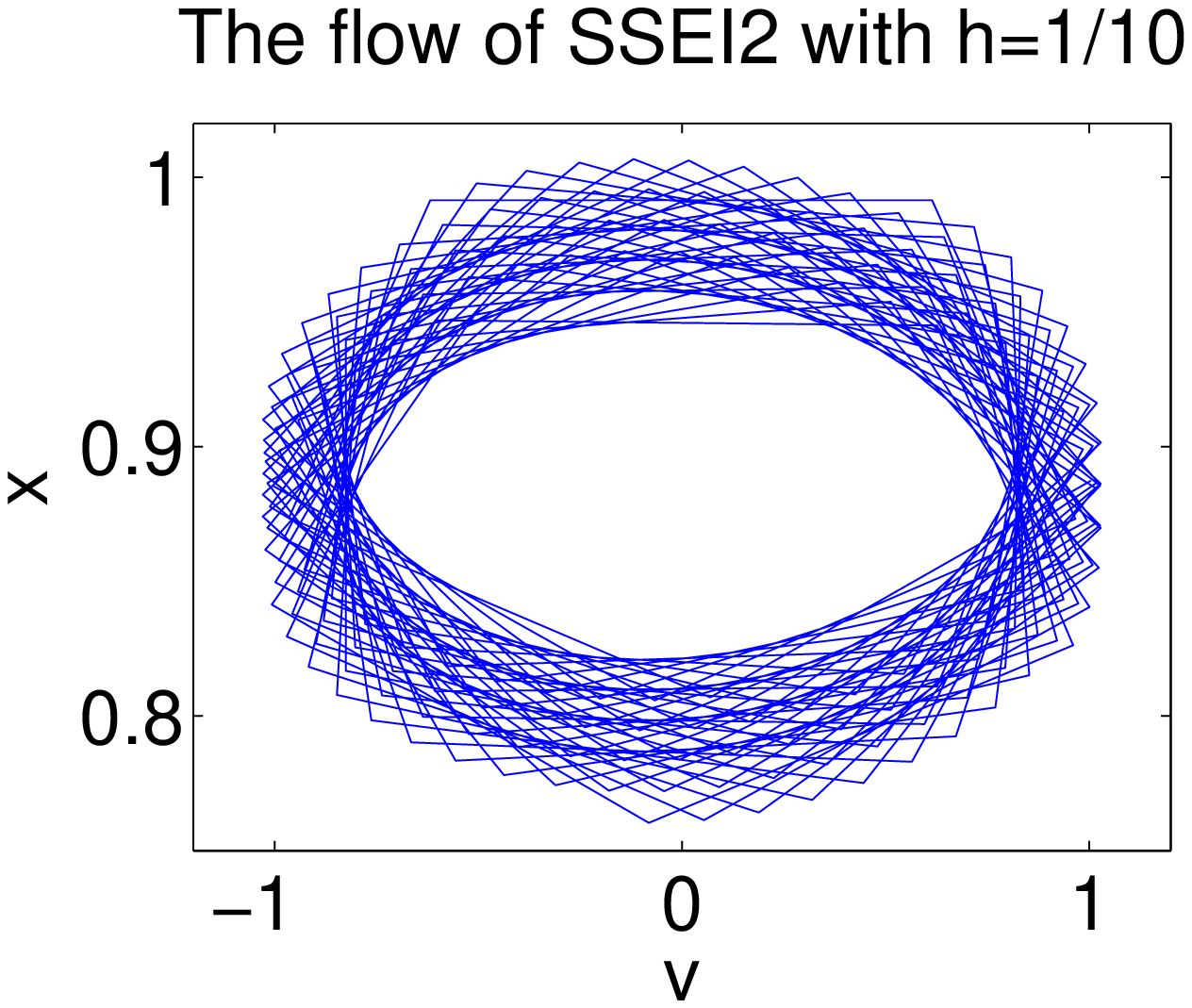}
\includegraphics[width=3.0cm,height=3.0cm]{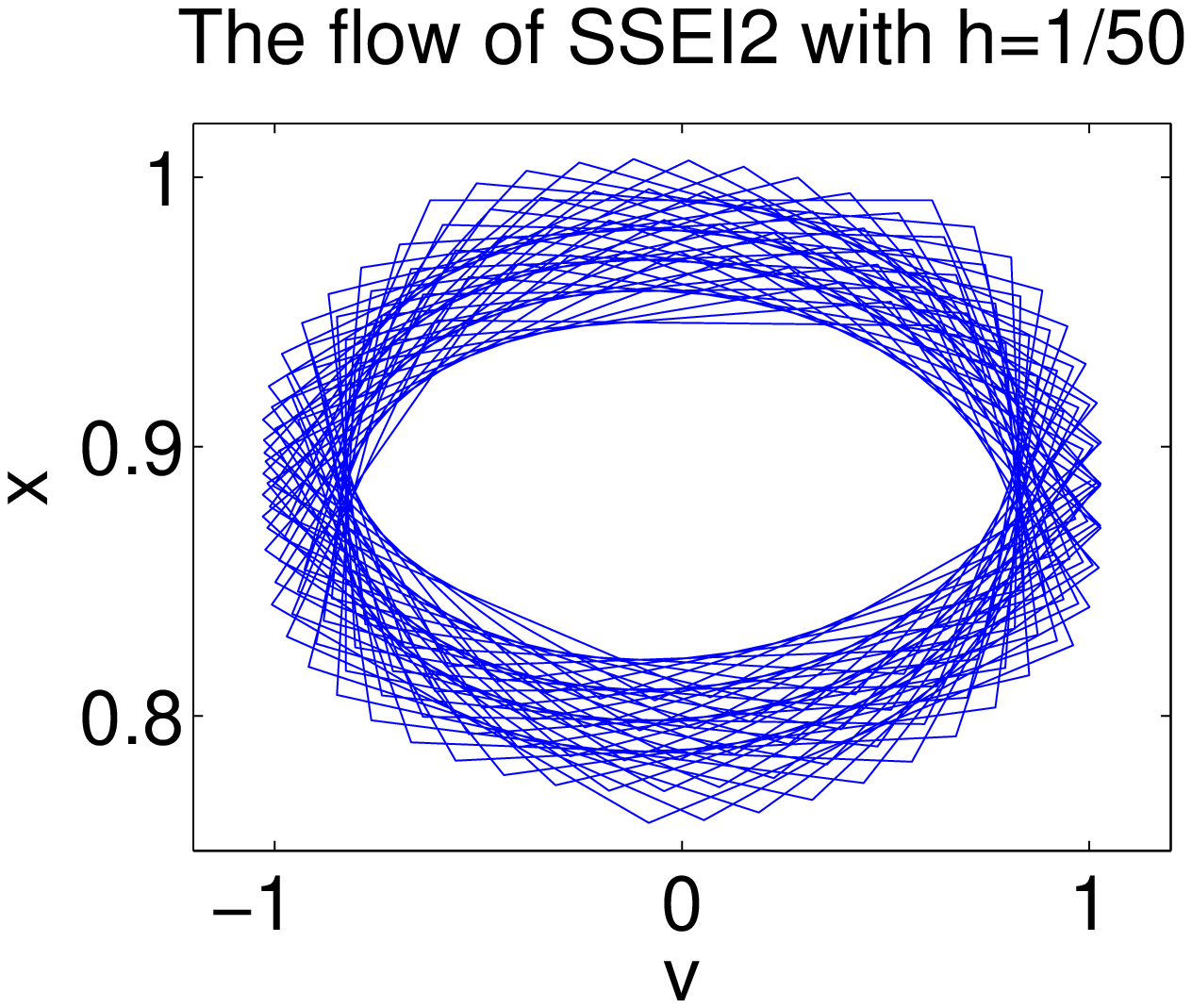}
\includegraphics[width=3.0cm,height=3.0cm]{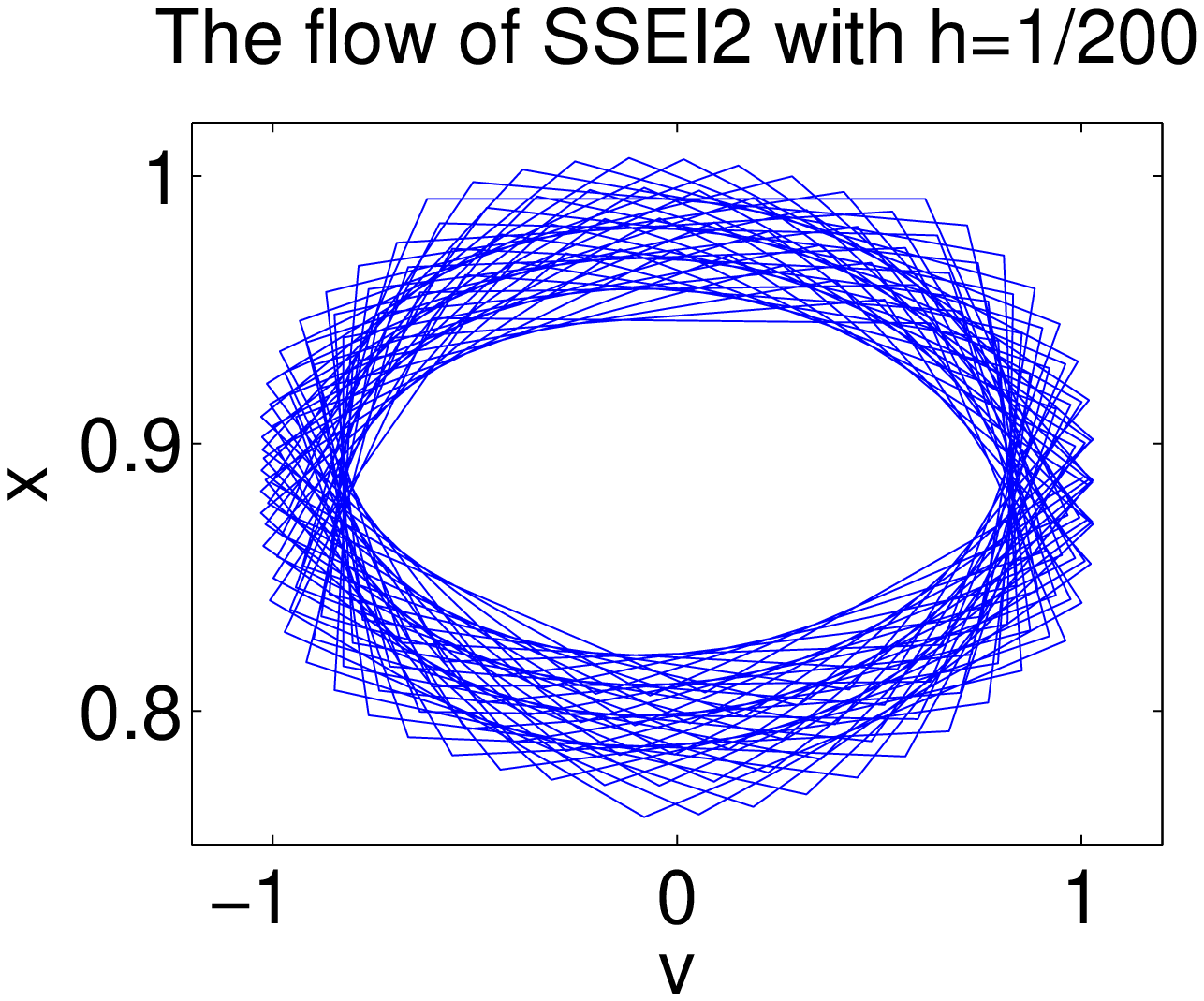}\\
\includegraphics[width=3.0cm,height=3.0cm]{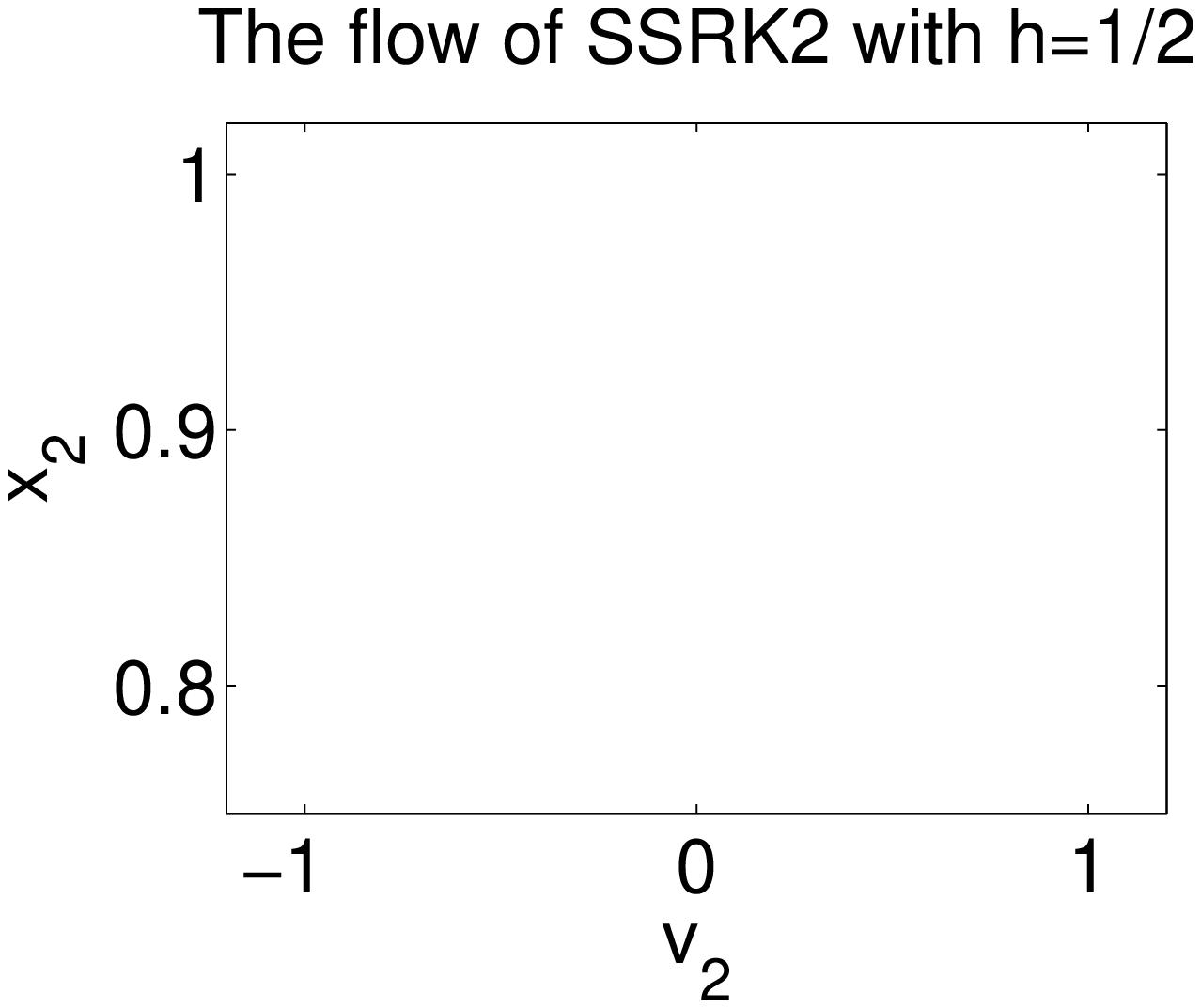}
\includegraphics[width=3.0cm,height=3.0cm]{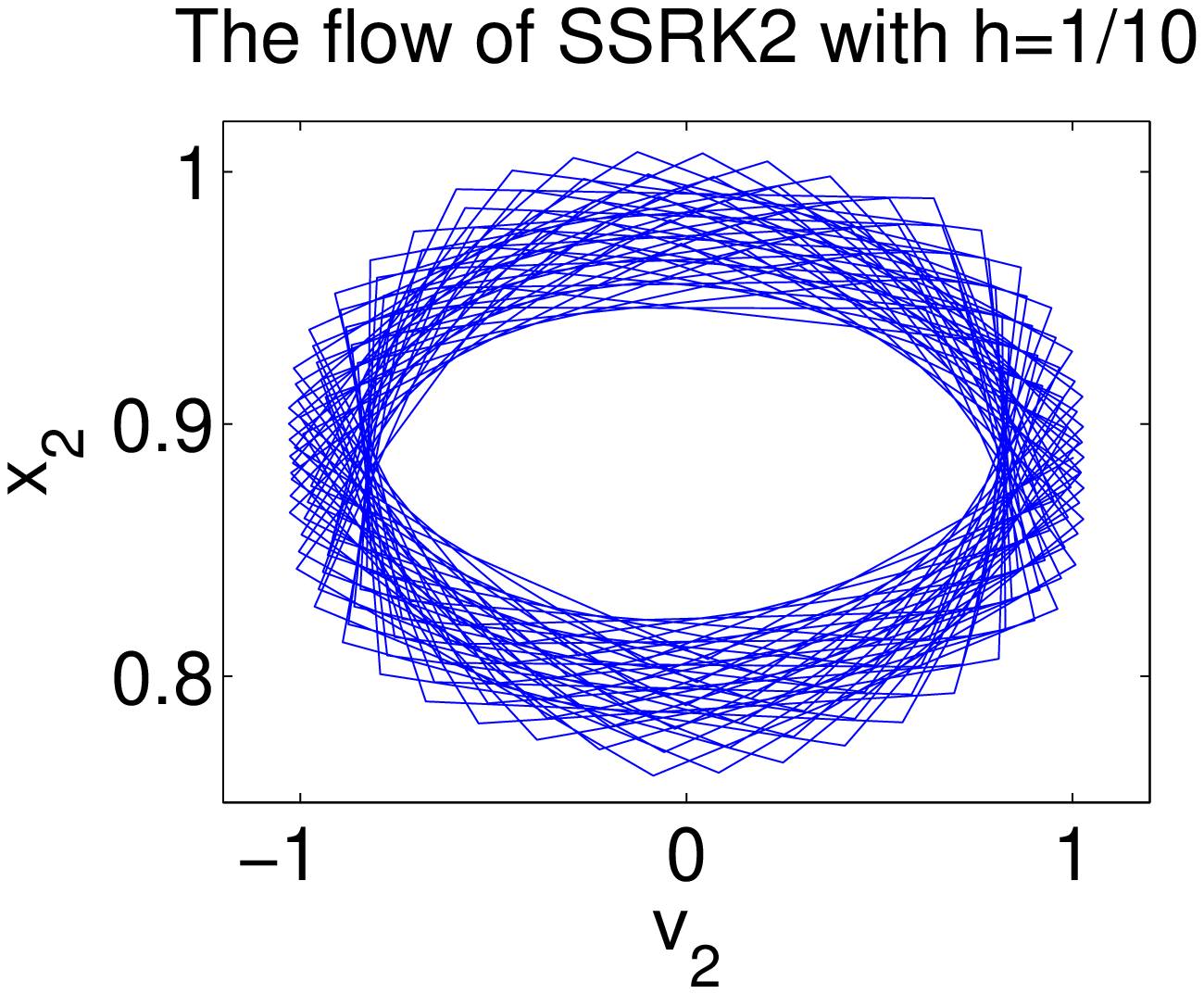}
\includegraphics[width=3.0cm,height=3.0cm]{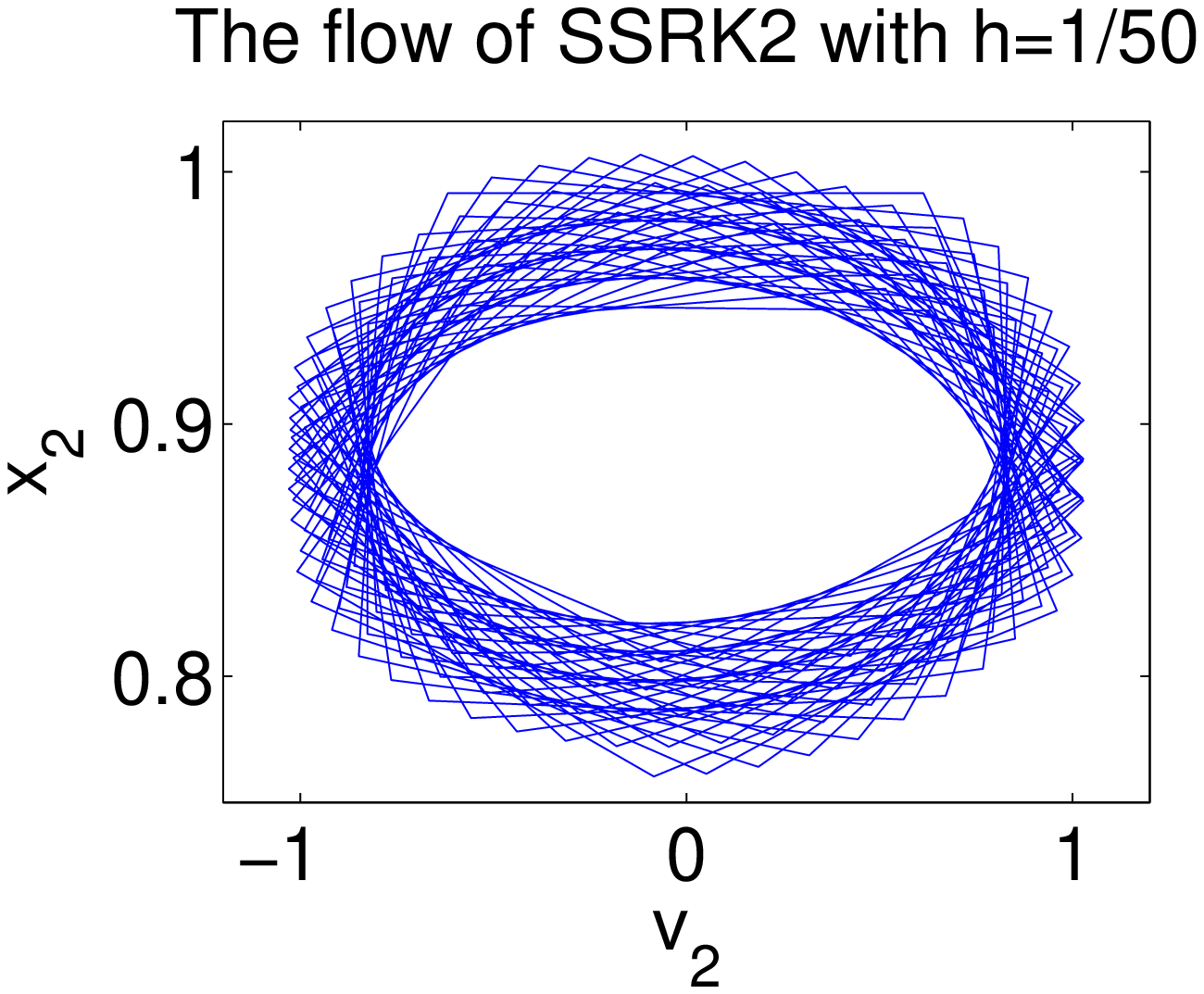}
\includegraphics[width=3.0cm,height=3.0cm]{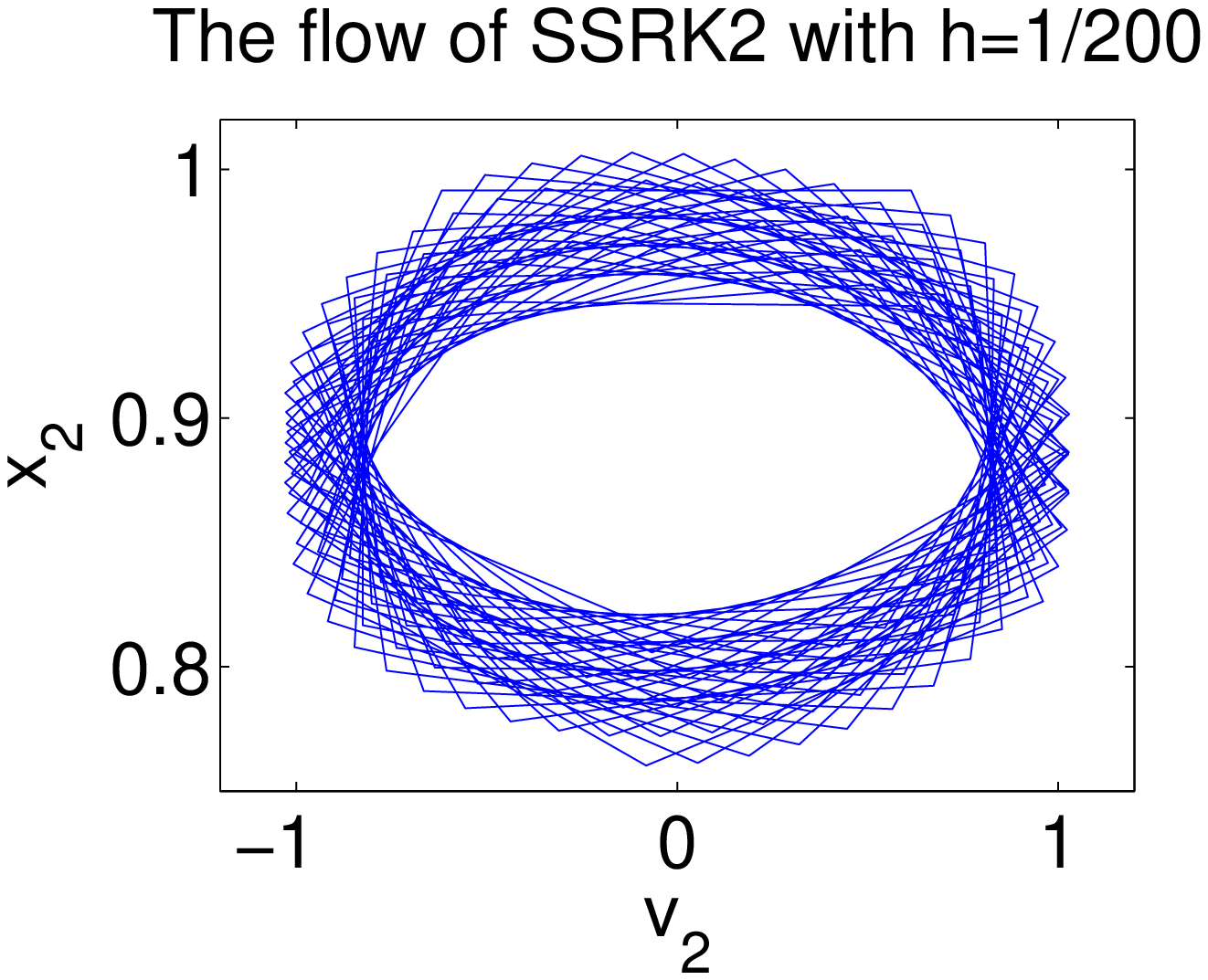}
\caption{Problem 4: the flows of different methods.} \label{p4-1}
\end{figure}

 \begin{figure}[ptb]
\centering
\includegraphics[width=3.8cm,height=4cm]{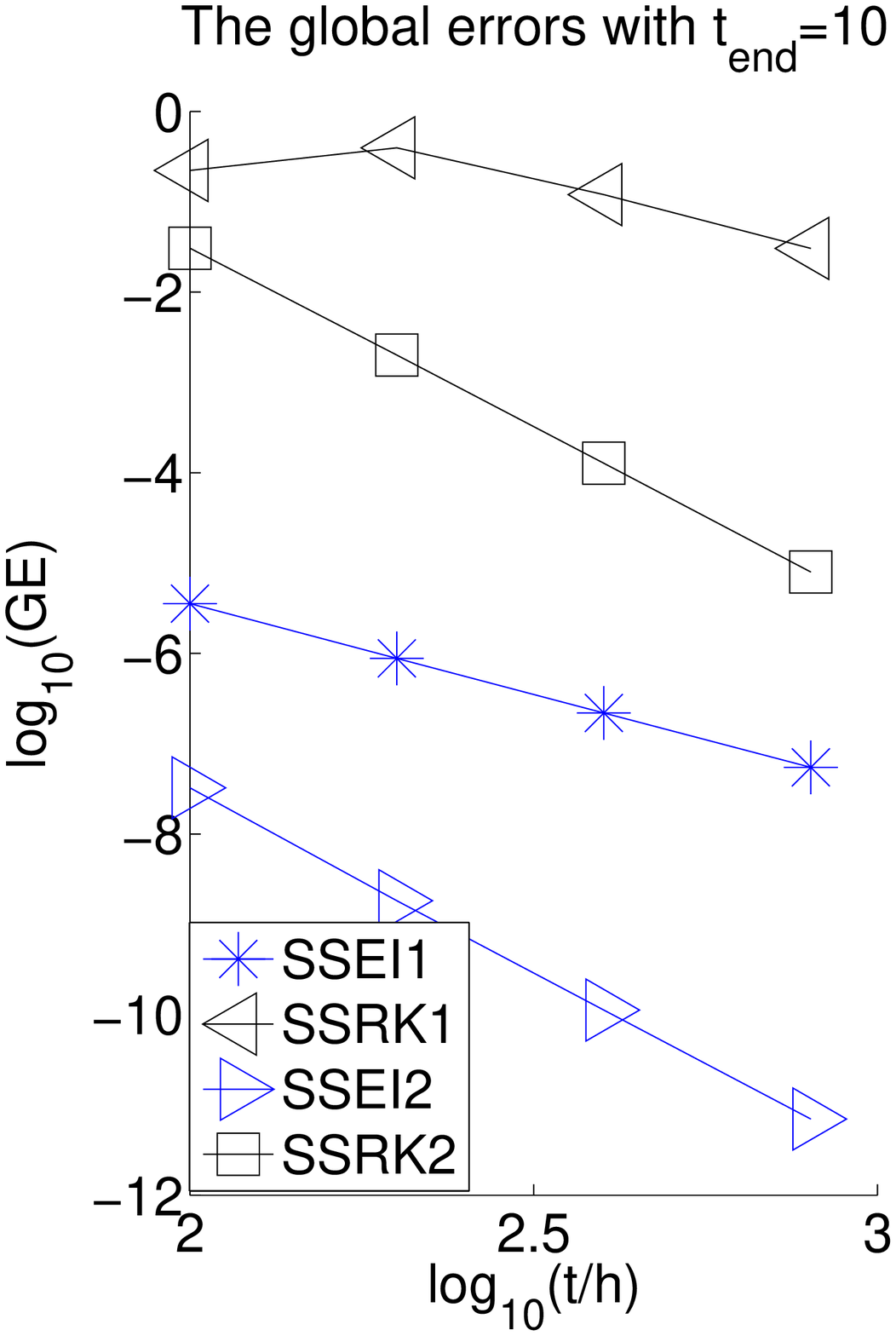}
\includegraphics[width=3.8cm,height=4cm]{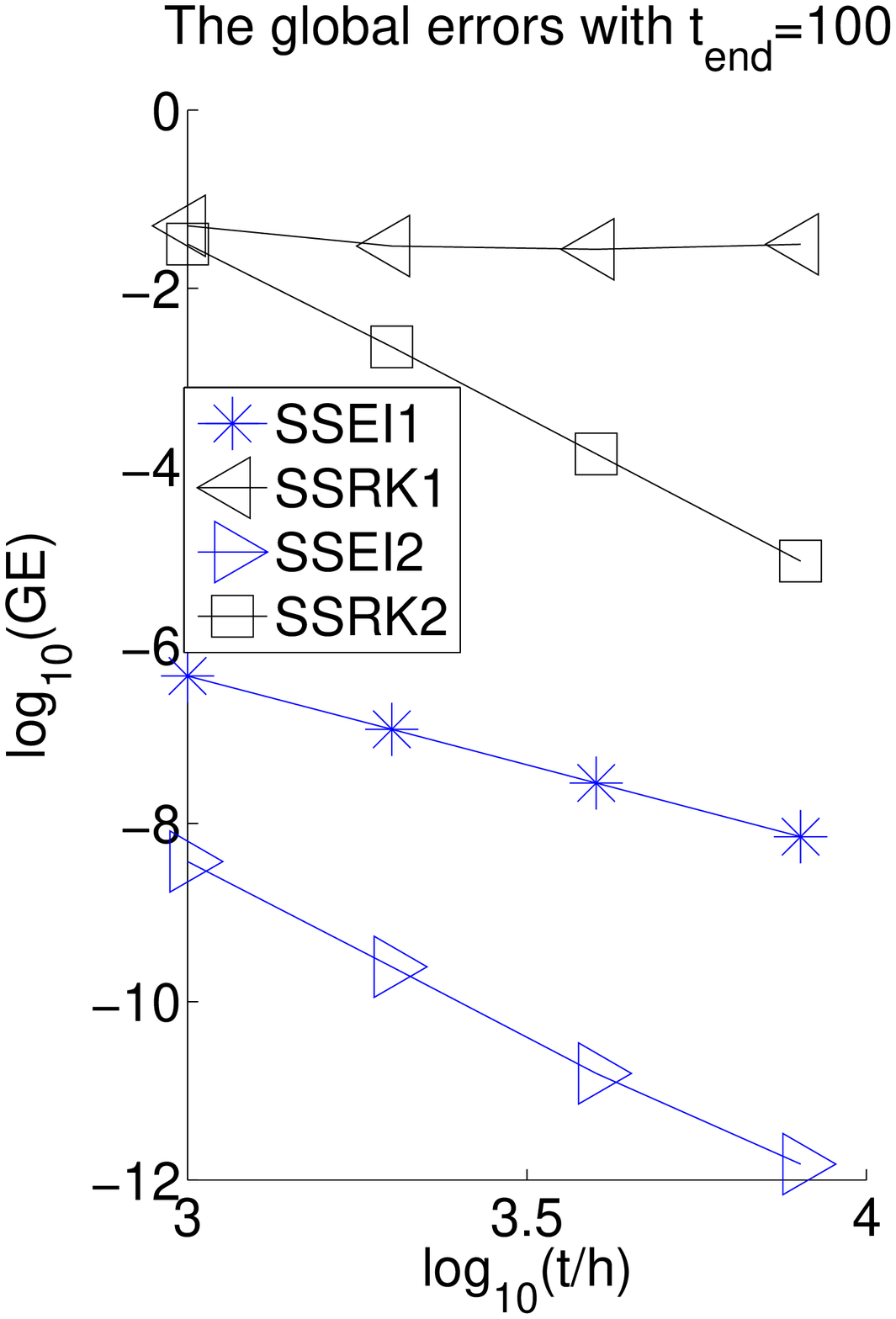}
\includegraphics[width=3.8cm,height=4cm]{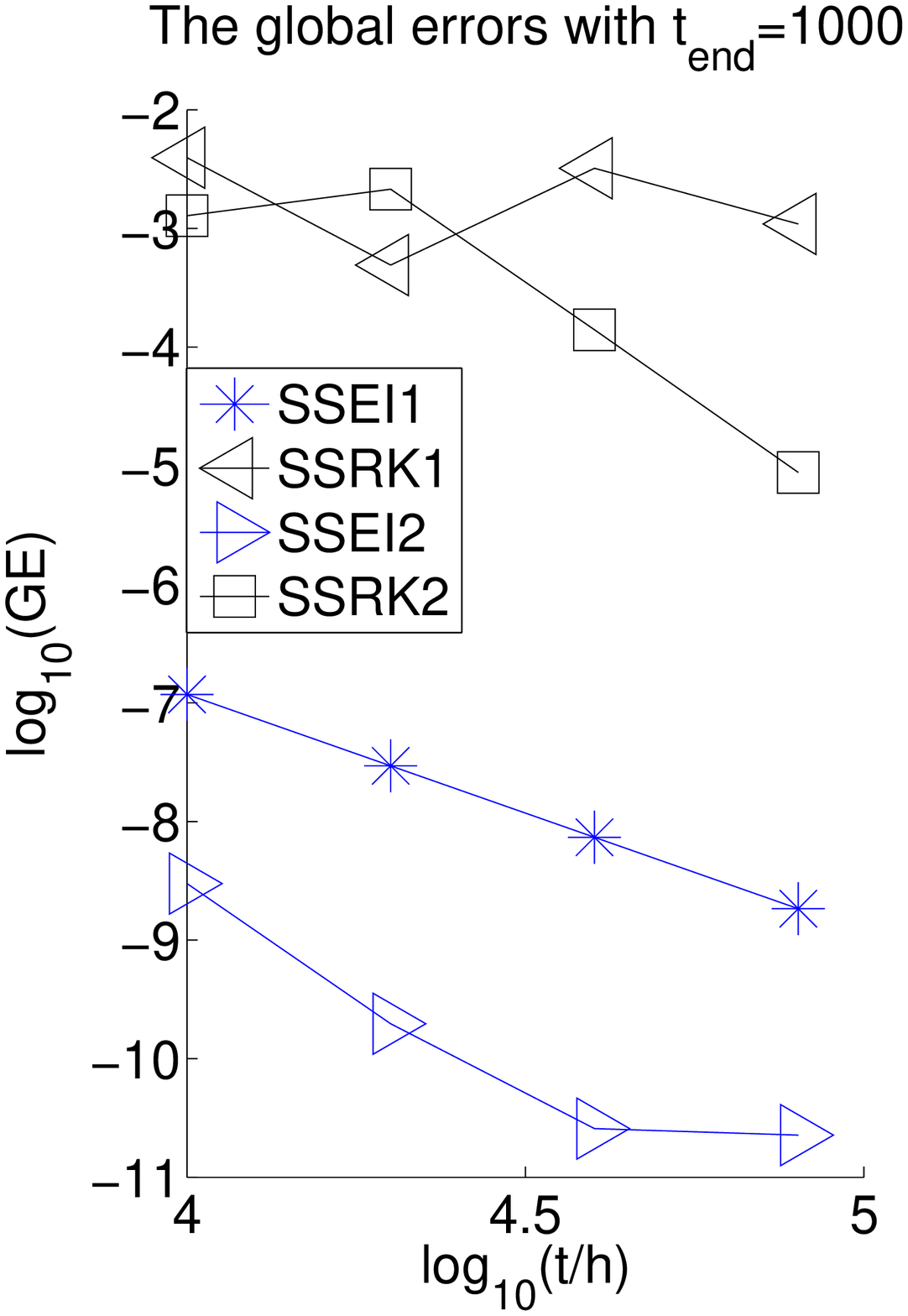}
\caption{Problem 4: the relative global errors.} \label{p4-2}
\end{figure}

\section{Conclusions} \label{sec:conclusions}
This paper studied    volume-preserving
 exponential integrators. The    necessary and
sufficient   volume-preserving condition for exponential integrators
was derived and      volume-preserving properties  were discussed
for four kinds of vector fields. It was shown that symplectic
exponential integrators can be of volume preservation for a much
larger class of vector fields than Hamiltonian systems.  It should
be noted that a new result has been proved that all
  a kind of adapted exponential integrators  methods is of volume preservation for the
highly oscillatory system \eqref{HOS-2}. Moreover, the
volume-preserving property  of ERKN/RKN methods  is discussed for
separable partitioned  systems. Some new results on Geometric
Integration were presented for second-order highly oscillatory
problems and separable partitioned  systems. We also carried out
four numerical experiments to demonstrate  the   remarkable
robustness and  superiority of
 volume-preserving  exponential integrators in comparison with volume-preserving
Runge-Kutta methods.


\end{document}